\documentclass[american,10pt]{amsart}
\pdfoutput=1
\usepackage[T1]{fontenc}
\usepackage[utf8]{inputenc}
\usepackage{lmodern}
\usepackage{enumitem}
\usepackage{mathtools}
\usepackage{amssymb}
\usepackage{bm}
\usepackage[unicode=true,pdfusetitle,
 bookmarks=true,bookmarksnumbered=false,bookmarksopen=false,
 breaklinks=false,pdfborder={0 0 1},backref=false,colorlinks=false]{hyperref}
\usepackage[foot]{amsaddr}
\usepackage{cite}

\theoremstyle{plain}
\newtheorem{thm}{Theorem}[section]
\newtheorem{lem}[thm]{Lemma}
\newtheorem{prop}[thm]{Proposition}

\newtheorem{cor}[thm]{Corollary}

\theoremstyle{definition}
\newtheorem{defi}{\textbf{Definition}}[section]

\theoremstyle{remark}
\newtheorem{rem}{Remark}[section]
\newtheorem*{rem*}{Remark}

\newcommand\R{\mathbb{R}}
\newcommand\Sn{\mathbb{S}^{n-1}}
\newcommand\N{\mathbb{N}}
\DeclareMathOperator{\vspan}{span}
\DeclareMathOperator{\diag}{diag}
\newcommand\upd{\textup{d}}
\newcommand\upT{\textup{T}}
\newcommand\upp{\textup{per}}
\newcommand\upe{\textup{e}}
\newcommand\upDir{\textup{Dir}}
\newcommand\upPF{\textup{PF}}
\newcommand\upF{\textup{F}}

\newcommand{\vect}[1]{\bm{\mathbf{#1}}} 
\newcommand\veL{\vect{L}} 
\newcommand\veu{\vect{u}} 
\newcommand\vev{\vect{v}} 
\newcommand\vew{\vect{w}} 
\newcommand\vez{\vect{0}} 
\newcommand\veo{\vect{1}} 
\newcommand\vei{\vect{\infty}} 
\newcommand\caP{\mathcal{P}} 
\newcommand\caC{\mathcal{C}} 
\newcommand\cbQ{\vect{\mathcal{Q}}} 
\newcommand\dcbP{\diag(\vect{\caP})} 
\newcommand\clOmper{\overline{\Omega_\upp}} 
\newcommand\FPH{F\"{o}ldes--Pol\'{a}\v{c}ik's Harnack } 

\usepackage{todonotes}

\title[Principal eigenvalues of space-time periodic cooperative operators]{Generalized principal eigenvalues of space-time periodic, weakly coupled, cooperative, parabolic operators}
\author{L\'{e}o Girardin}
\address[L. G.]{CNRS, ICJ UMR5208, \'{E}cole Centrale de Lyon, INSA Lyon, Universit\'{e} Claude Bernard Lyon 1, Universit\'{e} Jean Monnet,
69622 Villeurbanne, France}
\email{leo.girardin@math.cnrs.fr}

\author{Idriss Mazari-Fouquer}
\address[I. M.-F.]{CEREMADE, UMR CNRS 7534, Universit\'{e} Paris-Dauphine, Universit\'{e} PSL, Place du Mar\'{e}chal De Lattre De Tassigny, 75775 Paris cedex 16, France}
\email{mazari@ceremade.dauphine.fr}

\begin{document}
\begin{abstract}
This paper is concerned with generalizations of the notion of principal eigenvalue in the context of space-time 
periodic cooperative operators. When the spatial domain is the whole space, the Krein--Rutman theorem
cannot be applied and this leads to more sophisticated constructions and to the notion of generalized principal eigenvalues.
These are not unique in general and we focus on a one-parameter family corresponding to principal eigenfunctions that are 
space-time periodic multiplicative perturbations of exponentials of the space variable. Besides existence and uniqueness properties
of such principal eigenpairs, we also prove various dependence and optimization results illustrating how known 
results in the scalar setting can, or cannot, be extended to the vector setting. We especially prove an optimization property
on minimizers and maximizers among mutation operators valued in the set of bistochastic matrices that is, to the best of our knowledge, new.
\end{abstract}

\maketitle
\tableofcontents{}

\section{Introduction}
In recent years, the study of principal eigenvalues has proved very fruitful, especially (but not exclusively) for the 
study of several biological phenomena. Indeed, these eigenvalues encode several informations that are crucial in the understanding of 
population dynamics. Although the scalar case is 
now rather well understood, several problems remain open in the case of systems. In this paper, we propose a systematic approach for the 
case of parabolic linear operators with space-time periodic coefficients that satisfy a sufficiently strong form of maximum principle, and we offer several contributions to their spectral analysis and optimization. 

The remainder of Section 1 is devoted to a detailed introduction (scope, motivations, notations, definitions, main results and
applications to semilinear systems). Section 2 is devoted to technical preliminaries. Section 3 contains the proofs.

\subsection{Scope of the paper}
The goal of this first subsection is to present in a succinct fashion the mathematical objects at hands and their interest for applications, with a special emphasis on population dynamics. 
More details on possible applications can be found in Subsection \ref{sec:motivation}.

Formally, this paper is concerned with eigenvalues of linear operators of the form 
\[
    \cbQ:\veu\mapsto\dcbP\veu -\veL\veu,
\]
where: $\veu:\R\times\R^n\to\R^N$ is a vector-valued function of size $N\in\N^*$, with a time variable $t\in\R$ and a
space variable $x\in\R^n$, $n\in \N^*$ being the spatial dimension; 
each operator of the family $\vect{\caP}=(\caP_i)_{i\in[N]}$, where $[N]=\N\cap[1,N]$, has the form
\[
    \caP_i:u\mapsto\partial_t u -\nabla\cdot\left(A_i\nabla u\right)+q_i\cdot\nabla u,
\]
with $A_i:\R\times\R^n\to\R^{n\times n}$ and $q_i:\R\times\R^n\to\R^n$ periodic functions of $t$ and $x$, respectively square matrix-valued and
vector-valued;
$\veL:\R\times\R^n\to\R^{N\times N}$ is a square matrix-valued periodic function of $t$ and $x$. 

The standing assumptions on $\vect{\caP}$ and $\veL$ are the following.
\begin{enumerate}[label=$({\mathsf{A}}_{\arabic*})$]
    \item \label{ass:ellipticity} The family $(A_i)_{i\in[N]}$ is \textit{uniformly elliptic}:
    \[
        0<\min_{i\in[N]}\min_{y\in\Sn}\min_{(t,x)\in\R\times\R^n}\left(y\cdot A_i(t,x)y\right).
    \]
    \item \label{ass:cooperative} The matrix $\underline{\veL}\in\R^{N\times N}$, whose entries are 
    \[
        \underline{l}_{i,j}=\min_{(t,x)\in\R\times\R^n}l_{i,j}(t,x)\quad\text{for all }(i,j)\in[N]^2,
    \]
    is \textit{essentially nonnegative}: its off-diagonal entries are nonnegative.
    \item \label{ass:irreducible} The matrix $\overline{\veL}\in\R^{N\times N}$, whose entries are 
    \[
        \overline{l}_{i,j}=\max_{(t,x)\in\R\times\R^n}l_{i,j}(t,x)\quad\text{for all }(i,j)\in[N]^2,
    \]
    is \textit{irreducible}: it does not have a stable subspace of the form
    $\vspan(\vect{e}_{i_1},\dots,\vect{e}_{i_k})$, where $k\in[N-1]$, $i_1,\dots,i_k\in[N]$ and $\vect{e}_i=(\delta_{ij})_{j\in[N]}$.
    By convention, $[0]=\emptyset$ and $1\times 1$ matrices are irreducible, even if zero.
    \item \label{ass:smooth_periodic} The coefficients $\veL$, $(A_i)_{i\in [N]}$, $(q_i)_{i\in [N]}$ are H\"{o}lder-continuous 
    and periodic in their variables: there exists $\delta\in(0,1)$ such that 
    $\veL\in\caC^{\delta/2,\delta}_\upp(\R\times\R^n,\R^{N\times N})$ and, for any $i\in[N]$,
    $A_i\in\caC^{\delta/2,1+\delta}_\upp(\R\times\R^n,\R^{n\times n})$ and
    $q_i\in\caC^{\delta/2,\delta}_\upp(\R\times\R^n,\R^n)$. Moreover, $A_i=A_i^\upT$ for each $i\in[N]$.
\end{enumerate}

The precise definition of the functional spaces appearing in \ref{ass:smooth_periodic} will be clarified in Section \ref{Se:Notations} below.
As usual in such a smooth and generic framework, the symmetry of the diffusion matrices can be assumed without loss of generality\footnote{Indeed, if $A_i$ is not symmetric, then we can write it as the sum of its symmetric part 
$A_i^{\textup{sym}}=\frac12(A_i+A_i^\upT)$ 
and its skew-symmetric part $A_i^{\textup{skew}}=\frac12(A_i-A_i^\upT)$. The operator $\nabla\cdot(A_i^{\textup{skew}}\nabla)$ acting 
on the space of functions of class $\caC^2$ can be rewritten as an advection operator $a_i\cdot\nabla$, so that
$-\nabla\cdot(A_i\nabla)+q_i\cdot\nabla=-\nabla\cdot(A_i^{\textup{sym}}\nabla)+(q_i-a_i)\cdot\nabla$.  
The operator on the right-hand side has the same structure and has ``gained'' the symmetry of its diffusion matrix.}.
No symmetry assumption is made on $\veL$ and the irreducibility of $\overline{\veL}$ in \ref{ass:irreducible} is equivalent
to the irreducibility of the space-time average of $\veL$.

The linear partial differential operator with space-time periodic coefficients $\cbQ$ is 
\textit{weakly coupled} (namely, coupled only in the zeroth order term \cite{Protter_Weinberger}) and, by virtue of
\ref{ass:ellipticity}, \ref{ass:cooperative} and \ref{ass:irreducible} respectively, it is \textit{uniformly parabolic},
\textit{cooperative} (namely, satisfying the so-called Kamke condition \cite{Cantrell_Cosner_03})
and \textit{fully coupled} (namely, coupled in such a way that the system $\cbQ\veu=\vev$ contains 
no independent subsystem \cite{Sweers_1992,Bai_He_2020}).
A prototypical example of coupling matrix is: 
\begin{equation*}
    \veL=\begin{pmatrix} 0 & 1 \\ 1 & 0 \end{pmatrix}.
\end{equation*}
The dynamics associated to this coupling matrix are indeed cooperative: the presence of $u_2$ is favorable to $u_1$, and conversely.
More generally, \ref{ass:cooperative} and \ref{ass:irreducible} together imply that if, initially, all components are nonnegative
and one is positive, then this component will help out the growth of the others, and all will ultimately be positive. 
This loose statement is a form of maximum principle, that will be stated rigorously in Subsection \ref{sec:maximum_principle}
and that will be crucial in the analysis.  

Such operators, and their eigenvalues, have natural interpretations, in particular when considered from the point of view of 
population dynamics. For a wide class of reaction--diffusion models,
the long-time behavior of a population $\veu$, and in particular its ability to thrive in a given environment, is at least partially 
governed by the sign of the principal eigenvalue of the operator obtained by linearizing the model around the steady state 
$\veu=\vez$ (\textit{cf.} Subsection \ref{sec:motivation}).
In this paper, we set out to establish rigorously a number of analytical results on (generalized) principal eigenvalues, 
such as existence, characterization, asymptotic behaviors, 
to provide counter-examples to properties that are known to hold in the case of scalar equations,
and then to investigate related optimization problems. 
Indeed, interpreting the eigenvalue as a survival criterion has triggered a wide interest in spectral optimization: 
it makes sense to try and design the environment, or the interaction between individuals, in a way that optimizes the 
eigenvalue to ensure the survival or, conversely, extinction of the species. This point of view was adopted for instance in 
\cite{Berestycki_Ham_1} in the case of scalar equations. From the interpretation perspective, the goal of 
\cite{Berestycki_Ham_1} was to optimize the resources distribution. In the present paper, we consider similar questions 
(\textit{i.e.} how to distribute resources in a domain) in the case of systems and we also investigate the question of optimal 
interaction between individuals. The latter class of results refers for instance to the optimization of mutation strategies and 
proves to be mathematically more challenging.

\subsection{Motivations}\label{sec:motivation}
The linear parabolic system $\cbQ\veu =\vez$ can be understood as the linearization at the homogeneous steady state 
$\vez$ of a semilinear reaction--diffusion system $\dcbP\veu(t,x) = \vect{f}(t,x,\veu(t,x))$. In this interpretation, $\veL(t,x)$
denotes the Jacobian matrix $\textup{D}_{\veu}\vect{f}(t,x,\vez)$.

As we have already alluded to, from a modeling viewpoint, such systems appear for instance in population dynamics, in models
using growth terms such that, if $\veu(t,x)$ is in the positive cone of $\R^N$, then so is 
\begin{equation*}
    \veL(t,x)\veu(t,x)-\vect{f}(t,x,\veu(t,x))=\textup{D}_{\veu}\vect{f}(t,x,\vez)\veu(t,x)-\vect{f}(t,x,\veu(t,x)). 
\end{equation*}
In recent years, these growth terms have been referred to as Fisher--KPP, or simply KPP, reaction terms, since they generalize
 the standard scalar Fisher--KPP reaction term \cite{KPP_1937,Fisher_1937}.
In general they are not cooperative and in particular they do not satisfy the comparison principle.
Non-cooperative Fisher--KPP systems whose linearization around $\vez$ is nonetheless cooperative have been the object of 
a growing literature in the past few years, especially in the case of two components $N=2$ (see, \textit{e.g.}, 
\cite{Girardin_2016_2,Girardin_2016_2_add,Girardin_2018,Girardin_2017,Girardin_Griette_2020,Griette_Raoul,Alfaro_Griette,Cantrell_Cosner_Yu_2018,Cantrell_Cosner_Yu_2019,Morris_Borger_Crooks,Brown_Zhang_03,Hei_Wu_2005,Bouguima_Feikh_Hennaoui_2008,Henaoui_2012,Ma_Chen_Lu_2014,Griette_Matano_2021}).
They arise as models for populations structured in age classes or phenotypical trait classes
\cite{Girardin_2016_2,Cantrell_Cosner_Yu_2018,Cantrell_Cosner_Martinez_2020,Elliott_Cornel,Griette_Raoul_}.
In this context, the sign of the principal eigenvalue of the linearization at $\vez$ indicates, at least in
simple spatio-temporal settings, whether small populations survive and persist or, on the contrary, go extinct. 
It turns out that, for such models, population persistence is generically equivalent to small population persistence,
and this makes the study of the principal eigenvalue even more crucial.

When the underlying model is a population structured with respect to a phenotypical trait, then $\veL$ typically takes the form
$\veL = \diag(r_i) + \vect{M}$, where each $r_i>0$ is an intrinsic growth rate and the matrix $\vect{M}$ is a mutation matrix; 
in the simplest case $\vect{M}$ is a discrete Laplacian with Neumann boundary conditions:
\begin{equation}\label{eq:discrete_Lap}
    \vect{M}= \mu
    \begin{pmatrix}
    -1 & 1 & 0 & \dots & 0\\
    1 & -2 & \ddots & \ddots & \vdots\\
    0 & \ddots & \ddots & \ddots & 0\\
    \vdots & \ddots & \ddots & -2 & 1\\
    0 & \dots & 0 & 1 & -1
    \end{pmatrix}
\end{equation}
where $\mu>0$ is a mutation rate.

When the underlying model is a population structured with respect to age, then $\veL$ is a diagonally perturbed Leslie matrix:
\begin{equation}\label{eq:Leslie}
    \vect{L}=-\diag(d_i+a_i)+
    \begin{pmatrix}
    b_1 & b_2 & b_3 & \dots & b_{N}\\
    a_1 & 0 & 0 & \dots & 0\\
    0 & a_2 & 0 & \ddots & \vdots\\
    \vdots & \ddots & \ddots & \ddots & \vdots\\
    0 & \dots & 0 & a_{N-1} & 0
    \end{pmatrix}
\end{equation}
where each $d_i\geq 0$ is a death rate, each $a_i>0$ an aging rate and each $b_i\geq 0$ a birth rate with $b_N>0$.

Each one of these models can be understood as a discretized version of some nonlocal equation \cite{Girardin_2016_2}.

The second example \eqref{eq:Leslie} above explains in particular why we do not make any \textit{a priori} assumption on the symmetry of $\veL$.

Let us also point out that the periodic cooperative operators we consider find applications in the chemistry of nuclear reactor cores
\cite{Allaire_Hutridurga_2015,Capdeboscq_2002}. Due to our long-term goals (see Subsection \ref{sec:relation_with_semilinear}), 
in this paper, we favor a population dynamics interpretation.

\subsection{Notations}\label{Se:Notations}

In the whole paper, $\N$ is the set of nonnegative integers, which contains $0$.

We fix once and for all $n+1$ positive numbers $T, L_1, \dots, L_n\in \R_+^*$. For the sake of brevity, we use the notations
$L=(L_1,\dots,L_n)$, $(0,L)=(0,L_1)\times\dots\times(0,L_n)$ and $|[0,L]|=\prod_{\alpha=1}^n L_\alpha$. Unless otherwise specified, 
temporal and spatial periodicities refer to, respectively, $T$-periodicity with respect to $t$ and $L_\alpha$-periodicity with respect to 
$x_\alpha$ for each $\alpha\in[n]$ (or $L$-periodicity with respect to $x$ for short). The space-time periodicity cell $(0,T)\times(0,L)$ 
is denoted $\Omega_\upp$ and its volume is $T|[0,L]|$.

Vectors in $\R^N$ and matrices in $\R^{N\times N}$ are denoted in bold font. 
Functional operators are denoted in calligraphic typeface (bold if they act on functions valued in $\R^N$).
Functional spaces, \textit{e.g.} $\mathcal{W}^{1,\infty}(\R\times\R^n,\R^N)$, 
are also denoted in calligraphic typeface. A functional space $\mathcal{X}$ denoted with a subscript $\mathcal{X}_\upp$, 
$\mathcal{X}_{t-\upp}$ or $\mathcal{X}_{x-\upp}$ is restricted to functions that are space-time periodic, time periodic or space periodic respectively.

For clarity, H\"{o}lder spaces of functions with $k\in\mathbb{N}\cup\{0\}$ derivatives that are all
H\"{o}lder-continuous with exponent $\alpha\in(0,1)$ are denoted $\caC^{k+\alpha}$; when the domain is $\R\times\R^n$, it 
should be unambiguously understood that $\caC^{k+\alpha,k'+\alpha'}$ denotes the set of functions that have $k$ 
$\alpha$-H\"{o}lder-continuous derivatives in time and $k'$ $\alpha'$-H\"{o}lder-continuous derivatives in space.

For any two vectors $\veu,\vev\in\R^N$, $\veu\leq\vev$ means $u_i\leq v_i$ for all 
$i\in[N]$, $\veu<\vev$ means $\veu\leq\vev$ together with $\veu\neq\vev$ and $\veu\ll\vev$ means $u_i<v_i$ for all $i\in[N]$. If 
$\veu\geq\vez$, we refer to $\veu$ as \textit{nonnegative}; if $\veu>\vez$, as \textit{nonnegative nonzero}; if $\veu\gg\vez$, as 
\textit{positive}. The sets of all nonnegative, nonnegative nonzero, positive vectors are respectively denoted $[\vez,\vei)$,
$[\vez,\vei)\backslash\{\vez\}$ and $(\vez,\vei)$. The vector whose entries are all equal to $1$ is denoted by $\veo$ and this never refers to
an indicator function.
Similar notations and terminologies might be used in other dimensions and for matrices. The identity matrix is denoted $\vect{I}$.

Similarly, a function can be nonnegative, nonnegative nonzero, positive. For clarity, a positive function is a function with only positive values.

To avoid confusion between operations in the state space $\R^N$ and operations in the spatial domain $\R^n$, 
Latin indexes $i,j,k$ are assigned to vectors and matrices of size $N$ whereas Greek indexes $\alpha,\beta,\gamma$ 
are assigned to vectors and matrices of size $n$. 
We use mostly subscripts to avoid confusion with algebraic powers, but when both Latin and Greek indexes are involved, we 
move the Latin ones to a superscript position, \textit{e.g.} $A^i_{\alpha,\beta}(t,x)$.
We denote scalar products in $\R^N$ with the transpose operator, $\veu^\upT\vev=\sum_{i=1}^N u_i v_i$,
and scalar products in $\R^n$ with a dot, $x\cdot y =\sum_{\alpha=1}^n x_\alpha y_\alpha$. 

For any vector $\veu\in\R^N$, $\diag(\veu)$, $\diag(u_i)_{i\in[N]}$ or $\diag(u_i)$ for short refer to the diagonal matrix in $\R^{N\times N}$
whose $i$-th diagonal entry is $u_i$. These notations can also be used if $\veu$ is a function valued in $\R^N$.

Finite dimensional Euclidean norms are denoted $|\cdot |$ whereas the notation $\|\cdot \|$ is reserved for norms in functional spaces.

The notation $\circ$ is reserved in the paper for the Hadamard product (component-wise product of vectors or matrices)
and never refers to the composition of functions.

Finally, when the focus of the paper is on the dependence of an eigenvalue on (a parameter of) the underlying
operator, and when the context is unambiguous, we write with a slight abuse of notation the eigenvalue as 
a function of the varying parameter (\textit{e.g.}, an eigenvalue $\lambda$ of the operator $\cbQ$ might be 
denoted $\lambda(\cbQ)$, $\lambda(A_1)$, $\lambda(q_1,\dots,q_n)$, $\lambda(\veL)$, and so on).

\subsection{Generalized principal eigenvalues in space-time periodic media}

In \cite{Nadin_2007}, Nadin analyzed the scalar case $N=1$. Following previous efforts
\cite{Berestycki_Ros_1,Berestycki_Nir,Berestycki_Ham_1,Hess_1991},
he introduced and studied the following quantities:
\begin{equation*}
    \lambda_1 = \sup\left\{ \lambda\in\R\ |\ \exists u\in\caC^{1,2}_{t-\upp}(\R\times\R^n,(0,\infty))\ \mathcal{Q}u\geq \lambda u \right\},
\end{equation*}
\begin{equation*}
    \lambda_1' = \inf\left\{ \lambda\in\R\ |\ \exists u\in\mathcal{W}^{1,\infty}\cap\caC^{1,2}_{t-\upp}(\R\times\R^n,(0,\infty))\ \mathcal{Q}u\leq \lambda u \right\}.
\end{equation*}
These two quantities turn out to be eigenvalues of $\mathcal{Q}$ (in the sense that associated eigenfunctions exist), and are referred to as 
\textit{generalized principal eigenvalues} (their eigenfunctions are referred to as \textit{generalized principal eigenfunctions}). 
Due to the lack of compactness in the spatial variable, the existence of these eigenvalues cannot be directly deduced from the 
Krein--Rutman theorem \footnote{The Krein--Rutman theorem, which deals with the existence and simplicity of the principal eigenvalue of positivity preserving operators, is of crucial importance in the study of reaction--diffusion equations. We refer, for a statement and applications to reaction--diffusion equations, to Cantrell--Cosner \cite[Theorem 2.12]{Cantrell_Cosner_03} or Lam--Lou \cite[Appendix B]{Lam_Lou_2022}.}. 
However, they can be related with classical Krein--Rutman principal eigenvalues: 
the first one, $\lambda_1$, is the limit of the principal eigenvalues associated with the time periodic problem with 
Dirichlet boundary conditions in a sequence of growing balls; the second one, $\lambda_1'$, coincides with the principal eigenvalue of 
the space-time periodic problem. Actually, both eigenvalues are related to the family $\left(\lambda_{1,z}\right)_{z\in\R^n}$ of 
principal eigenvalues of the space-time periodic problems associated with the operators
\begin{equation*}
    \mathcal{Q}_z:u\mapsto \upe_{-z}\mathcal{Q}\left(\upe_ z u\right)\quad\text{where }\upe_{\pm z}:x\mapsto\upe^{\pm z\cdot x},
\end{equation*}
which can be expanded as
\begin{equation*}
    \mathcal{Q}_z u=\mathcal{Q}u -(A+A^\upT)z\cdot\nabla u -(z\cdot Az + \nabla\cdot(Az)-q\cdot z)u.
\end{equation*}
Since $\mathcal{Q}(\upe_{z}u)=\lambda_{1,z}\upe_{z}u$, $\lambda_{1,z}$ can be understood as the principal eigenvalue of $\mathcal{Q}$ acting on
the set $\upe_{z}.\caC^{1,2}_\upp$ of space-time periodic multiplicative perturbations of the planar exponential $\upe_{z}$. 
Nadin showed that $\lambda_1'=\lambda_{1,0}\leq\lambda_1=\max_{z\in\R^n}\lambda_{1,z}$ and subsequently exhibited sufficient conditions for 
the equality $\lambda_1=\lambda_1'$ to hold; his study is completed by several dependence and optimization results.

Our aim in this paper is twofold. First, we want to generalize the results of Nadin; second, we want to illustrate the originality of 
systems compared to scalar equations by means of new results and counter-examples without scalar counterpart. Let us point out that most generalizations of scalar
results we consider here require work indeed. On one hand, many proofs of \cite{Nadin_2007} rely on algebraic operations that are at least 
ambiguous, at worst unavailable, in the vector setting, like powers or quotients, and this often leads to counter-examples. 
On the other hand, the full coupling assumption that we use to emulate the scalar strong comparison principle, \ref{ass:irreducible}, 
is not a pointwise property but rather a global property, and this makes some adaptations quite technical.

Replacing scalar operators and test functions ($N=1$) by vector ones ($N\in\N^\star$),
we will therefore study the following quantities:
\begin{equation}
    \label{def:lambda1}
    \lambda_1 = \sup\left\{ \lambda\in\R\ |\ \exists \veu\in\caC^{1,2}_{t-\upp}(\R\times\R^n,(\vez,\vei))\ \cbQ\veu\geq\lambda\veu \right\},
\end{equation}
\begin{equation}
    \label{def:lambda1prime}
    \lambda_1' = \inf\left\{ \lambda\in\R\ |\ \exists \veu\in\mathcal{W}^{1,\infty}\cap\caC^{1,2}_{t-\upp}(\R\times\R^n,(\vez,\vei))\ \cbQ\veu\leq\lambda\veu \right\},
\end{equation}
as well as the family $\left(\lambda_{1,z}\right)_{z\in\R^n}$, where:
\begin{equation}
    \label{def:lambdaz}
    \lambda_{1,z} = \lambda_{1,\upp}\left( \cbQ_z \right),
\end{equation}
\begin{equation}
    \label{def:Qz}
    \cbQ_z = \upe_{-z}\cbQ\upe_z = \cbQ-\diag\left((A_i+A_i^\upT)z\cdot\nabla+z\cdot A_i z+\nabla\cdot\left(A_i z\right)-q_i\cdot z\right).
\end{equation}

As in the scalar case, it is a standard result that the Krein--Rutman theorem can be successfully applied to the 
weakly coupled, fully coupled, cooperative operator $\cbQ_z$ in the following two ways: the 
\textit{periodic principal eigenvalue} $\lambda_{1,z}=\lambda_{1,\upp}(\cbQ_z)$ is well-defined;  for any nonempty smooth bounded connected open set $\Omega\subset\R^n$, the 
\textit{Dirichlet principal eigenvalue} $\lambda_{1,\upDir}(\cbQ_z,\Omega)$ is 
well-defined.
The first one corresponds to the operator $\cbQ_z$ acting on $\caC^{1,2}_\upp(\R\times\R^n)$, 
and hereafter we denote $\veu_z$ such a positive principal eigenfunction. 
The second one corresponds to the operator acting on $\caC^{1,2}_{t-\upp}(\R\times\Omega)\cap\caC^1_0(\R\times\overline{\Omega})$,
where the subscript $0$ denotes functions that vanish on $\partial\Omega$. Eigenfunctions for these principal eigenvalues 
are unique up to multiplication by a constant. For detailed applications of the Krein--Rutman theory in the Dirichlet case,
we refer to Bai--He \cite{Bai_He_2020} or Ant\'{o}n--L\'{o}pez-G\'{o}mez \cite{Anton_Lopez-Gomez_1996}.

In contrast, and again as in the scalar case, the generalized principal eigenproblems 
for $\lambda_1$ and $\lambda_1'$ (namely, the question of knowing whether or not eigenpairs exist) 
are mathematically challenging.

\begin{defi}
A \textit{generalized principal eigenfunction associated with $\lambda_1$} is a function 
$\veu\in\caC^{1,2}_{t-\upp}(\R\times\R^n,(\vez,\vei))$ such that $\cbQ\veu=\lambda_1\veu$.

A \textit{generalized principal eigenfunction associated with $\lambda_1'$} is a function
$\veu\in\mathcal{W}^{1,\infty}\cap\caC^{1,2}_{t-\upp}(\R\times\R^n,(\vez,\vei))$ such that $\cbQ\veu=\lambda_1'\veu$.
\end{defi}

\subsection{Results}

Although the theorems and definitions in Subsection \ref{sec:theorems_existence_characterization} are completely analogous to the 
scalar setting \cite{Nadin_2007}, the ones in Subsections \ref{sec:theorems_monotonic_convex_dependence}--\ref{sec:theorems_optimization}, 
will require new restrictions specific to the parabolic vector setting and will show how the time structure, the spatial structure
and the multidimensional state space interact intricately.

\subsubsection{Existence and characterization of generalized principal eigenpairs}\label{sec:theorems_existence_characterization}

\begin{thm}\label{thm:existence_characterization_Rn}
The generalized principal eigenvalues $\lambda_1$ and $\lambda_1'$ are well-defined real numbers
related to the family $\left(\lambda_{1,z}\right)_{z\in\R^n}$:
\[
    \lambda_1'=\lambda_{1,0},\quad\lambda_1=\max_{z\in\R^n}\lambda_{1,z}.
\]
The maximum is uniquely achieved.

Consequently, $\lambda_1'\leq \lambda_1$, $\veu_0$ is a generalized principal eigenfunction associated with $\lambda_1'$ and there exists 
a unique $z^\star\in\R^n$ such that $\upe_{z^\star}\veu_{z^\star}$ is a generalized principal eigenfunction associated with $\lambda_1$.

Furthermore, the following max--min and min--max characterizations hold:
\[
    \lambda_{1,z}=\max_{\veu\in\caC^{1,2}_\upp(\R\times\R^n,(\vez,\vei))}\min_{i\in[N]}\min_{\clOmper}\left(\frac{(\cbQ_z\veu)_i}{u_i}\right)\quad\text{for all }z\in\R^n,
\]
\[
    \lambda_{1,z}=\min_{\veu\in\caC^{1,2}_\upp(\R\times\R^n,(\vez,\vei))}\max_{i\in[N]}\max_{\clOmper}\left(\frac{(\cbQ_z\veu)_i}{u_i}\right)\quad\text{for all }z\in\R^n,
\]
\[
	\lambda_1=\max_{\veu\in\caC^{1,2}_{t-\upp}(\R\times\R^n,(\vez,\vei))}\min_{i\in[N]}\inf_{\R\times\R^n}\left(\frac{(\cbQ\veu)_i}{u_i}\right).
\]
\end{thm}

By simplicity of the periodic principal eigenvalue,  the only non-negative periodic eigenfunctions are periodic 
principal eigenfunctions. Under assumptions \ref{ass:ellipticity}--\ref{ass:smooth_periodic}, if we further impose 
standard normalisation conditions on the eigenfunction (\textit{e.g.}, $|\veu_z(0,0)|=1$) , classical compactness estimates on 
the family $(\lambda_{1,z},\veu_z)$ imply that the spectral elements $(\lambda_{1,z},\veu_z)$  are continuous with
respect to the coefficients of $\cbQ_z$. In particular, this shows the continuity of $\lambda_1$ and $\lambda_1'$ as 
functions of the coefficients of $\cbQ$.

Since generalized principal eigenfunctions associated with $\lambda_1'$ are globally bounded, a simple comparison argument with the 
uniformly positive $\veu_0$ shows that it is, up to a multiplicative constant, the unique generalized principal 
eigenfunction associated with $\lambda_1'$. On the contrary, generalized principal eigenfunctions associated with
$\lambda_1$ cannot, in general, be compared. The possible existence of generalized principal eigenfunctions for $\lambda_1$ 
that are not of the form $\upe_{z^\star}\veu_{z^\star}$ remains an open question.

It is well-known that the equality $\lambda_1'=\lambda_1$ can be false: in the scalar case, the differential operator 
$u\mapsto -u''+u'$ is a classical counter-example.
The key to this counter-example is the nonzero advection term that moves the maximum of $\lambda_{1,z}$ away from $z=0$; a similar counter-example of a fully coupled cooperative parabolic system that does not reduce trivially to an elliptic 
scalar equation is, in spatial dimension $n=1$, $\cbQ = \partial_t - \partial_{xx} + \partial_{x} -(1/8)\vect{I}-\vect{M}$,
where $\vect{I}$ is the identity matrix in $\R^{N\times N}$ and $\vect{M}$ is the discrete Laplacian defined in \eqref{eq:discrete_Lap}. 
By uniqueness of the periodic principal eigenpair and the fact that the coefficients depend neither on time nor space, 
\[
\lambda_{1,z}=-\lambda_{\upPF}\left(-z(1-z)\vect{I}+\frac{1}{8}\vect{I}+\vect{M}\right) = z(1-z)-\frac{1}{8}-\lambda_{\upPF}(\vect{M}) = z(1-z)-\frac{1}{8},
\]
where $\lambda_\upPF$ denotes the \textit{Perron--Frobenius eigenvalue} of an essentially nonnegative irreducible matrix in $\R^{N\times N}$.
Therefore $\lambda_1'=-1/8<\lambda_1=1/8$, and this also confirms that, as in the scalar case, $\lambda_1$ and $\lambda_1'$ need 
not have the same sign.

In the elliptic scalar setting, the absence of advection implies that $z\mapsto\lambda_{1,z}$ is even, whence the equality
$\lambda_1=\lambda_1'$ follows \cite[Proposition 3.2]{Nadin_2007}. In the elliptic vector setting, Griette and Matano have
very recently proved with a counter-example that this is not the case \cite[Proposition 4.1]{Griette_Matano_2021}: 
the mere asymmetry of $\veL(x)$ can induce the strict inequality $\lambda_1'<\lambda_1$. For the sake of completeness,
we recall their counter-example in Remark \ref{rem:counter-example_evenness_without_advection}.

As in \cite{Nadin_2007}, our method of proof actually establishes a few results on $\lambda_1$ in arbitrary 
domains\footnote{In the spirit of Berestycki--Rossi
\cite{Berestycki_Ros}, $\lambda_1'$ can also be defined in an arbitrary domain $\Omega$ and further results on $\lambda_1(\Omega)$ and
$\lambda_1'(\Omega)$ are likely achievable. As the focus of this paper is on the influence of space--time periodicity, we do not pursue this direction here.}.
For any nonempty open connected set $\Omega\subset\R^n$, we define:
\begin{equation}\label{def:lambda1_Omega}
	\lambda_1(\Omega)=\sup\left\{ \lambda\in\R\ |\ \exists \veu\in\caC^{1,2}_{t-\upp}(\R\times\Omega,(\vez,\vei))\cap\caC^{1}(\R\times\overline{\Omega})\ \cbQ\veu\geq\lambda\veu \right\}.
\end{equation}

Since $\partial\Omega$ is not necessarily smooth, the set $\caC^1(\R\times\overline{\Omega})$ is understood here as the set of 
functions $\veu\in\caC^1(\R\times\Omega)$ such that both $\veu$ and $\nabla\veu$ can be continuously extended at any boundary point 
admitting a strong barrier (see Berestycki--Nirenberg--Varadhan \cite{Berestycki_Nir}).
The subset $\caC^1_0(\R\times\overline{\Omega})$ is the set of functions in $\caC^1(\R\times\overline{\Omega})$ vanishing continuously at
such boundary points.

\begin{defi}
Let $\Omega\subset\R^n$ be a nonempty open connected set. 
A \textit{generalized principal eigenfunction associated with $\lambda_1(\Omega)$} is a function 
$\veu\in\caC^{1,2}_{t-\upp}(\R\times\Omega,(\vez,\vei))\cap\caC^1_0(\R\times\overline{\Omega})$ such that $\cbQ\veu=\lambda_1\veu$.
\end{defi}

\begin{thm}\label{thm:existence_characterization_Omega}
Let $\Omega\subset\R^n$ be a nonempty open connected set such that there exists $x_0\in\Omega$ satisfying $[x_0,x_0+L]\subset\Omega$. 
Then the generalized principal eigenvalue $\lambda_1(\Omega)$ is a well-defined real number
and there exists an associated generalized principal eigenfunction.

If $\Omega=\R^n$, $\lambda_1(\Omega)=\lambda_1$. If $\Omega$ is bounded and smooth, $\lambda_1(\Omega)=\lambda_{1,\upDir}(\Omega)$.

Furthermore, the following max--min characterization holds true:
\[
	\lambda_1(\Omega)=\max_{\veu\in\caC^{1,2}_{t-\upp}(\R\times\Omega,(\vez,\vei))\cap\caC^{1}(\R\times\overline{\Omega})}\min_{i\in[N]}\inf_{\R\times\Omega}\left(\frac{(\cbQ\veu)_i}{u_i}\right).
\]
\end{thm}

\subsubsection{Monotonic or convex dependence with respect to the coefficients}\label{sec:theorems_monotonic_convex_dependence}
As an immediate corollary of the max--min characterization of Theorem \ref{thm:existence_characterization_Rn}, we already know 
that the eigenvalues $\lambda_{1,z}$, as functions of the matrix entries $l_{i,j}$, are decreasing: if $l_{i,j}<\widetilde{l}_{i,j}$ 
(\textit{i.e.}, $(t,x)\mapsto\widetilde{l}_{i,j}(t,x)-l_{i,j}(t,x)$ is a nonnegative nonzero function), then
$\lambda_{1,z}(l_{i,j})>\lambda_{1,z}(\widetilde{l}_{i,j})$. This applies in particular to $\lambda_1$
and $\lambda_1'$, by virtue of the identifications $\lambda_1=\max\lambda_{1,z}$ and $\lambda_1'=\lambda_{1,0}$.

Our first theorem on coefficient dependence is concerned with the concavity of the eigenvalues $\lambda_{1,z}$ as functions of the entries
$l_{i,j}$. It generalizes a well-known result by Nussbaum \cite{Nussbaum_1986} on matrices in $\R^{N\times N}$ 
as well as a result by Nadin \cite{Nadin_2007} on the scalar parabolic case. 

\begin{thm}\label{thm:concavity_eigenvalue_L}
Let $z\in\R^n$ and let 
\[
\left(\veL[s]\right)_{s\in[0,1]}\in\left(\caC^{\delta/2,\delta}_\upp(\R\times\R^n,\R^{N\times N})\right)^{[0,1]}
\]
be a family of matrices satisfying the same assumptions as $\veL$ (\textit{i.e.}, \ref{ass:cooperative}, \ref{ass:irreducible})
and such that, for all $(t,x)\in\R\times\R^n$ and $i\in[N]$,
\begin{enumerate}
    \item $s\mapsto l_{i,i}[s](t,x)$ is convex;
    \item for all $j\in[N]\backslash\{i\}$, $s\mapsto l_{i,j}[s](t,x)$ is either identically zero or log-convex.
\end{enumerate}

Then the map
\[
    s\in[0,1]\mapsto \lambda_{1,\upp}(\cbQ_z[s]),
\]
where $\cbQ_z[s]$ is the operator $\cbQ_z$ with $\veL$ replaced by $\veL[s]$, is affine or strictly concave.

It is affine if and only if 
there exist a constant vector $\vect{b}\gg\vez$, a function $\vect{c}\in\caC_\upp(\R\times\R^n,(\vez,\vei))$ and 
a function $\vect{f}\in\caC_{\upp}(\R,\R^N)$ satisfying $\int_0^T\vect{f}\in\vspan(\veo)$ such that the entries of $\veL[s]$ have the form:
\begin{equation*}
    l_{i,j}[s]:(t,x)\mapsto
    \begin{cases}
        l_{i,i}[0](t,x)-sf_i(t) & \text{if }i=j, \\
        l_{i,j}[0](t,x)\left(\frac{b_j}{c_i(t,x)}\right)^s \upe^{s\left(\int_0^t f_j-\frac{t}{T}\int_0^T f_j\right)} & \text{if }i\neq j,
    \end{cases}
\end{equation*}
and such that the function $\vect{c}$ satisfies, at all $(t,x)\in\clOmper$ and for each $i\in[N]$, 
\[
    c_i(t,x)=b_i\upe^{\int_0^t f_i-\frac{t}{T}\int_0^T f_i}\quad\text{or}\quad\forall j\in[N]\backslash\{i\},\ l_{i,j}[0](t,x)=0.
\]
\end{thm}

As explained in Remark \ref{rem:sharpness_of_NSC_for_equality_case}, the function $\vect{c}$ in the above statement is in
general not uniquely determined, but it is so if, for instance, $\veL[0]$ is pointwise irreducible.

Although Theorem \ref{thm:concavity_eigenvalue_L} directly applies to $\lambda_1'=\lambda_{1,0}$, we are only able to prove
a weaker concavity property on the generalized principal eigenvalue $\lambda_1$ in arbitrary domains -- in bounded and smooth domains,
a result exactly analogous to Theorem \ref{thm:concavity_eigenvalue_L} applies, see Proposition \ref{prop:concavity_lambda1_Omega_L}.
Similarly, in the elliptic case with general spatial heterogeneities in $\R^n$, Arapostathis--Biswas--Pradhan 
\cite[Lemma 2.3]{Arapostathis_Biswas_Pradhan_2020} proved the concavity of $\lambda_1$ with respect to the diagonal
entries of $\veL$ -- they did not consider the off-diagonal entries but, their arguments being the same as ours, their result can be
extended accordingly.

\begin{thm}\label{thm:concavity_lambda1_L}
Let $\Omega\subset\R^n$ be a nonempty open connected set such that there exists $x_0\in\Omega$ satisfying $[x_0,x_0+L]\subset\Omega$.

Let 
\[
    \left(\veL[s]\right)_{s\in[0,1]}\in\left(\caC^{\delta/2,\delta}_\upp(\R\times\Omega,\R^{N\times N})\right)^{[0,1]}
\]
be a family of matrices satisfying the same assumptions as $\veL$ (\textit{i.e.}, \ref{ass:cooperative}, \ref{ass:irreducible}) 
and such that, for all $(t,x)\in\R\times\Omega$ and $i\in[N]$,
\begin{enumerate}
    \item $s\mapsto l_{i,i}[s](t,x)$ is convex;
    \item for all $j\in[N]\backslash\{i\}$, $s\mapsto l_{i,j}[s](t,x)$ is either identically zero or log-convex.
\end{enumerate}

Then the mapping $s\in[0,1]\mapsto \lambda_1(\Omega,\cbQ[s])$, where $\cbQ[s]$ is the operator $\cbQ$ with $\veL$ replaced $\veL[s]$, is concave.
\end{thm}

Monotonicity or convexity results on the dependence on the diffusion matrices $A_i$ or the advection vectors $q_i$ are in 
full generality false (in the scalar setting, cases of non-monotonic and non-concave dependence with respect to the 
diffusion rate are exhibited in Hutson--Mischaikow--Polacik \cite{Hutson_Mischaikow_Polacik}). 

\subsubsection{Asymptotic dependence with respect to the coefficients}\label{sec:theorems_asymptotics}
The next theorem shows how the generalized principal eigenvalues $\lambda_{1,z}$ and $\lambda_1$
behave close to the boundary where \ref{ass:ellipticity}, \ref{ass:cooperative} and \ref{ass:smooth_periodic} 
remain satisfied but the full coupling assumption \ref{ass:irreducible} does not\footnote{
If $\cbQ$ is spatio-temporally homogeneous, then the theorem reduces to the well-known continuity of the 
dominant eigenvalue in the set of essentially nonnegative square matrices.}. We recall that a nonnegative square matrix can be 
conjugated into a block upper triangular Frobenius normal form by a permutation matrix, with each diagonal block an irreducible nonnegative square matrix
(recall that $1\times 1$ matrices are by convention referred to as irreducible even if zero). 
For a space-time periodic cooperative parabolic operator of the form $\dcbP-\veL$ but where $\veL$ does not satisfy \ref{ass:irreducible},
conjugating with the permutation matrix associated with the aforementioned Frobenius normal form of the matrix $\overline{\veL}$ brings similarly the system into block upper triangular form with each block satisfying
\ref{ass:irreducible}. Therefore we can assume without loss of generality 
that the limiting matrix $\veL$ is already in block upper triangular form with each block satisfying \ref{ass:irreducible}.

\begin{thm}\label{thm:continuity_eigenvalue_L}
Let $\veL^\triangle\in\caC_\upp^{\delta/2,\delta}(\R\times\R^n,\R^{N\times N})$ be a block upper triangular essentially nonnegative matrix. 
Let $N'\in[N]$ and $(N_k)_{k\in[N'-1]}$ such that 
\[
    N_0=0< 1\leq N_1\leq N_2\leq\dots\leq N_{N'-1}\leq N_{N'}=N
\]
and such that 
\[
    (l^\triangle_{i,j})_{(i,j)\in([N_k]\backslash[N_{k-1}])^2}
\]
is the $k$-th diagonal block of $\veL^\triangle$ (with the convention $[0]=\emptyset$). Assume
\[
    \left(\max_{(t,x)\in\clOmper}l^\triangle_{i,j}(t,x)\right)_{(i,j)\in([N_k]\backslash[N_{k-1}])^2}\quad\text{is irreducible for all }k\in[N'].
\]

Let
\[
    \cbQ_k= \diag(\caP_i)_{i\in[N_k]\backslash[N_{k-1}]}-(l^\triangle_{i,j})_{(i,j)\in([N_k]\backslash[N_{k-1}])^2}\quad\text{for all }k\in[N'].
\]

Then, as $\veL\to\veL^\triangle$ in $\caC_\upp^{\delta/2,\delta}(\R\times\R^n,\R^{N\times N})$, 
\begin{equation*}
    \lambda_{1,z}(\cbQ)\to\min_{k\in[N']}\lambda_{1,z}\left(\cbQ_k\right)\quad\text{for all }z\in\R^n,
\end{equation*}
\begin{equation*}
    \lambda_1(\cbQ)\to\max_{z\in\R^n}\min_{k\in[N']}\lambda_{1,z}(\cbQ_{k})\leq\min_{k\in[N']}\lambda_1(\cbQ_{k}).
\end{equation*}
\end{thm}

We comment specifically on this important result in Subsection \ref{sec:extension_coupling_default}.

The next theorem is concerned with concurrently vanishing diffusion and advection rates -- the question of vanishing diffusions rates when
the advection rates remain nonnegligible is much more difficult, even in the scalar case \cite{Liu_Lou_Peng_Zhou-1}, and is beyond the 
scope of this paper; for now, it remains open.

In the statement below, we use the generalized principal eigenvalues of a time periodic weakly coupled linear 
degenerate parabolic operator, that combines uniformly parabolic equations and ordinary 
differential equations through cooperative coupling. Such operators satisfy a form of strong 
comparison principle and, consequently, admit generalized principal eigenvalues $\lambda_1'$, $\lambda_1$, $(\lambda_{1,z})_{z\in\R^n}$ that could be studied quite similarly. 
For the case of bounded domains with homogeneous Dirichlet boundary 
conditions, we refer to Liang, Zhang and Zhao \cite{Liang_Zhang_Zhao_2017}; the
adaptation to space-time periodic settings should be straightforward, following ideas
developed in the present paper.
    
\begin{thm}\label{thm:continuity_eigenvalue_diffusion_advection}
Let $\vect{f}\in\caC^1\left([0,+\infty),[\vez,\vei)\right)$ such that $\vect{f}^{-1}(\{\vez\})=\{0\}$ and 
$\vect{f}'(0)\neq\vez$.

For all $\varepsilon\geq 0$, let $(q_i^\varepsilon)_{i\in[N]}\in\caC^{\delta/2,\delta}_\upp(\R\times\R^n,\R^n)$. Assume
$(q_i^\varepsilon)_{i\in[N]}\to (q_i^0)_{i\in[N]}$ in $\caC^{\delta/2,\delta}_\upp(\R\times\R^n,\R^n)$ as $\varepsilon\to 0$.

Let $\cbQ_{\varepsilon}$ be the operator $\cbQ$ with $(A_i)_{i\in[N]}$ and $(q_i)_{i\in[N]}$ replaced respectively by
$(f_i(\varepsilon)^2 A_i)_{i\in[N]}$ and $(f_i(\varepsilon)q_i^\varepsilon)_{i\in[N]}$.

Denote (with a slight abuse of notation), for every $x\in[0,L]$,
\[
    \left(\veL(x),(A_i(x))_{i\in[N]},(q_i(x))_{i\in[N]}\right):t\mapsto\left(\veL(t,x),(A_i(t,x))_{i\in[N]},(q_i(t,x))_{i\in[N]}\right).
\]

Then, for all $z\in\R^n$, 
\begin{equation*}
    \liminf_{\substack{\varepsilon\to 0\\\varepsilon>0}}\lambda_{1,z}\left(\cbQ_\varepsilon\right)\geq \min_{x\in[0,L]}\lambda_{1,\upp}\left(\frac{\upd}{\upd t}-\veL(x)\right).
\end{equation*}

Furthermore, the equality
\[
    \lim_{\substack{\varepsilon\to 0\\\varepsilon>0}}\lambda_{1,z}(\cbQ_\varepsilon)= \min_{x\in[0,L]}\lambda_{1,\upp}\left(\frac{\upd}{\upd t}-\veL(x)\right).
\]
holds true for all $z\in\R^n$ if at least one of the following conditions is satisfied:
\begin{enumerate}
    \item all coefficients of $\cbQ_\varepsilon$ do not depend on $x$;
    \item there exists $\underline{x}\in[0,L]$ such that 
    \[
        \lambda_{1,\upp}\left(\frac{\upd}{\upd t}-\veL(\underline{x})\right)=\min_{x\in[0,L]}\lambda_{1,\upp}\left(\frac{\upd}{\upd t}-\veL(x)\right)
    \] 
    and such that the operator
    \[
        \widetilde{\cbQ} = \partial_t-\diag(f_i'(0)^2\nabla\cdot(A_i(\underline{x})\nabla)-f_i'(0) q_i^0(\underline{x})\cdot\nabla)-\veL(\underline{x})
    \]
    satisfies
    \[
        \lambda_1(\widetilde{\cbQ})=\lambda_1'(\widetilde{\cbQ}).
    \]
\end{enumerate}
\end{thm}

The assumptions on $\vect{f}$ could be relaxed with a marginal impact on the proof; they are mostly
used to simplify the statement of the Theorem and are in any case sufficient for our purposes.

We conjecture that, under reasonable regularity assumptions, the inferior and superior limits of
$\lambda_{1,z}(\cbQ_\varepsilon)$ should always coincide; at this stage, it remains an open question.

The above two conditions that are sufficient to guarantee that 
\begin{equation*}
    \limsup_{\substack{\varepsilon\to 0\\\varepsilon>0}}\lambda_{1,z}\left(\cbQ_\varepsilon\right)\leq \min_{x\in[0,L]}\lambda_{1,\upp}\left(\frac{\upd}{\upd t}-\veL(x)\right)
\end{equation*}
cover very different cases. However, as explained in Remark \ref{rem:counter-example_vanishing_diffusion}, 
there also exist cases where the limit of $\lambda_{1,z}(\cbQ_\varepsilon)$ is well-defined and satisfies
\begin{equation*}
    \lim_{\substack{\varepsilon\to 0\\\varepsilon>0}}\lambda_{1,z}\left(\cbQ_\varepsilon\right)> \min_{x\in[0,L]}\lambda_{1,\upp}\left(\frac{\upd}{\upd t}-\veL(x)\right).
\end{equation*}
More generally, the problem of characterizing the limit is subtle and, in our opinion, worthy of 
future attention. It requires a better understanding of the impact of non-vanishing advection rates; indeed, the condition 
$\lambda_1(\widetilde{\cbQ})=\lambda_1'(\widetilde{\cbQ})$ is, by virtue of Corollary 
\ref{cor:lambda_1_space_homogeneous_no_advection}, satisfied as soon as
$\widetilde{\cbQ}$ contains no advection term, \textit{i.e.} $f_1'(0)q_1^0=f_2'(0)q_2^0=\dots=f_N'(0)q_N^0=0$. 

Beyond pointing out cases where $z\mapsto\lambda_{1,z}$ converges pointwise to a constant in a correctly scaled vanishing diffusion--advection 
limit\footnote{Actually, since $z\mapsto\lambda_{1,z}$ is concave, this convergence is locally uniform in $z$, but there is really 
no hope for uniform convergence since $\lambda_{1,z}\to-\infty$ as $|z|\to+\infty$. Also, in general, $\lambda_1=\max\lambda_{1,z}$ does not 
converge to the same limit: indeed, even for the spatio-temporally homogeneous one-dimensional operator
$\partial_t-\varepsilon^2\partial_{xx}+\varepsilon\partial_x-\vect{M}$, with $\vect{M}$ the discrete Laplacian
defined in \eqref{eq:discrete_Lap}, the maximum of 
$z\mapsto\lambda_{1,z}$ is $1/4$, independently of $\varepsilon$, whereas the pointwise limit as $\varepsilon\to 0$ is $0$.}, 
the possible dependence on $\varepsilon$ of the advection rates $q_i^\varepsilon$ implies that the following two distinct limits are true:
\begin{equation*}
    \liminf_{\substack{\varepsilon\to 0 \\\varepsilon>0}}\lambda_1'\left(\partial_t-\diag(\varepsilon^2\nabla\cdot(A_i\nabla)-\varepsilon q_i\cdot\nabla)-\veL\right)\geq\min_{x\in[0,L]}\lambda_{1,\upp}\left(\frac{\upd}{\upd t}-\veL(x)\right),
\end{equation*}
\begin{equation*}
    \lim_{\substack{\vect{d}\to\vez \\\vect{d}\gg\vez}}\lambda_1'\left(\partial_t-\diag(\vect{d})\diag\left(\nabla\cdot(A_i\nabla)- q_i\cdot\nabla\right)-\veL\right)=\min_{x\in[0,L]}\lambda_{1,\upp}\left(\frac{\upd}{\upd t}-\veL(x)\right).
\end{equation*}
Although the two limits look similar, they do not refer to the same underlying questions. 

The first one is related to a slowly varying medium. Indeed, assume, for the sake of simplicity, that all $L_\alpha$ 
coincide and denote $\varepsilon=L_1^{-1}$. Then the change of variable $x\to \varepsilon x$ changes the 
$[0,T]\times[0,\varepsilon^{-1}]^n$-periodic operator $\cbQ$ into the $[0,T]\times[0,1]^n$-periodic operator
$\partial_t-\varepsilon^2\diag(\nabla\cdot(A_i^\varepsilon\nabla)+\varepsilon\diag(q_i^\varepsilon\cdot\nabla)-\veL^\varepsilon$, where
\[
    ((A_i^\varepsilon)_{i\in[n]},(q_i^\varepsilon)_{i\in[N]},\veL^\varepsilon):(t,x)\mapsto ((A_i)_{i\in[N]},(q_i)_{i\in[N]},\veL)\left(t,\frac{x}{\varepsilon}\right).
\] 
In the scalar case, the limit $\varepsilon\to 0$ has been studied by Hamel, Nadin and Roques \cite{Hamel_Nadin_Roques_2011} in
the elliptic case and by Nadin \cite{Nadin_2007} in the parabolic case. 
In the vector case with temporally homogeneous coefficients and an extra time scaling, it has been studied by Allaire and Hutridurga
\cite{Allaire_Hutridurga_2015} (parabolic scaling) and by Mirrahimi and Souganidis \cite{Mirrahimi_Souganidis_2013} (hyperbolic scaling).

The second one corresponds for instance to the early stages $t\to t/\varepsilon$ of a fast-reaction system $\veL\to\frac{1}{\varepsilon}\veL$, 
when spatial processes (dispersal, transport) are still negligible. In bounded domains with Dirichlet boundary conditions, the singular
limit $\varepsilon\to 0$ has been studied by Bai and He \cite{Bai_He_2020}. As explained by Lam and Lou in their paper on the Neumannn
elliptic case \cite{Lam_Lou_2015}, the fact that the vanishing parameter is the vector $\vect{d}\in\R^N$ and not a scalar amplitude 
parameter $\varepsilon\in\R$ is meaningful: the spatial processes for one species may be much faster than for the others 
(\textit{e.g.}, $d_N=\varepsilon\max_{i\in[N-1]}d_i$), as long as they are all slow compared to the parameter $\varepsilon$ 
measuring the time scale and the speed of the reaction. 

By considering a uniform limit and advection rates $q_i^\varepsilon$ that might vanish as $\varepsilon\to 0$, 
we bring together these two frameworks and prove both limits concurrently. We believe this approach is new.

The next theorem is, on the contrary, concerned with how very large diffusion rates impact the 
periodic principal eigenvalue $\lambda_{1,\upp}$. 
The large diffusivity limit for the whole family $(\lambda_{1,z}(\cbQ_{\vect{d}}))_{z\in\R^n}$ is 
an entirely different problem, since the large parameter $\vect{d}$ appears also in the zeroth order 
term which will therefore blow-up as soon as $z$ is nonzero\footnote{More precisely, as will
be shown in Corollary \ref{cor:comparison_eigenvalues_underline_overline_L},
$-|P_1(z)|\max_{i\in[N]}d_i\leq \lambda_{1,z}(\cbQ_{\vect{d}})\leq -|P_2(z)|\min_{i\in[N]}d_i$
for some second-order polynomials $P_1, P_2$. What would then be relevant would be to figure out
an asymptotic expansion of $\lambda_{1,z}(\cbQ_{\vect{d}})$.}. This problem is beyond our scope
and is left open.

The question of very large advection rates, already much more delicate in the scalar case 
\cite{Liu_Lou_Peng_Zhou-2}, is also beyond our scope.

\begin{thm}\label{thm:limit_eigenvalue_large_diffusion}
Let 
\[
    \left((\langle A_i\rangle,\langle q_i\rangle)_{i\in[N]},\langle\veL\rangle\right):t\mapsto
    \frac{1}{|[0,L]|}\int_{[0,L]}\left((A_i,q_i)_{i\in[N]},\veL\right)(t,x)\upd x
\]
and, for all $\vect{d}\in(\vez,\vei)$, let $\cbQ_{\vect{d}}$ be the operator $\cbQ$ with $(A_i)_{i\in[N]}$ replaced by $(d_i A_i)_{i\in[N]}$.

Then, as $\min_{i\in[N]}d_i\to+\infty$, 
\[
    \lambda_{1,\upp}(\cbQ_{\vect{d}})\to\lambda_{1,\upp}\left(\frac{\upd}{\upd t}-\langle\veL\rangle\right).
\]
\end{thm}

This homogenization theorem basically means that, for the periodic principal eigenvalue $\lambda_{1,\upp}$,
very large diffusion rates tend to replace spatially heterogeneous
coefficients by their averages in space. 
Again, the fact that the multiplicative coefficients $d_i$ can go to $+\infty$ at different 
rates is meaningful.

In the scalar case \cite{Nadin_2007}, the inequality 
\[
    \lambda_{1,\upp}\left(\frac{\upd}{\upd t}-\langle l_{1,1}\rangle\right)\geq \min_{x\in[0,L]}\lambda_{1,\upp}\left(\frac{\upd}{\upd t}-l_{1,1}(x)\right)
\]
holds, and implies a comparison between the large diffusion asymptotic and the vanishing diffusion asymptotic. In the vector case,
this inequality is still true if for instance the periodic principal eigenfunction associated with $\partial_t-\veL(x)$ depends 
neither on $t$ nor on $x$. Indeed, by integrating the equality it satisfies over $[0,L]$,
\[
    \lambda_{1,\upp}\left(\frac{\upd}{\upd t}-\langle\veL\rangle\right)=\frac{1}{|[0,L]|}\int_{[0,L]}\lambda_{1,\upp}\left(\frac{\upd}{\upd t}-\veL(\cdot)\right)
    \geq \min_{x\in[0,L]}\lambda_{1,\upp}\left(\frac{\upd}{\upd t}-\veL(x)\right).
\]
However it is not true in full generality, as shown by the counter-example of Remark \ref{rem:counter-example_comparison_small_large_diffusion}.

\subsubsection{Dependence on the space or time frequencies}\label{sec:theorems_dependence_frequencies}
As explained before, assuming that all spatial periods $L_\alpha$ coincide and changing appropriately the variables $t$ and $x$, the
$\clOmper$-periodic operator $\cbQ$ becomes the following $[0,1]\times[0,1]^n$-periodic operator:
\[
    \cbQ_{T,L_1}=\frac{1}{T}\partial_t-\frac{1}{L_1^2}\diag(\nabla\cdot(A_i^\diamondsuit\nabla))+\frac{1}{L_1}\diag(q_i^\diamondsuit\cdot\nabla)-\veL^\diamondsuit,
\]
where
\[
    \left(\left(A_i^\diamondsuit,q_i^\diamondsuit\right)_{i\in[N]},\veL^\diamondsuit\right):(t,x)\mapsto \left((A_i,q_i)_{i\in[N]},\veL\right)\left(Tt,L_1x\right).
\]

Theorems \ref{thm:continuity_eigenvalue_diffusion_advection} and \ref{thm:limit_eigenvalue_large_diffusion} and Remark
\ref{rem:counter-example_comparison_small_large_diffusion} have immediate interpretations in this context, summarized in the following corollary.

However, before stating the corollary, we draw attention on the fact that after such changes of variables, the family of 
generalized principal eigenvalues $(\lambda_{1,z})_{z\in\R^n}$ is dilated. Indeed, $\lambda_{1,z}(\cbQ)=\lambda_1'(\cbQ_z)$
coincides with
\[
\lambda_{1,L_1 z}\left(\cbQ_{T,L_1}\right)=\lambda_1'\left(\cbQ_{T,L_1}-\diag\left(2\frac{1}{L_1}A_i^\diamondsuit z\cdot\nabla+z\cdot A_i^\diamondsuit z+\frac{1}{L_1}\nabla\cdot\left(A_i^\diamondsuit z\right)-q_i^\diamondsuit\cdot z\right)\right)
\]
and not with 
\[
\lambda_{1,z}\left(\cbQ_{T,L_1}\right)=\lambda_1'\left(\cbQ_{T,L_1}-\diag\left(2\frac{1}{L_1^2}A_i^\diamondsuit z\cdot\nabla+\frac{1}{L_1^2}z\cdot A_i^\diamondsuit z+\frac{1}{L_1^2}\nabla\cdot\left(A_i^\diamondsuit z\right)-\frac{1}{L_1}q_i^\diamondsuit\cdot z\right)\right).
\]
It is in some sense disappointing that the limit $L_1\to+\infty$ in the corollary below is concerned with the
wrong family. In fact it is known that the right family requires a different asymptotic analysis. It is
outside the scope of the present paper and instead we refer for instance
to \cite[Proposition 3.1]{Hamel_Nadin_Roques_2011} where the particular case of scalar space periodic elliptic equations
is investigated.

\begin{cor}\label{cor:asymptotics_small_large_spatial_frequency}
If $q_i^\diamondsuit=0$ for each $i\in[N]$, then the generalized principal eigenvalues of $\cbQ_{T,L_1}$ satisfy the following properties.
\begin{enumerate}
    \item If $L_1\to +\infty$, then 
    \[
        \lambda_{1,z}(\cbQ_{T,L_1})\to\min_{x\in[0,1]^n}\lambda_{1,\upp}\left(\frac{1}{T}\frac{\upd}{\upd t}-\veL^\diamondsuit(x)\right)\quad\text{for all }z\in\R^n.
    \]
    \item If $L_1\to 0$, then
    \[
        \lambda_{1,\upp}(\cbQ_{T,L_1})\to\lambda_{1,\upp}\left(\frac{1}{T}\frac{\upd}{\upd t}-\langle\veL^\diamondsuit\rangle\right).
    \]
    \item There exist a choice of $\left(A_i^\diamondsuit\right)_{i\in[N]}$ and $\veL^\diamondsuit$ such that	
    $L_1\mapsto\lambda_{1,\upp}(\cbQ_{T,L_1})$ is decreasing, a choice such that it is constant and a choice such that it is neither.
\end{enumerate}
\end{cor}

It remains to investigate the effect of the time frequency $1/T$. In the case of a scalar equation in a bounded domain 
with Dirichlet boundary
conditions, this problem was recently studied thoroughly by Liu, Lou, Peng and Zhou \cite{Liu_Lou_Peng_Zhou}. 
They identified cases where $T\mapsto\lambda_{1,\upDir}(\cbQ_{T,L_1})$ is constant,
cases where it is decreasing and cases where it is neither; additionally, they  studied the asymptotics $T\to 0$ 
and $T\to+\infty$ -- reusing
the proof of Nadin \cite{Nadin_2007} for the limit $T\to 0$. Recently, similar results were obtained
for the space homogeneous, time periodic operator \cite{Benaim_Lobry_Sari_Strickler_2023}.
We will adapt the scalar arguments \cite{Liu_Lou_Peng_Zhou,Nadin_2007} to prove the following result.

\begin{thm}\label{thm:limits_eigenvalue_time_frequency}
For all $\omega\in(0,+\infty)$, let $\cbQ_{\omega}$ be the operator $\cbQ$ with $\partial_t$ replaced by
$\omega\partial_t$. Then:
\begin{enumerate}
    \item if $\omega\to 0$, then for all $z\in\R^n$,
    \[
        \lambda_{1,z}(\cbQ_{\omega})\to\frac{1}{T}\int_0^T\lambda_{1,z}\left(-\diag(\nabla\cdot(A_i(t)\nabla)-q_i(t)\cdot\nabla)-\veL(t)\right)\upd t,
    \]
    \[
        \lambda_1(\cbQ_{\omega})\to\frac{1}{T}\int_0^T\lambda_1\left(-\diag(\nabla\cdot(A_i(t)\nabla)-q_i(t)\cdot\nabla)-\veL(t)\right)\upd t,
    \]
    where we denote (with a slight abuse of notation) $((A_i(t),q_i(t))_{i\in[N]},\veL(t)):x\mapsto((A_i(t,x),q_i(t,x))_{i\in[N]},\veL(t,x))$;
    \item if $\omega\to+\infty$, then for all $z\in\R^n$,
    \[
        \lambda_{1,z}(\cbQ_{\omega})\to\lambda_{1,z}\left(-\diag(\nabla\cdot(\hat{A}_i\nabla)-\hat{q}_i\cdot\nabla)-\hat{\veL}\right),
    \]
    \[
        \lambda_1(\cbQ_{\omega})\to\lambda_1\left(-\diag(\nabla\cdot(\hat{A}_i\nabla)-\hat{q}_i\cdot\nabla)-\hat{\veL}\right),
    \]
    where
    \[
        \left((\hat{A}_i,\hat{q}_i)_{i\in[N]},\hat{\veL}\right):x\mapsto
        \frac{1}{T}\int_0^T\left((A_i,q_i)_{i\in[N]},\veL\right)(t,x)\upd t.
    \]
\end{enumerate}
\end{thm}

\begin{cor}
The generalized principal eigenvalues of $\cbQ_{T,L_1}$ satisfy the following properties.
\begin{enumerate}
    \item If $T\to +\infty$, then for all $z\in\R^n$,
    \[
        \lambda_{1,z}(\cbQ_{T,L_1})\to\int_0^1\lambda_{1,z}\left(-\frac{1}{L_1^2}\diag(\nabla\cdot(A_i^\diamondsuit(t)\nabla))+\frac{1}{L_1}\diag(q_i^\diamondsuit(t)\cdot\nabla)-\veL^\diamondsuit(t)\right)\upd t.
    \]
    \[
        \lambda_1(\cbQ_{T,L_1})\to\int_0^1\lambda_1\left(-\frac{1}{L_1^2}\diag(\nabla\cdot(A_i^\diamondsuit(t)\nabla))+\frac{1}{L_1}\diag(q_i^\diamondsuit(t)\cdot\nabla)-\veL^\diamondsuit(t)\right)\upd t.
    \]
    \item If $T\to 0$, then for all $z\in\R^n$,
    \[
        \lambda_{1,z}(\cbQ_{T,L_1})\to\lambda_{1,z}\left(-\frac{1}{L_1^2}\diag(\nabla\cdot(\hat{A}_i^\diamondsuit\nabla))+\frac{1}{L_1}\diag(\hat{q}_i^\diamondsuit\cdot\nabla)-\hat{\veL}^\diamondsuit\right).
    \]
    \[
        \lambda_{1}(\cbQ_{T,L_1})\to\lambda_{1}\left(-\frac{1}{L_1^2}\diag(\nabla\cdot(\hat{A}_i^\diamondsuit\nabla))+\frac{1}{L_1}\diag(\hat{q}_i^\diamondsuit\cdot\nabla)-\hat{\veL}^\diamondsuit\right).
    \]
\end{enumerate}
\end{cor}

Theorem \ref{thm:limits_eigenvalue_time_frequency} shows that large time frequencies tend to replace time 
heterogenous coefficients by their averages in time, whereas small time frequencies tend to replace the parabolic 
operator by the elliptic operator parametrized by $t$ before averaging the eigenvalue.

\subsubsection{Formulas and estimates in special cases}\label{sec:theorems_explicit_formulas}

Recall the notations $\hat{A}_i$, $\hat{q}_i$, $\hat{\veL}$ for the averages
in time and $\langle A_i\rangle$, $\langle q_i\rangle$, $\langle\veL\rangle$ for the averages in space. In this section, for the sake
of brevity, we use these notations repeatedly. The combined notation, \textit{e.g.} $\langle\hat{\veL}\rangle$, denotes naturally
a space-time average.

As a preliminary, we note that, by simplicity of the periodic principal eigenspace, the generalized principal eigenvalue
$\lambda_{1,z}$ can be simplified when coefficients do not depend on space:
\begin{equation}\label{eq:reduction_lambdaz_space_homogeneous}
    \lambda_{1,z}=\lambda_{1,\upp}\left(\frac{\upd}{\upd t}-\veL-\diag\left(z\cdot A_iz - q_i\cdot z\right)\right)
\end{equation}
or when they do not depend on time:
\begin{equation}\label{eq:reduction_lambdaz_time_homogeneous}
    \lambda_{1,z}=\lambda_{1,\upp}\left(-\diag\left(\nabla\cdot\left(A_i\nabla\right)-\left(q_i-2A_i z\right)\cdot\nabla+\left(z\cdot A_i z+\nabla\cdot(A_i z) - q_i\cdot z\right)\right)-\veL\right).
\end{equation}
When all coefficients are space-time constant, the expression of generalized principal eigenvalue can be simplified even further:
\begin{equation}\label{eq:reduction_lambdaz_space-time_homogeneous}
    \lambda_{1,z}=-\lambda_\upPF\left(\veL+\diag\left(z\cdot A_iz - q_i\cdot z\right)\right).
\end{equation}
These reductions to ordinary differential, elliptic partial differential or matrix operators are valid for any $z\in\R^n$. 
Moreover, when considering only
the specific case $z=0$, the condition of space, time or space-time homogeneity needs to be verified only by $\veL$, and not by 
$(A_i)_{i\in[N]}$ or $(q_i)_{i\in[N]}$.
These reductions will be used repeatedly in the proofs and in the constructions of counter-examples.

In the statements below, the Perron--Frobenius eigenvalue of a reducible matrix is defined
by continuous extension of the dominant eigenvalue on the set of essentially nonnegative matrices; 
for the sake of simplicity, its nonnegative eigenvectors are still referred to as Perron--Frobenius eigenvectors, even though they 
might not be positive and the eigenvalue might not be simple (algebraically and/or geometrically).

As an immediate consequence of \eqref{eq:reduction_lambdaz_space_homogeneous} and of the monotonicity of $\lambda_{1,\upp}$ 
with respect to the diagonal entries of $\veL$, we have the following corollary, which shows that in the
absence of advection and of spatial heterogeneities, there is no gap between $\lambda_1$ and $\lambda_1'$.
\begin{cor}\label{cor:lambda_1_space_homogeneous_no_advection}
    Assume:
    \begin{enumerate}
        \item $(A_i)_{i\in[N]}$ and $\veL$ do not depend on $x$,
        \item for each $i\in[N]$, $q_i=0$.
    \end{enumerate}

    Then $\lambda_1=\lambda_1'$.
\end{cor}

Our next two theorems are concerned with operators whose coefficients depend only on time or only on space, and generalize 
\cite[Propositions 3.1 and 3.2]{Nadin_2007}. 

\begin{thm}\label{thm:lambdaz_space_homogeneous}
Assume:
\begin{enumerate}
    \item $(A_i)_{i\in[N]}$, $(q_i)_{i\in[N]}$ and $\veL$ do not depend on $x$,
    \item there exists a constant positive vector $\veu\in(\vez,\vei)$ such that $\veu$ is a Perron--Frobenius eigenvector of $\veL(t)$ 
    for all $t\in\R$.
\end{enumerate}

Let $z\in\R^n$. The equality
\[
    \lambda_{1,z}=-z\cdot \hat{A}_1 z+\hat{q}_1 \cdot z-\lambda_{\upPF}(\hat{\veL})
\]
is true if $z=0$ or $(A_1,q_1)=(A_2,q_2)=\dots=(A_N,q_N)$.

Consequently, if:
\begin{enumerate}
    \item $(A_1,q_1)=(A_2,q_2)=\dots=(A_N,q_N)$,
    \item $\hat{q}_1=0$,
\end{enumerate}
then
\[
    \lambda_1=\lambda_1'=-\lambda_{\upPF}(\hat{\veL}).
\]
\end{thm}

We  explain in Remark \ref{rem:counter-example_average_eigenvalue} that if the assumption on the existence of a constant positive
eigenvector is not satisfied, then the claimed equality on $\lambda_{1,z}$ is false in general. This is striking, since in the scalar 
setting, the existence of a constant positive eigenvector is not required.

\begin{thm}\label{thm:lambdaz_time_homogeneous}
Assume:
\begin{enumerate}
    \item $(A_i)_{i\in[N]}$ and $\veL$ do not depend on $t$,
    \item $\veL(x)$ is symmetric for all $x\in\R^n$,
    \item there exists $z\in\R^n$ and $Q\in\caC^2(\R^n,\R)$ such that $\int_{[0,L]}\nabla Q=0$ and 
    \[
        A_1^{-1}q_1=A_2^{-1}q_2=\dots=A_N^{-1}q_N=2z+\nabla Q.
    \]
\end{enumerate}

Then
\[
    \lambda_1=\lambda_{1,z}=\min_{\veu\in\caC^2_{\upp}(\R^n,\R^N)\backslash\{\vez\}}\frac{\displaystyle\int_{[0,L]}\left(\sum_{i=1}^N \nabla u_i\cdot A_i\nabla u_i-\veu^\upT\veL_{Q,z}\veu\right)}{\displaystyle\int_{[0,L]}|\veu|^2},
\]
where
\[
    \veL_{Q,z}=\veL+\diag\left(\frac12\nabla\cdot(A_i\nabla Q)-\frac{1}{4}\nabla Q\cdot A_i\nabla Q+\nabla\cdot\left(A_i z\right)-z\cdot A_i(z+\nabla Q)\right).
\]

Furthermore, if there exists a constant positive vector $\veu\in(\vez,\vei)$ such that $\veu$ is a Perron--Frobenius eigenvector 
of $\veL_{Q,z}(x)$ for all $x\in\R^n$, then
\[
    \lambda_1=\lambda_1'\leq -\lambda_\upPF(\langle\veL_{Q,z}\rangle).
\]
\end{thm}
    
We will explain in Remark \ref{rem:counter-example_variational_formula} that if $\veL$ is not symmetric, then, even in the 
simple case $z=q_1=q_2=\cdots=0$, there are counter-examples where
\[
    \lambda_1'>\min_{\veu\in\caC^2_{\upp}(\R^n,\R^N)\backslash\{\vez\}}\frac{\displaystyle\int_{[0,L]}\left(\sum_{i=1}^N \nabla u_i\cdot A_i\nabla u_i-\veu^\upT\veL\veu\right)}{\displaystyle\int_{[0,L]}|\veu|^2}.
\]
As noted before, we will also recall in Remark \ref{rem:counter-example_evenness_without_advection} the counter-example of Griette--Matano 
\cite{Griette_Matano_2021} where the mere asymmetry of $\veL$ breaks the equality $\lambda_1=\lambda_1'$.

We will also explain in Remark \ref{rem:counter-example_self-adjoint_average_value} that if the assumption on the existence of a 
constant positive eigenvector is not satisfied, then the inequality between $-\lambda_1'$ and the Perron--Frobenius eigenvalue
of $\langle\veL_{Q,z}\rangle$ can fail. Again, in the scalar case, this assumption is not required \cite{Berestycki_Ham_1}.

The following theorem is similar in spirit and requires a line-sum-symmetry assumption ($\veL\veo=\veL^\upT\veo$). Examples
of line-sum-symmetric essentially nonnegative matrices are doubly stochastic matrices, essentially nonnegative symmetric matrices 
and essentially nonnegative circulant matrices. For more details on line-sum-symmetric matrices, we refer for instance to
Eaves--Hoffman--Rothblum--Schneider \cite{Eaves_1985}.

\begin{thm}\label{thm:lambdaz_divergence-free}
Assume $\veL(t,x)$ is line-sum-symmetric at all $(t,x)\in\clOmper$.

Let $z\in\R^n$. If, for all $i\in[N]$, $q_i\in\caC^1_{\upp}(\R^n,\R^n)$ and $\nabla\cdot(q_i-2A_iz)=0$, then
\[
    \lambda_{1,z} \leq -\frac{1}{N}\left(\sum_{i,j=1}^N\langle\hat{l}_{i,j}\rangle+z\cdot\sum_{i=1}^N\left(\langle\hat{A}_i\rangle z-\langle\hat{q}_i\rangle\right)\right)
\]
and this inequality is an equality if $\veL+\diag(\nabla\cdot(A_i z)+z\cdot(A_i z -q_i))$ is irreducible at all $(t,x)\in\clOmper$
with Perron--Frobenius eigenvector $\veo$ and constant Perron--Frobenius eigenvalue.
\end{thm}

This theorem has several interesting consequences, which we detail in Corollaries \ref{cor:lambda1prime_line-sum-symmetry} 
and \ref{cor:lambda1_line-sum-symmetry}.

Two similar results without line-sum-symmetry follow.

\begin{thm}\label{thm:lambdaz_ari-geo_space_averages}
Let $z\in\R^n$. If, for all $i\in[N]$, $q_i\in\caC^1_{\upp}(\R^n,\R^n)$ and $\nabla\cdot(q_i-2A_iz)=0$, then
\[
    \lambda_{1,z}\leq \lambda_{1,z}\left(\partial_t-\diag(\nabla\cdot(\langle A_i\rangle\nabla)-\langle q_i\rangle)-\veL^\#\right),
\]
where the entries of the matrix $\veL^\#=\left(l_{i,j}^\#\right)_{(i,j)\in[N]^2}$ are defined by:
\[
    l_{i,j}^\#:t\mapsto
    \begin{cases}
        \frac{1}{|[0,L]|}\int_{[0,L]}l_{i,i}(t,x)\upd x & \text{if }i=j, \\
        \exp\left(\frac{1}{|[0,L]|}\int_{[0,L]}\ln l_{i,j}(t,x)\upd x\right) & \text{if }i\neq j\text{ and }\displaystyle\min_{(t,x)\in\clOmper}l_{i,j}(t,x)> 0, \\
        0 & \text{otherwise}.
    \end{cases}
\]
\end{thm}

\begin{thm}\label{thm:lambdaz_ari-geo_time_averages}
Let $z\in\R^n$. If $(A_i)_{i\in[N]}$, $(q_i)_{i\in[N]}$ and $\veL$ do not depend on $x$, then
\[
    \lambda_{1,z}\leq -\lambda_{\upPF}\left(\veL^\flat+\diag\left(z\cdot\hat{A}_iz-\hat{q}_i\cdot z\right)\right),
\]
where the entries of the matrix $\veL^\flat=\left(l_{i,j}^\flat\right)_{(i,j)\in[N]^2}$ are defined by:
\[
    l_{i,j}^\flat=
    \begin{cases}
        \frac{1}{T}\int_0^T l_{i,i} & \text{if }i=j, \\
        \exp\left(\frac{1}{T}\int_0^T\ln l_{i,j}\right) & \text{if }i\neq j\text{ and }\displaystyle\min_{t\in[0,T]}l_{i,j}(t)> 0, \\
        0 & \text{otherwise}.
    \end{cases}
\]
\end{thm}

The operator introduced in Theorem \ref{thm:lambdaz_ari-geo_space_averages} is spatially homogeneous, so that
\[
\lambda_{1,z}\left(\partial_t-\diag\left(\nabla\cdot\left(\langle A_i\rangle\nabla\right)-\langle q_i\rangle\cdot\nabla \right)-\veL^\#\right)
=\lambda_{1,\upp}\left(\frac{\upd}{\upd t}-\veL^\#-\diag(z\cdot\langle A_i\rangle z-\langle q_i\rangle \cdot z)\right).
\]
Therefore the last two theorems can be applied consecutively to find the following corollary.
\begin{cor}\label{cor:lambdaz_ari-geo_space-time_averages}
    Let $z\in\R^n$. If, for all $i\in[N]$, $q_i\in\caC^1_{\upp}(\R^n,\R^n)$ and $\nabla\cdot(q_i-2A_iz)=0$, then
    \[
        \lambda_{1,z}\leq -\lambda_\upPF\left(\veL^{\#\flat}+\diag\left(z\cdot\langle\hat{A}_i\rangle z-\langle\hat{q}_i\rangle\cdot z\right)\right)
    \]
    where
    \[
        l^{\#\flat}_{i,j}=
        \begin{cases}
            \frac{1}{T|[0,L]|}\int_0^T\int_{[0,L]} l_{i,i} & \text{if }i=j, \\
            \exp\left(\frac{1}{T|[0,L]|}\int_0^T\int_{[0,L]}\ln l_{i,j}\right) & \text{if }i\neq j\text{ and }\displaystyle\min_{(t,x)\in\clOmper}l_{i,j}(t,x)> 0, \\
            0 & \text{otherwise}.
        \end{cases}
    \]
\end{cor}
We emphasize that this upper estimate accounts for off-diagonal entries of $\veL$ and is therefore better than the one
that could be obtained by writing $\veL\geq\diag(l_{i,i})$ and then using the well-known scalar estimate 
$\lambda_{1,z}(\caP_i-l_{i,i})\leq -\frac{1}{T|[0,L]|}\int_0^T\int_{[0,L]}(l_{i,i}+z\cdot A_i z-z\cdot q_i)$
under the assumption $\nabla\cdot(q_i-2A_iz)=0$.

Theorems \ref{thm:lambdaz_ari-geo_time_averages} and \ref{thm:lambdaz_ari-geo_space_averages} show that when comparing 
heterogeneous environments with averaged environments, heterogeneities tend 
to decrease the generalized principal eigenvalues, provided the geometric average is used for the off-diagonal entries of 
$\veL$. This is of course related to the convexity property of Theorem \ref{thm:concavity_eigenvalue_L}. This is
also related to the asymptotic results of Theorems \ref{thm:limit_eigenvalue_large_diffusion} and 
\ref{thm:limits_eigenvalue_time_frequency}, although in these asymptotics the off-diagonal entries are averaged with
the arithmetic mean instead of the geometric mean. By comparing the arithmetic and geometric averages and 
using the monotonicity of $\lambda_{1,z}$ with respect to $\veL$, we can try to compare these results; however, inequalities
are in the wrong sense. For instance, in the simple case $z=0$ with each $q_i$ divergence-free, what we get is:
\begin{align*}
    \lambda_{1,\upp}\left(\frac{\upd}{\upd t}-\veL^\#\right) & \geq
    \max\left[\lambda_1'(\cbQ),\lambda_{1,\upp}\left(\frac{\upd}{\upd t}-\langle\veL\rangle\right)\right] \\
    & =\max\left[\lambda_1'(\cbQ),\lim_{\min_{i\in[N]}d_i\to+\infty}\lambda_1'(\cbQ_{\vect{d}})\right].
\end{align*}

\subsubsection{Optimization}\label{sec:theorems_optimization}
Our first optimization result is a highly nontrivial generalization of a result on matrices of Neumann--Sze \cite{Neumann_Sze_2007}. 
To the best of our knowledge, in the context of cooperative partial differential operators, it is the first time such a result is 
stated and proved.

Recall that a doubly stochastic matrix $\vect{S}\in\R^{N\times N}$ is a nonnegative matrix such that $\vect{S}\veo=\vect{S}^\upT\veo=\veo$.
Denote $\vect{\mathcal{S}}\subset\mathcal{L}^\infty_\upp(\R\times\R^n,\R^{N\times N})$ the set of all periodic functions whose values are doubly 
stochastic matrices almost everywhere and $\vect{\mathcal{S}}_{\{0,1\}}$ the restriction to functions valued in the set of permutation matrices
almost everywhere.

A decomposition $\veL=\diag(\vect{r})+(\vect{S}-\vect{I})\diag(\vect{\mu})$ of a given 
essentially nonnegative matrix $\veL$ with $\vect{S}$ doubly stochastic and $\vect{\mu}$ nonnegative
exists in many cases (see Lemma \ref{lem:decomposition_L}). Such a decomposition is not unique: replacing $(\vect{S},\vect{\mu})$
by $(\vect{I}+\gamma(\vect{S}-\vect{I}),\gamma^{-1}\vect{\mu})$ with a small $\gamma>0$ gives
another decomposition. 
The main property of this decomposition is that the so-called mutation part 
$(\vect{S}-\vect{I})\diag(\vect{\mu})$ admits $\veo$ as left Perron--Frobenius eigenvector, 
with eigenvalue $0$. In other words, summing the lines of the system makes the mutations 
disappear: if the phenotypes do not differ in intrinsic growth rate (all $r_i$ coincide), 
then the phenotype distribution has no effect on the growth of the meta-population 
$\sum_{i=1}^N u_i$. This is indeed under this form that $\veL$ appears in several papers on
reaction--diffusion models for phenotypically structured populations 
\cite{Cantrell_Cosner_Yu_2018,Morris_Borger_Crooks,Griette_Raoul}. 

\begin{thm}\label{thm:optim_doubly_stochastic}
Assume $\veL$ has the form 
\[
    \veL=\diag(\vect{r})+(\vect{S}-\vect{I})\diag(\vect{\mu})
\]
with $\vect{S}\in\vect{\mathcal{S}}$, $\vect{r}\in\mathcal{L}^\infty_\upp(\R\times\R^n,\R^N)$ and
$\vect{\mu}\in\mathcal{L}^\infty_\upp(\R\times\R^n,[\vez,\vei))$.

Then, for all $z\in\R^n$,
\[
    \min_{\vect{S}\in\vect{\mathcal{S}}_{\{0,1\}}}\lambda_{1,z}(\vect{S})
    =\min_{\vect{S}\in\vect{\mathcal{S}}}\lambda_{1,z}(\vect{S})
    \leq\max_{\vect{S}\in\vect{\mathcal{S}}}\lambda_{1,z}(\vect{S})
    =\max_{\vect{S}\in\vect{\mathcal{S}}_{\{0,1\}}}\lambda_{1,z}(\vect{S}).
\]
\end{thm}

This theorem does not require the assumption \ref{ass:irreducible} (which is not satisfied for some choices of $\vect{S}$; in such
cases, the generalized principal eigenvalues $\lambda_{1,z}$ are defined by continuous extension, see
Theorem \ref{thm:continuity_eigenvalue_L}). 
In particular, the set of optimal permutation matrices might \textit{a priori} be reduced to the singleton 
$\{\vect{I}\}$. Also, in this theorem, and as usual in optimization problems, we consider $\mathcal{L}^\infty$ constraints on $\vect{S}$ 
instead of H\"{o}lder-continuity constraints; the optimizers might be for instance ``bang-bang'' discontinuous piecewise-constant functions.
Let us also point out that, as explained in Remark \ref{rem:optimization_more_general_decomposition}, the result remains true
with any more general decomposition $\veL=\vect{B}+\vect{S}\vect{A}$ with $\vect{A}$ nonnegative and $\vect{B}$ essentially nonnegative.

The modeling viewpoint on this result is natural and enlightening. Say we want to optimize the chances of, for instance, survival
of a population, and, for simplicity, that the environment is homogeneous; the phenotypes are labelled as follows: $u_1$ is the best 
phenotype when there are no mutations, $u_2$ is the second best phenotype, and so forth. Intuitively we should select a (reducible) 
mutation strategy such that the type $u_1$ is $100\%$ heritable. Thus the first column of $\vect{S}$ should be $\vect{e}_1$. Since 
$\vect{S}$ is doubly stochastic, its first line is then $\vect{e}_1^\upT$, whence the first phenotype is in fact completely isolated 
from the others. Subsequently, whatever the mutation strategy for the phenotypes $u_2$, $u_3$, etc., is, the periodic principal 
eigenvalue is optimal and equal to the periodic principal eigenvalue of the scalar equation satisfied by $u_1$. If $u_2$ is just as 
good as $u_1$, then similarly the pair $\{u_1,u_2\}$ has to be isolated, but apart from this restriction the two blocks 
of $\vect{S}$ can be chosen freely, and in particular they can have the form of permutation matrices. 
The extension of this intuition to spatio-temporally heterogeneous environments explains why the optimal $\vect{S}$ is not in general 
constant; it has to ``switch'' as soon as the optimal family of phenotypes changes.

Let us stress that although the set of doubly stochastic matrices is the convex hull of the set of permutation matrices (a classical result 
known as the Birkhoff--von Neumann theorem), $\vect{S}\in\vect{\mathcal{S}}\mapsto\lambda_{1,z}(\vect{S})$ is not concave (see 
Theorem \ref{thm:concavity_eigenvalue_L}), so that Theorem \ref{thm:optim_doubly_stochastic} does not follow from mere 
convexity considerations. Let us also stress that as soon as all $(\caP_i,r_i)$ coincide with constant $r_i$,
$\vect{S}\mapsto\lambda_{1,z}(\vect{S})$ is constant: maximizers and minimizers need not be in $\vect{\mathcal{S}}_{\{0,1\}}$ and can coincide.

The proof of Theorem \ref{thm:optim_doubly_stochastic} is in fact quite involved and requires the construction of an explicit rank-one perturbation 
of $\vect{S}$.

Our second optimization result, closely related to Theorem \ref{thm:optim_doubly_stochastic}, generalizes a theorem due to Karlin
and later generalized by Altenberg \cite{Karlin_1982,Altenberg_2012} which states that, for any irreducible
stochastic matrix $\vect{S}$ and any diagonal matrix $\vect{D}$ with positive diagonal entries, the mapping
$\tau\in[0,1]\mapsto\lambda_\upPF(((1-\tau)\vect{I}+\tau\vect{S})\vect{D})$ is nonincreasing. 
The Karlin theorem has been interpreted as ``greater mixing yields slower growth'' and shows how, in a space-time homogeneous
setting, mutations reduce the chances of survival. 

\begin{thm}\label{thm:generalized_karlin_time_homogeneous}
Assume $(A_i)_{i\in[N]}$ is independent of $t$, $(q_i)_{i\in[N]}=0$, and $\veL$ has the form
$\veL=\diag(\vect{r})+(\vect{S}-\vect{I})\diag(\vect{\mu})$ with $\vect{r}\in\caC^{\delta/2,\delta}_{\upp}(\R^n,\R^N)$, 
$\vect{\mu}\in\caC^{\delta/2,\delta}_{\upp}(\R^n,(\vez,\vei))$ and $\vect{S}\in\vect{\mathcal{S}}$ all independent of $t$.

For any $\rho>0$, let $\cbQ_\rho$ be the operator with $(A_i)_{i\in[N]}$ and $\veL$ replaced by $(\rho A_i)_{i\in[N]}$
and $\diag(\vect{r})+\rho(\vect{S}-\vect{I})\diag(\vect{\mu})$ respectively.

Then $\rho\in[0,1]\mapsto\lambda_1'(\cbQ_\rho)$ is concave and nondecreasing. Furthermore, if $\vect{r}$ depends on $x$ and $s>0$, 
then it is strictly concave and increasing.
\end{thm}

Consequently, $\lambda_1'$ is maximized at $\rho=1$ and minimized at $\rho=0$: ``greater mutation+diffusion 
yields slower growth''.

We emphasize that the main interest of Theorem \ref{thm:generalized_karlin_time_homogeneous} is that it does not require the symmetry of the mutation matrix 
$(\vect{S}-\vect{I})\diag(\vect{\mu})$. 
When it is symmetric, the variational formula of Theorem \ref{thm:lambdaz_time_homogeneous} can be used to deduce a stronger result, the concavity and monotonicity with respect to the diffusion rate on one hand and to the mutation rate on the other hand, with no need to couple the two rates.

Our last optimization result deals with the spatial distribution in the matrix $\veL$ in one dimension of space. 
In this context, the spatial periodicity cell is then the interval $(0,L_1)$. Our result is a generalization of a result 
by Nadin \cite{Nadin_2007} and makes use of the periodic rearrangement. We recall that for any scalar $L_1$-periodic function $u$ 
there exists a unique $L_1$-periodic function $u^\dagger$ whose restriction to $[0,L_1]$ is symmetric (with respect to the 
midpoint $L_1/2$) and non-increasing in $[L_1/2,L_1]$ and that has the same distribution function as $u$. 
The distribution function of $u$ is:
\[
    \mu_{u}:t\mapsto\left|\{u\geq t\}\cap[0,L_1]\right|.
\]
For a time dependent scalar function $u$, $u^\dagger$ stands for the function rearranged, at every $t$, with respect to $x$. For a time-dependent, vector (respectively matrix)  valued function $\vect{u}$, the notation $\vect{u}^\dagger$ is understood as the vector-valued function with $i$-th (resp. $(i,j)$-th) component $u_i^\dagger$ (resp. $u_{i,j}^\dagger$).

\begin{thm}\label{thm:optimization_spatial_distribution}
Assume $n=1$ and $\dcbP=\partial_t-\vect{D}\Delta$ for some diagonal matrix $\vect{D}$ with constant, positive diagonal entries.

Then
\[
\lambda_{1,\upp}(\cbQ)\geq \lambda_{1,\upp}(\dcbP-\veL^\dagger)
\]
where $\veL^\dagger$ is the entry-wise periodic rearrangement of $\veL$.
\end{thm}

Note that this theorem optimizes the distribution of each $l_{i,j}$ but does not optimize the distribution of mass in the matrix $\veL$.
Brenier \cite{Brenier_1991} showed that the rearrangement of a function of $x$ and the polar decomposition of an invertible matrix in 
$\R^{N\times N}$ are related notions, via the relations $\mu(x)=\mu^\#(u(x))$ ($\mu^\#$ is the spatial rearrangement, $u$ is unitary)
and $\vect{M}=\vect{R}\vect{U}$ ($\vect{R}=(\vect{M}\vect{M}^\upT)^\frac12$ is symmetric positive definite, $\vect{U}$ is orthogonal).
In particular, it is well-known that, similarly to $\lambda_{1,\upp}(-\Delta-\mu^\#)\leq\lambda_{1,\upp}(-\Delta-\mu)$,
any essentially nonnegative matrix $\vect{M}\in\R^{N\times N}$ satisfies 
$\lambda_{\upPF}(\vect{M})\leq \lambda_{\max}((\vect{M}\vect{M}^\upT)^\frac12)$, where
$\lambda_{\max}$ denotes the maximal eigenvalue of a real symmetric matrix. In other words, if $\cbQ$ has only constant coefficients,
\[
    \lambda_{1,\upp}(\dcbP-(\veL^\upT\veL)^\frac12)\leq\lambda_{1,\upp}(\dcbP-\veL),
\]
where the periodic principal eigenvalue $\lambda_{1,\upp}$ on the left-hand side is defined via the spectral theorem for self-adjoint 
compact operators instead of via the Krein--Rutman theorem -- the matrix $(\veL^\upT\veL)^\frac12$ is not, in general, essentially nonnegative.
However the proofs of $\lambda_{1,\upp}(-\Delta-\mu^\#)\leq\lambda_{1,\upp}(-\Delta-\mu)$ and of 
$\lambda_{\upPF}(\vect{M})\leq \lambda_{\max}((\vect{M}\vect{M}^\upT)^\frac12)$ differ strongly. The first one typically uses
the Hardy--Littlewood inequality, which is false for matrices as showed by Brenier \cite{Brenier_1991}. Therefore
it seems that optimizing $\veL$ in the spatial sense and in the matrix sense simultaneously is much more difficult and 
we leave it as a very interesting open problem.

\subsection{Extension to systems with a coupling default}\label{sec:extension_coupling_default}

Theorem \ref{thm:continuity_eigenvalue_L} shows how results on fully coupled cooperative systems (and especially the 
results of Subsections
\ref{sec:theorems_existence_characterization}--\ref{sec:theorems_optimization}) can be applied to more general cooperative systems, that
need not satisfy \ref{ass:irreducible}, by understanding them as networks of fully coupled subsystems. It also shows that such a perspective
is limited regarding $\lambda_1$, as we are now going to explain.

Recall that the Perron--Frobenius eigenvalue $\lambda_\upPF$ can be understood as the restriction to the set of irreducible
essentially nonnegative matrices of the dominant eigenvalue, which is a well-defined continuous mapping from the set of 
essentially nonnegative matrices to $\R$.
Therefore it is natural to suggest the following extension of the generalized principal eigenvalues $\lambda_{1,z}$ and $\lambda_1(\Omega)$:
\begin{equation*}
    \lambda_{1,z}(\cbQ)=\min_{k\in[N']}\lambda_{1,z}(\cbQ_k)=\min_{k\in[N']}\lambda_{1,\upp}(\upe_{-z}\cbQ_{k}\upe_z),
\end{equation*}
\begin{equation*}
    \lambda_1(\cbQ,\Omega)=\min_{k\in[N']}\lambda_1(\cbQ_k,\Omega),
\end{equation*}
where $\cbQ_k$ denotes as in the statement of Theorem \ref{thm:continuity_eigenvalue_L} the $k$-th fully coupled block of
$\cbQ=\dcbP-\veL^\triangle$ in block upper triangular form. With these definitions, Theorem \ref{thm:continuity_eigenvalue_L} shows that
the extension of each $\lambda_{1,z}$, and in particular that of $\lambda_1'$, is continuous. However, as explained in Remark
\ref{rem:counter-example_lambda1_reducible}, the inequality 
\[
    \lim_{\veL\to\veL^\triangle}\lambda_1(\cbQ)=\max_{z\in\R^n}\min_{k\in[N']}\lambda_{1,z}(\cbQ_k)\leq\min_{k\in[N']}\lambda_1(\cbQ_k)
\]
is in some cases strict: the extension of $\lambda_1$ suggested above is not lower semi-continuous, and \textit{a fortiori} not continuous. 

It might be tempting to think that this discontinuity is caused by a wrong choice of generalized definition, and that the correct
choice should be continuous. For instance, defining $\lambda_1$ as $\max_{z\in\R^n}\lambda_{1,z}$ would give a continuous
extension to systems with a coupling default. In view of the literature 
\cite{Berestycki_Nir,Berestycki_Ros_1,Arapostathis_Biswas_Pradhan_2020,Nadin_2007}, 
it is also natural to consider the original definition \eqref{def:lambda1} of $\lambda_1$, and since the coupling default induces a
weaker maximum principle, it is also natural to consider a relaxed definition with nonnegative nonzero super-solutions instead 
of positive super-solutions. In order to compare these quantities, let us denote them as follows:
\[
    \lambda_1^0=\min_{k\in[N']}\lambda_1(\cbQ_k),
\]
\[
    \lambda_1^1=\max_{z\in\R^n}\lambda_{1,z},
\]
\[
    \lambda_1^2=\sup\left\{ \lambda\in\R\ |\ \exists\veu\in\caC^{1,2}_{t-\upp}(\R\times\R^n,(\vez,\vei))\ \cbQ\veu\geq\lambda\veu \right\},
\]
\[
    \lambda_1^3=\sup\left\{ \lambda\in\R\ |\ \exists\veu\in\caC^{1,2}_{t-\upp}(\R\times\R^n),[\vez,\vei)),
    \ \veu\neq\vez,\ \cbQ\veu\geq\lambda\veu \right\}.
\]
Then we can show\footnote{The proof is voluntarily not detailed, for the sake of brevity.} that
\[
    \lambda_1^1\leq\lambda_1^0\leq\max_{k\in[N']}\lambda_1(\cbQ_k)=\lambda_1^3,\quad\lambda_1^2\leq\lambda_1^0.
\]
The inequality $\lambda_1^0\leq\lambda_1^3$ is strict as soon as two $\lambda_1(\cbQ_k)$ differ. 
The equality $\lambda_1^0=\lambda_1^2$ can be verified if $\cbQ$ is block diagonal; although the proof seems to require some work,
we believe that it remains true even if $\cbQ$ is not block diagonal. In any case, since the counter-example of Remark
\ref{rem:counter-example_lambda1_reducible} is block diagonal, there are block diagonal operators $\cbQ$ such that
$\lambda_1^1<\lambda_1^0=\lambda_1^2<\lambda_1^3$. This shows that reasonable definitions of $\lambda_1$ other than 
$\lambda_1^1$ cannot be continuous as \ref{ass:irreducible} ceases to be true. 

Let us point out that $\lambda_1^1$ is indisputably the least natural definition.
In particular, having in mind that $\lambda_1<0$ should be a criterion for population growth (see Subsection \ref{sec:relation_with_semilinear} 
below), then the natural definitions would be either $\lambda_1^0$ (growth of at least one population) or $\lambda_1^3$ (growth of all 
populations). In both cases, the default of lower semi-continuity means that populations with vanishingly small couplings might have 
much stronger chances than decoupled 
populations. This has strong implications for modeling, as simplifying a vanishingly coupled model into a decoupled one is often tempting.
It has been related to the emergence in eco-evolutionary models of unexpectedly large spreading speeds in the vanishing mutation limit. 
We refer to Elliott--Cornell \cite{Elliott_Cornel} for the first formal calculations and to Morris--B\"{o}rger--Crooks 
\cite{Morris_Borger_Crooks} for the rigorous analysis. 

\subsection{Relation with KPP-type semilinear systems}\label{sec:relation_with_semilinear}

In the scalar framework of KPP-type reaction--diffusion equations, $\lambda_1<0$ implies the locally uniform convergence of
all solutions to the unique periodic and uniformly positive entire solution, whereas $\lambda_1'\geq 0$ implies the 
uniform convergence of all solutions to $0$, as proved by Nadin \cite{Nadin_2010}.
The study of entire solutions is much more delicate in the multidimensional setting, simply due to topological freedom
\cite{Morris_Borger_Crooks,Girardin_2017,Girardin_2018,Girardin_Griette_2020}, and their uniqueness and stability 
properties cannot in general be inferred from the linearization at $\vez$. 
However, we will show in a sequel \cite{Girardin_2023} that in the multidimensional case, the results of 
Nadin \cite{Nadin_2010} can be generalized in the following weak form: $\lambda_1<0$ implies the locally uniform 
persistence of all solutions and the existence of a periodic and uniformly positive entire solution, whereas 
$\lambda_1'\geq 0$ implies the uniform convergence of all solutions to $\vez$.

Going toward these results is one of our main motivations for the present work, the other one being the future construction of
pulsating traveling waves \cite{Nadin_2009}.

\section{Preliminaries}

Many of our proofs will use a strong maximum principle and a Harnack inequality for parabolic cooperative systems.
These already exist in the literature under slightly different forms (we refer for instance to 
\cite{Foldes_Polacik_2009,Protter_Weinberger,Bai_He_2020} or to
\cite{Alziary_Fleckinger_Lecureux,Araposthathis_,Chen_Zhao_97,Sweers_1992,Birindelli_Mitidieiri_Sweers,Figueiredo_1994,Figueiredo_Mit} 
for the elliptic case).
For the sake of self-containment and because the parabolic Harnack inequality in \cite{Foldes_Polacik_2009} is insufficient for our purposes, 
in this section, we state or prove what we need afterward. 

\subsection{Strong maximum principle}\label{sec:maximum_principle}
The strong maximum principle for time periodic nonnegative solutions of $\cbQ\veu+K\veu=\vez$ with large $K>0$ (actually, $K>\lambda_1$) 
is established as a side result of the preparation of the application of the Krein--Rutman theorem, just as in Bai--He \cite{Bai_He_2020}. 
In fact, we can repeat the argument of \cite[p. 9882]{Bai_He_2020} to obtain the strong maximum principle for
all values of $K\in\R$, including $K=0$ (large values of $K$ are required only for the inversion of the operator), and for super-solutions
that might not be time periodic but are well-defined in a sufficiently distant past. For clarity, we state this version of the 
strong maximum principle below.

\begin{prop}[Strong maximum principle]
\label{prop:strong_maximum_principle}
Let $\veu\in\caC^{1,2}((0,+\infty)\times\R^n,[\vez,\vei))\cap\caC([0,+\infty)\times\R^N)$ such that $\cbQ\veu\geq\vez$ 
in $(0,+\infty)\times\R^N$. 

If there exist $t^\star> T$, $x^\star\in\R^n$ and $i^\star\in[N]$ such that $u_{i^\star}(t^\star,x^\star)=0$, then $\veu=\vez$ in $[0,+\infty)\times\R^N$.
\end{prop}

A similar property is satisfied in bounded domains. For the sake of simplicity, we only consider smooth boundaries.

\begin{prop}[Strong maximum principle in bounded domains]
\label{prop:strong_maximum_principle_bounded_domains}
Let $\Omega\subset\R^n$ be a nonempty smooth bounded open connected set and 
$\veu\in\caC^{1,2}((0,+\infty)\times\Omega,[\vez,\vei))\cap\caC^{0,1}([0,+\infty)\times\overline{\Omega})$ such that $\cbQ\veu\geq\vez$ 
in $(0,+\infty)\times\Omega$. 

Assume that there exists $x_0\in\Omega$ such that $[x_0,x_0+L]\subset\Omega$.

If there exist $t^\star> T$, $x^\star\in\Omega$ and $i^\star\in[N]$ such that $u_{i^\star}(t^\star,x^\star)=0$, then $\veu=\vez$ in
$[0,+\infty)\times\overline{\Omega}$.

If there exist $t^\star>T$, $x^\star\in\partial\Omega$ and $i^\star\in[N]$ such that $u_{i^\star}(i^\star,x^\star)=\nu(x^\star)\cdot\nabla u_{i^\star}(t^\star,x^\star)=0$, where $\nu(x^\star)\in\R^n$ is the outward pointing unit normal vector,
then $\veu=\vez$ in $[0,+\infty)\times\overline{\Omega}$.
\end{prop}

These versions of the strong maximum principle exploit the full coupling assumption \ref{ass:irreducible}:
if one component of $\veu$ is zero, then so are the others. Nonnegative super-solutions are either zero or positive. 
Without \ref{ass:irreducible}, this alternative is false in general; we refer, for weaker statements applicable to general 
cooperative systems, to the celebrated book by Protter and Weinberger \cite[Chapter 3, Section 8]{Protter_Weinberger}.

\subsection{Harnack inequality}
In this section, we denote by $\sigma>0$ the smallest positive entry of $\overline{\veL}$ and by $K\geq 1$ the smallest positive number such that
\[
    K^{-1}\leq\min_{i\in[N]}\min_{y\in\Sn}\min_{(t,x)\in\clOmper}\left(y\cdot A_i(t,x)y\right),
\]
\[
    \max_{i\in[N]}\max_{y\in\Sn}\max_{(t,x)\in\clOmper}\left(y\cdot A_i(t,x)y\right)\leq K,
\]
\[
    \max_{i\in[N]}\max_{\alpha\in[n]}\max_{(t,x)\in\clOmper}|q_{i,\alpha}(t,x)|\leq K,
\]
\[
    \max_{i,j\in[N]}\sup_{(t,x)\in\clOmper}|l_{i,j}(t,x)|\leq K.
\]

Applying \FPH inequality \cite[Theorem 3.9]{Foldes_Polacik_2009} to the operator $\cbQ$, we obtain the following property.

\begin{prop}
Let $\theta>0$. Assume the irreducibility of the matrix 
\[
    \underline{\veL}=\left(\displaystyle\min_{(t,x)\in\clOmper}l_{i,j}(t,x)\right)_{(i,j)\in[N]^2}
\]
and denote $\eta>0$ its smallest positive entry.

There exists a constant $\overline{\kappa}_{\theta,\eta}>0$, determined only by $n$, $N$, $\eta$, $K$ and the parameter 
$\theta$ such that, if $\veu\in\caC([-2\theta,6\theta]\times[-\frac{3\theta}{2},\frac{3\theta}{2}]^n,[\vez,\vei))$ is a solution of 
$\cbQ\veu=\vez$, then
\[
    \min_{i\in[N]}\min_{(t,x)\in[5\theta,6\theta]\times[-\frac{\theta}{2},\frac{\theta}{2}]^n}u_i(t,x)
    \geq \overline{\kappa}_{\theta,\eta}\max_{i\in[N]}\max_{(t,x)\in[0,2\theta]\times[-\frac{\theta}{2},\frac{\theta}{2}]^n}u_i(t,x).
\]
\end{prop}

However, our irreducibility assumption \ref{ass:irreducible} is concerned with the matrix 
\[
    \overline{\veL}=\left( \max_{(t,x)\in\clOmper}l_{i,j}(t,x)\right)_{(i,j)\in[N]^2}
\]
and not with $\underline{\veL}$. By continuity and essential nonnegativity, $\overline{\veL}$ is irreducible if and only if
\[
    \left(T|[0,L]|\right)^{-1}\int_{\Omega_\upp}\veL(t,x)\upd t\upd x
\]
is itself irreducible. Hence we can understand the assumption \ref{ass:irreducible} as ``$\veL(t,x)$ is irreducible on average''.
It is known that such an assumption is sufficient, and in some sense necessary, for full coupling of the parabolic or elliptic 
operator; refer, for instance, to \cite{Bai_He_2020,Araposthathis_,Sweers_1992,Birindelli_Mitidieiri_Sweers}.

Since \FPH inequality requires the pointwise irreducibility of $\veL$, 
which is a much stronger assumption than the irreducibility on average (there are simple examples of matrices that are irreducible 
on average but reducible pointwise at all $(t,x)$, see for instance Remark \ref{rem:counter-example_average_eigenvalue}), 
it is not satisfying for our purposes. Actually, going through the proof of \cite[Theorem 3.9]{Foldes_Polacik_2009}, it appears that 
its adaptation to our setting is not straightforward, as F\"{o}ldes and Pol\'{a}\v{c}ik overcome the key obstacle by 
constructing a nonnegative nonzero sub-solution smaller than $\eta$ multiplied by some positive constant. Nevertheless, since 
\ref{ass:irreducible} is known to be the optimal assumption for full coupling, it is natural to expect a similar Harnack inequality 
to hold, provided the parabolic cylinder under consideration is sufficiently larger than the periodicity cell $\Omega_\upp$. 
This is what we prove below, drawing inspiration from the elliptic case studied in Araposthathis--Ghosh--Marcus \cite{Araposthathis_}.

By convenience for future use, we state the result for a zeroth order, diagonal, non-necessarily periodic perturbation of $\cbQ$. The diffusion 
and advection terms can be perturbed similarly if needed.

\begin{prop}[Fully coupled Harnack inequality]\label{prop:harnack_inequality}
Let $\theta\geq\max\left(T,L_1,\dots,L_n\right)$ and $\vect{f}\in\mathcal{L}^\infty\cap\caC^{\delta/2,\delta}(\R\times\R^n,\R^N)$ 
with $\delta\in(0,1)$. Let $F>0$ such that  
\[
    \max_{i\in[N]}\sup_{(t,x)\in\R\times\R^n}|f_i(t,x)|\leq F.
\]

There exists a constant $\overline{\kappa}_{\theta,F}>0$, determined only by $n$, $N$, $\sigma$, $K$ and the parameters
$\theta$ and $F$ such that, if $\veu\in\caC([-2\theta,6\theta]\times[-\frac{3\theta}{2},\frac{3\theta}{2}]^n,[\vez,\vei))$ 
is a solution of $\cbQ\veu=\diag(\vect{f})\veu$, then
\[
    \min_{i\in[N]}\min_{(t,x)\in[5\theta,6\theta]\times[-\frac{\theta}{2},\frac{\theta}{2}]^n}u_i(t,x)
    \geq \overline{\kappa}_{\theta,F}\max_{i\in[N]}\max_{(t,x)\in[0,2\theta]\times[-\frac{\theta}{2},\frac{\theta}{2}]^n}u_i(t,x).
\]
\end{prop}

\begin{proof}
Define, for all $i\in[N]$, the $n+1$-dimensional hypercube
\[
Q_i=\left(5\theta-\frac{\theta}{2^{i-1}},6\theta\right)\times\left(-\frac{\theta}{2}-\frac{\theta}{2^{i}},\frac{\theta}{2}+\frac{\theta}{2^{i}}\right)^n\subset\R\times\R^n.
\]
Note the series of compact inclusions
\[
Q_1 = (4\theta,6\theta)\times(-\theta,\theta)^n\supset Q_2\supset \dots\supset
Q_N\supset(5\theta,6\theta)\times\left(-\frac{\theta}{2},\frac{\theta}{2}\right)^n.
\]

Following carefully the proof of \FPH inequality \cite{Foldes_Polacik_2009}, we observe that we only have to prove the following claim.

\textbf{Claim 1:} let $k\in[N-1]$. If there exists $I\subset[N]$ of cardinal $k$ and a positive constant $\kappa_k$ determined only by $k$, $n$, $N$, 
$\sigma$, $K$, $\theta$ and $F$, such that, for all $j\in I$,
\[
    \min_{(t,x)\in\overline{Q_k}}u_j(t,x)
    \geq \kappa_k \max_{(t,x)\in[0,2\theta]\times[-\frac{\theta}{2},\frac{\theta}{2}]^n}u_1(t,x),
\]
then there exists $i\in[N]\backslash I$ and a positive constant $\kappa_{k+1}\leq \kappa_k$ determined only by $k$, $n$, $N$, $\sigma$, 
$K$, $\theta$ and $F$, such that
\[
    \min_{(t,x)\in\overline{Q_{k+1}}}u_i(t,x)
    \geq \kappa_{k+1} \max_{(t,x)\in[0,2\theta]\times[-\frac{\theta}{2},\frac{\theta}{2}]^n}u_1(t,x).
\]

We prove first the following simpler claim, inspired by \cite[Lemma 3.6]{Araposthathis_}.

\textbf{Claim 2:} Let $k\in[N-1]$, $i\in[N]$ and $g\in\caC(\overline{Q_k},[0,+\infty))$. 
There exists a positive constant $C_k$ determined only by $k$, $n$, $K$, $\theta$ and $F$, such that,
if $u$ is a solution of $\caP_i u -l_{i,i}u-f_i u = g$ in $Q_k$ with $u=0$ on the parabolic boundary $\partial_P Q_k$, then
\[
    \min_{(t,x)\in\overline{Q_{k+1}}}u(t,x)\geq C_k\max_{(t,x)\in\overline{Q_k}}g(t,x).
\]

\begin{proof}[Proof of Claim 2]
When $g=0$, $u=0$ as well and the result is obvious (with, say, $C_k=1$). Therefore we assume without loss of generality that $g>0$.

Up to dividing $u$ by $\max_{\overline{Q_k}}g$, we assume without loss of generality $\max_{\overline{Q_k}}g=1$. Since the solution
$u$ of the Cauchy--Dirichlet problem with zero data on the parabolic boundary is unique, we only have to prove that this solution has
a positive minimum in $\overline{Q_{k+1}}$, and that the infimum of these minima, when $\cbQ$, $f$ and $g$ vary in the correct class, is 
still positive.

The nonnegativity of $u$ is a direct consequence of the (weak) maximum principle. The positivity of its minimum in $\overline{Q_{k+1}}$
is a consequence of the strong maximum principle and the fact that $0$ cannot be the solution.

Now, define $\mathcal{U}_k$ as the set of all $U=\left(A,q,l,f,g\right)$ such that
\[
    A\in\caC^{\delta/2,1+\delta}(\overline{Q_k},\R^{n\times n}),
\]
\[
    (q,l,f,g)\in\caC^{\delta/2,\delta}(\overline{Q_k},[-K,K]^n\times[-K,K]\times[-F,F]\times[0,1]),
\]
such that $A=A^\upT$, $\max_{\overline{Q_k}}g=1$ and
\[
    K^{-1}\leq\min_{y\in\Sn}\min_{(t,x)\in\overline{Q_k}}\left(y\cdot A(t,x)y\right)\leq \max_{y\in\Sn}\max_{(t,x)\in\overline{Q_k}}\left(y\cdot A(t,x)y\right)\leq K.
\]
For all $U\in\mathcal{U}_k$, denote $u_U$ the solution of 
\[
    \begin{cases}
        \partial_t u -\nabla\cdot(A\nabla u)+q\cdot\nabla u -lu-fu=g & \text{in }Q_k, \\
        u = 0 & \text{on }\partial_P Q_k,
    \end{cases}
\]
and denote $m(U)=\min_{\overline{Q_{k+1}}}u_U>0$. Let us verify that $\inf_{U\in\mathcal{U}_k}m(U)>0$.

Assume by contradiction $\inf_{U\in\mathcal{U}_k}m(U)=0$. Then there exists a minimizing sequence $(U_p)_{p\in\N}$ such that $m(U_p)\to 0$ as 
$p\to+\infty$. By classical compactness and regularity estimates \cite{Lieberman_2005}, up to extraction, $(U_p)$ converges uniformly to a limit $U_\infty=(A_\infty,q_\infty,l_\infty,f_\infty,g_\infty)$ such that
\[
    A_\infty\in\caC^{0,1}(\overline{Q_k},\R^{n\times n}),
\]
\[
    (q_\infty,l_\infty,f_\infty,g_\infty)\in\caC(\overline{Q_k},\R^{n\times n}\times[-K,K]^n\times[-K,K]\times[-F,F]\times[0,1]),
\]
such that $A_\infty=A_\infty^\upT$, $\max_{\overline{Q_k}}g_\infty=1$ and
\[
    K^{-1}\leq\min_{y\in\Sn}\min_{(t,x)\in\overline{Q_k}}\left(y\cdot A_\infty(t,x)y\right)\leq \max_{y\in\Sn}\max_{(t,x)\in\overline{Q_k}}\left(y\cdot A_\infty(t,x)y\right)\leq K
\]
and $(u_p)_{p\in\N}=\left(u_{U_p}\right)_{p\in\N}$ converges uniformly to the solution $u_\infty$ of
\[
    \begin{cases}
        \partial_t u -\nabla\cdot(A_\infty\nabla u)+q_\infty\cdot\nabla u -l_\infty u-f_\infty u=g_\infty & \text{in }Q_k, \\
        u = 0 & \text{on }\partial_P Q_k.
    \end{cases}
\]
Moreover, by definition of $(U_p)$, $m(U_\infty)=\lim_{p\to+\infty}m(U_p)=0$. But then the strong maximum principle yields $u_\infty=0$,
and this contradicts $g_\infty>0$. Hence $\inf_{U\in\mathcal{U}_k}m(U)>0$ and Claim 2 is proved with $C_k=\inf_{U\in\mathcal{U}_k}m(U)>0$.
\end{proof}

\begin{proof}[Proof of Claim 1]
Let $k\in[N-1]$, $I\subset[N]$ of cardinal $k$, 
\[
    M=\max_{(t,x)\in[0,2\theta]\times[-\frac{\theta}{2},\frac{\theta}{2}]^n}u_1(t,x),
\]
and assume that for all $j\in I$,
\[
    \min_{(t,x)\in\overline{Q_k}}u_j(t,x)\geq \kappa_k M.
\]
By \ref{ass:irreducible}, there exists $i\in[N]\backslash I$ and $j\in I$ such that $\max_{\clOmper}l_{i,j}>0$.

Let $\underline{u}$ be the solution of $\caP_i \underline{u}-l_{i,i}\underline{u}-f_i\underline{u}=\kappa_k M l_{i,j}$ in $Q_k$, 
$\underline{u}=0$ on $\partial_P Q_k$. Since $Q_k$ contains a translation of $\clOmper$ and $l_{i,j}$ is periodic, applying Claim 2, we get:
\[
    \min_{\overline{Q_{k+1}}}\underline{u}\geq C_k \kappa_k M \max_{\clOmper}l_{i,j}.
\]

Moreover, in $Q_k$,
\[
    \caP_i u_i-l_{i,i}u_i-f_i u_i = \sum_{k\in[N]\backslash\{i\}}l_{i,k}u_k\geq l_{i,j}u_j\geq l_{i,j}\min_{(t,x)\in\overline{Q}_k}u_j\geq \kappa_k M l_{i,j}.
\]
Also, on the parabolic boundary $\partial_P Q_k$, $u_i\gg 0=\underline{u}$. Therefore, by virtue of the comparison principle, 
$u_i\geq\underline{u}$ in $Q_k$, and subsequently, using the definition of $\sigma$,
\[
    \min_{\overline{Q_{k+1}}}u_i\geq \min_{\overline{Q_{k+1}}}\underline{u}\geq C_k \kappa_k M \max_{\clOmper}l_{i,j}\geq C_k\kappa_kM\sigma.
\]
Setting $\kappa_{k+1}=C_k\kappa_k\sigma$, we have proved Claim 1.
\end{proof}

This ends the proof.
\end{proof}

\begin{rem}
As an immediate corollary, if $\veu$ is time periodic, then $\overline{\kappa}_{\theta,F}<1$ and
\[
    \min_{i\in[N]}\min_{(t,x)\in\R\times[-\frac{\theta}{2},\frac{\theta}{2}]^n}u_i(t,x)
    \geq \overline{\kappa}_{\theta,F}\max_{i\in[N]}\max_{(t,x)\in\R\times[-\frac{\theta}{2},\frac{\theta}{2}]^n}u_i(t,x).
\]
If $\veu$ is space-time periodic, then an even stronger estimate holds:
\[
    \min_{i\in[N]}\min_{(t,x)\in\R\times\R^n}u_i(t,x)\geq \overline{\kappa}_{\theta,F}\max_{i\in[N]}\max_{(t,x)\in\R\times\R^n}u_i(t,x).
\]
\end{rem}

\section{Proofs}

\subsection{Existence, characterization and concavity: proof of Theorems \ref{thm:existence_characterization_Rn}--\ref{thm:concavity_eigenvalue_L}}

The main result of this subsection is Theorem \ref{thm:existence_characterization_Rn}. It is actually
a consequence of Theorems \ref{thm:existence_characterization_Omega} and of a concavity result on $z\mapsto\lambda_{1,z}$,
Corollary \ref{cor:lambdaz_strictly_concave}, that will follow from a more general concavity result,
Proposition \ref{prop:concavity_eigenvalue_z_L}, that also contains \ref{thm:concavity_eigenvalue_L}.

Most proofs in this subsection are direct adaptations to the vector case
of the proofs by Nadin \cite{Nadin_2007}, written here for the paper to be self-contained.
The only proofs whose adaptations truly require some care are those of Propositions
\ref{prop:concavity_eigenvalue_z_L} and \ref{prop:existence_eigenfunction_exp_times_per}.

\subsubsection{The generalized principal eigenvalue $\lambda_1$ in arbitrary domains: proof of Theorem \ref{thm:existence_characterization_Omega}}

\begin{prop}\label{prop:eigenvalue_well_defined}
    Let $\Omega\subset\R^n$ be a nonempty open connected set such that there exists $x_0\in\Omega$ satisfying $[x_0,x_0+L]\subset\Omega$. 
    Then the generalized principal eigenvalue $\lambda_1(\Omega)\in\R$ is well-defined. 
    
    Furthermore, if $\partial\Omega$ is bounded and smooth, then $\lambda_1(\Omega)=\lambda_{1,\upDir}(\Omega)$.
\end{prop}

\begin{proof}
We begin with the case of bounded smooth domains. The inequality $\lambda_{1,\upDir}(\Omega)\leq\lambda_1(\Omega)$ follows 
by using the Dirichlet principal eigenfunction as test function in the definition of $\lambda_1(\Omega)$. 
The converse inequality is proved by contradiction: assume that
$\lambda_{1,\upDir}(\Omega)<\lambda_1(\Omega)$. Then there exists $\mu\in(\lambda_{1,\upDir}(\Omega),\lambda_1(\Omega))$ and
$\veu\in\caC^{1,2}_{t-\upp}(\R\times\Omega,(\vez,\vei))\cap\caC^{1}(\R\times\overline{\Omega})$ such that $\cbQ\veu\geq\mu\veu$.
By boundedness of the Dirichlet principal eigenfunction $\vev$, the quantity
\[
    \kappa^\star=\inf\left\{ \kappa>0\ |\ \kappa\veu-\vev\gg\vez \right\}
\]
is well-defined in $\R$. The function $\vew=\kappa^\star\veu-\vev$ satisfies 
\[
    \cbQ\vew=\kappa^\star\mu\veu-\lambda_{1,\upDir}(\Omega)\vev\gg\lambda_{1,\upDir}(\Omega)\vew\quad\text{in }\R\times\Omega,
\]
\[
    \vew\geq\vez\quad\text{in }\R\times\overline{\Omega},
\]
\[
    \vew\geq\vez\quad\text{on }\R\times\partial\Omega,
\]
and there exists $(i^\star,t^\star,x^\star)\in[N]\times[0,T]\times\overline{\Omega}$ such that $w_{i^\star}(t^\star,x^\star)=0$. 
If $x^\star\in\Omega$, then by virtue of the strong maximum principle (see Proposition \ref{prop:strong_maximum_principle_bounded_domains}), 
$\vew$ is the zero function, which contradicts $\mu>\lambda_{1,\upDir}(\Omega)$. Hence $\vew\gg\vez$ in $\R\times\Omega$. Since
$\veu\in\caC^1(\R\times\overline{\Omega})$, the normal derivative of $\vew$ at any point $(t,x)\in\R\times\partial\Omega$ is well-defined.
The optimality of $\kappa^\star$ implies the existence of $(i',t',x')\in[N]\times[0,T]\times\partial\Omega$ such that
that both $w_{i'}(t',x')$ and the normal derivative of $w_{i'}$ at $(t',x')$ are zero, which contradicts the boundary version 
of the strong maximum principle. Hence $\lambda_1(\Omega)\leq\lambda_{1,\upDir}(\Omega)$. This ends the proof in the case of bounded smooth domains.

Then we turn to general, not necessarily bounded and smooth, domains.
    Let $\nu=-\lambda_{\upPF}(\overline{\veL})\in\R$, where the square matrix $\overline{\veL}$ is defined in
    \ref{ass:irreducible}, and let $\veu\in\R^N$ be a positive Perron--Frobenius eigenvector for
    $\overline{\veL}$, namely $\overline{\veL}\veu=-\nu\veu$. Then clearly $\cbQ\veu\geq\nu\veu$,
    which proves that the set 
    \[
        \left\{ \lambda\in\R\ |\ \exists \veu\in\caC^{1,2}_{t-\upp}(\R\times\Omega,(\vez,\vei))\cap\caC^{1}(\R\times\overline{\Omega})\ \cbQ\veu\geq\lambda\veu \right\}
    \]
    is nonempty. Hence its supremum, $\lambda_1(\Omega)$, is well-defined in $\R\times\{\infty\}$.

    Next, it follows directly from the definition that $\lambda_1(\Omega)\leq\lambda_1(\Omega')$
    for any open set $\Omega'\subset\Omega$. Since $\Omega$ is open and contains a periodicity cell $[x_0,x_0+L]$, it contains a 
    bounded smooth connected open set $\Omega'$ satisfying $[x_0,x_0+L]\subset\Omega'\subset\overline{\Omega'}\subset\Omega$. Therefore 
    \[
        \lambda_1(\Omega)\leq\lambda_1(\Omega')=\lambda_{1,\upDir}(\Omega')<+\infty.
    \]
    This ends the proof.
\end{proof}

\begin{prop}\label{prop:eigenvalue_limit_Dirichlet}
    Let $\Omega\subset\R^n$ be a nonempty open connected set and let $(\Omega_k)_{k\in\N}$ be a sequence of nonempty
    open connected sets such that, for some $x_0\in\Omega$,
    \[
        [x_0,x_0+L]\subset\Omega_1,\quad\Omega_k\subset\Omega_{k+1},\quad\bigcup_{k\in\N}\Omega_k=\Omega.
    \]

    Then $\lambda_1(\Omega_k)\to\lambda_1(\Omega)$ as $k\to+\infty$. 
    
    Furthermore, there exists a generalized principal eigenfunction associated with $\lambda_1(\Omega)$.
\end{prop}
\begin{proof}
    In order to work with bounded and smooth domains, we consider a family 
    $(\widetilde{\Omega}_k)_{k\in\N}$, nondecreasing and convergent to $\Omega$ in the inclusion sense,
    and such that $\overline{\widetilde{\Omega}_k}\subset\Omega_k$ for all $k\in\N$ 
    (with $[x_0,x_0+L]\subset\widetilde{\Omega}_1$, which is always possible since $[x_0,x_0+L]$ is closed and $\Omega_1$ is open). 
    Denote $(\mu_{k})_{k\in\N}=(\lambda_{1,\upDir}(\Omega_k))_{k\in\N}$,
    $(\nu_{k})_{k\in\N}=(\lambda_{1,\upDir}(\widetilde{\Omega}_k))_{k\in\N}$, and note
    that both sequences converge, with limits satisfying 
    \[
        	\lambda_1(\Omega)\leq \lim_{k\to+\infty}\mu_k\leq \lim_{k\to+\infty}\nu_k.
    \]

    Let $\nu=\lim\nu_k$. We now aim to prove that $\nu\leq\lambda_1(\Omega)$ by constructing an
    eigenfunction for the eigenvalue $\nu$ of the operator $\cbQ$ acting on $\caC^{1,2}_{t-\upp}(\R\times\Omega,(\vez,\vei))\cap\caC^1_0(\R\times\overline{\Omega})$. Since such an eigenfunction
    will in fact be a generalized principal eigenfunction for the generalized principal eigenvalue $\lambda_1(\Omega)$, this will
    complete the proof.
    
    Fix $y\in\widetilde{\Omega}_1=\bigcap_{k\in\N}\widetilde{\Omega}_k$ and consider the sequence
    $(\veu_k)_{k\in\N}$ of positive principal eigenfunctions associated with $\nu_k$ and normalized
    by $\max_{i\in[N]}u_{i,k}(0,y)=1$. Extend these eigenfuctions as functions defined in
    $\R\times\Omega$ by setting $\veu_k=\vez$ in $\R\times\Omega\backslash\widetilde{\Omega}_k$.
    
    By virtue of the time periodicity of $\veu_k$, of the normalization at time $t=0$ and of the Harnack inequality of 
    Proposition \ref{prop:harnack_inequality}, 
    the sequence $\left(\|\veu_k\|_{\mathcal{L}^\infty([0,T]\times\widetilde{\Omega}_{k_0})}\right)_{k\in\N,k>k_0}$ 
    is bounded for any $k_0\in\N$. By standard regularity estimates \cite{Lieberman_2005}, $(\veu_k)_{k\in \N}$ converges
    up to a diagonal extraction to a function 
    $\veu_\infty\in\caC^{1,2}_{t-\upp}(\R\times\Omega)$ satisfying
    \[
	\cbQ\veu_\infty=\nu\veu_\infty\quad\text{in }\R\times\Omega.
    \]
    Moreover, $\veu_\infty$ is nonnegative, nonzero at $(t,x)=(0,y)$, and by the maximum principle
    it is therefore positive in $\R\times\Omega$. 
    
    In order to establish $\nu\leq \lambda_1(\Omega)$, it only remains to verify that $\veu_\infty\in\caC^1_0(\R\times\overline{\Omega})$. 
    
    Let 
    \[
        C=|\lambda_1(\widetilde{\Omega}_1)|\sup_{k\in[N]}\|\veu_k\|_{\mathcal{L}^\infty([0,T]\times\widetilde{\Omega}_k)}
    \]
    and define $\hat{\veu}\in\caC^{1,2}_{t-\upp}(\R\times\Omega,(\vez,\vei))\cap\caC^1_0(\R\times\overline{\Omega},[\vez,\vei))$ 
   as the time periodic solution of the following (decoupled) system:
    \[
        \begin{cases}
            \dcbP\hat{\veu}=\veo &\text{in }\R\times\Omega, \\
            \hat{\veu}=\vez &\text{on }\R\times\partial\Omega.
        \end{cases}
    \]
    Then, for any $k\in\N$, 
    \[
        \dcbP((C\hat{\veu}-\veu_k)\geq 
        C\veo-\sup_{k\in[N]}\left(\lambda_1(\widetilde{\Omega}_k)\right)\veu_k= C\veo-\lambda_1(\widetilde{\Omega}_1)\veu_k\geq\vez.
    \]
    This leads to $\veu_k\leq C\hat{\veu}$ for all $k\in[N]$, and then, passing to the limit, $\veu_\infty\leq C\hat{\veu}$ in $\R\times\Omega$.
    Hence $\veu_\infty\in\caC_0(\R\times\overline{\Omega})$. 
    The continuity of its gradient $\nabla\veu_\infty$ on the regular boundary points follows from classical regularity estimates up to the 
    boundary \cite{Lieberman_2005}. This ends the proof.
\end{proof}

\begin{rem}
The proof uses the interior Harnack inequality of Proposition \ref{prop:harnack_inequality}, which, as stated, requires that the domain of definition
contains a translation of $[0,3\theta]^n$, with $\theta\geq\max(T,L_1,\dots,L_n)$. This is not optimal and just for convenience of notation; what 
truly matters for the interior Harnack inequality is that the domain of definition is strictly larger than a closed periodicity cell, as expressed 
in the preceding statement. We leave the necessary correction of the proof of Proposition \ref{prop:harnack_inequality} as an exercise for 
interested readers.
\end{rem}

\begin{prop}\label{prop:max--min_characterization_lambda1}
    Let $\Omega\subset\R^n$ be a nonempty open connected set such that there exists $x_0\in\Omega$ satisfying $[x_0,x_0+L]\subset\Omega$. 
    Then the generalized principal eigenvalue $\lambda_1(\Omega)$ can be characterized as:
    \[
	\lambda_1(\Omega)=\max_{\veu\in\caC^{1,2}_{t-\upp}(\R\times\Omega,(\vez,\vei))\cap\caC^{1}(\R\times\overline{\Omega})}\min_{i\in[N]}\inf_{\R\times\Omega}\left(\frac{(\cbQ\veu)_i}{u_i}\right).
    \]
\end{prop}
\begin{proof}
    Testing $\cbQ$ against a generalized principal eigenfunction (whose existence is guaranteed by 
    Proposition \ref{prop:eigenvalue_limit_Dirichlet}), we directly find 
    \[
        \lambda_1(\Omega)\leq \sup_{\veu\in\caC^{1,2}_{t-\upp}(\R\times\Omega,(\vez,\vei))\cap\caC^{1}(\R\times\overline{\Omega})}\min_{i\in[N]}\inf_{\R\times\Omega}\left(\frac{(\cbQ\veu)_i}{u_i}\right).
    \]

    Next we assume by contradiction that the above inequality is actually strict. Then there exists
    $\mu>\lambda_1(\Omega)$ and a test function $\veu$ such that $\cbQ\veu\geq\mu\veu$. This contradicts the definition of $\lambda_1(\Omega)$. 

    Finally, the existence of a generalized principal eigenfunction shows that the supremum is in fact a maximum, as in the statement.
\end{proof}

\subsubsection{Characterizations of the periodic principal eigenvalues $\lambda_{1,z}$}

For any $z\in\R^n$, the existence and uniqueness of the eigenpair $(\lambda_{1,z},\veu_z)$, up to multiplication of 
the eigenfunction by a constant, follows from the Krein--Rutman theorem. We do not detail the proof of this claim.
Below, we prove a generalization of the classical Collatz--Wielandt formula for Perron--Frobenius eigenvalues.

\begin{prop}\label{prop:max--min_characterization_lambdaz}
Let $z\in\R^n$. Then the periodic principal eigenvalue $\lambda_{1,z}$ can be characterized as:
\begin{equation}
    \lambda_{1,z}=\max_{\veu\in\caC^{1,2}_\upp(\R\times\R^n,(\vez,\vei))}\min_{i\in[N]}\min_{\clOmper}\left(\frac{(\cbQ_z\veu)_i}{u_i}\right),
\end{equation}
\begin{equation}
    \lambda_{1,z}=\min_{\veu\in\caC^{1,2}_\upp(\R\times\R^n,(\vez,\vei))}\max_{i\in[N]}\max_{\clOmper}\left(\frac{(\cbQ_z\veu)_i}{u_i}\right).
\end{equation}
\end{prop}
\begin{proof}
We prove only the max--min characterization, the min--max one being proved quite similarly.

Using the existence of the periodic principal eigenfunction $\veu_z$, we immediately obtain
\[
    \lambda_{1,z}\leq\sup_{\veu\in\caC^{1,2}_\upp(\R\times\R^n,(\vez,\vei))}\min_{i\in[N]}\min_{\clOmper}\left(\frac{(\cbQ_z\veu)_i}{u_i}\right).
\]

Next we assume by contradiction that the above inequality is actually strict. Then there exists a test function 
$\veu\in\caC^{1,2}_\upp(\R\times\R^n,(\vez,\vei))$ and a real number $\mu>\lambda_{1,z}$ such that $\cbQ_z\veu\geq \mu\veu$.
Let 
\[
    \kappa^\star=\inf\left\{ \kappa>0\ |\ \kappa\veu-\veu_z\gg\vez \right\}.
\]
Applying the strong maximum principle to $\kappa^\star\veu-\veu_z$, just as in the proof of Proposition
\ref{prop:eigenvalue_well_defined}, we find a contradiction.

Finally, the existence of $\veu_z$ shows that the supremum is in fact a maximum, as in the statement.
\end{proof}

\subsubsection{Concave dependence on $z$ and $\veL$}

In order to show later on that $\lambda_1=\max_{z\in\R^n}\lambda_z$, we need to establish first the strict concavity of 
$z\mapsto\lambda_z$. This is stated below in Corollary \ref{cor:lambdaz_strictly_concave}.
Since the proof of Theorem \ref{thm:concavity_eigenvalue_L} on the concavity of $\veL\mapsto\lambda_z(\veL)$ is quite similar, 
we prove the two results directly together.

\begin{prop}\label{prop:concavity_eigenvalue_z_L}
Let $z_1,z_2\in\R^n$.

Let 
\[
    \left(\veL[s]\right)_{s\in[0,1]}\in\left(\caC^{\delta/2,\delta}_\upp(\R\times\R^n,\R^{N\times N})\right)^{[0,1]}
\]
a family of matrices satisfying the same assumptions as $\veL$ (\textit{i.e.}, \ref{ass:cooperative}, \ref{ass:irreducible}) and 
such that, for all 
$(t,x)\in\R\times\R^n$ and $i\in[N]$,
\begin{enumerate}
    \item $s\mapsto l_{i,i}[s](t,x)$ is convex;
    \item for all $j\in[N]\backslash\{i\}$, $s\mapsto l_{i,j}[s](t,x)$ is either identically zero or log-convex.
\end{enumerate}

For all $s\in[0,1]$, denote 
\[
    \cbQ[s]=\upe_{-(1-s)z_1-sz_2}(\diag(\caP_i)-\veL[s])\upe_{(1-s)z_1+sz_2}
\]
and $\lambda[s]=\lambda_{1,\upp}(\cbQ[s])$ the associated periodic principal eigenvalue.

Then $s\in[0,1]\mapsto \lambda[s]$ is affine or strictly concave and it is affine if and only if the following conditions are both satisfied:
\begin{enumerate}[label=$(\text{Cond. }\arabic*)$]
    \item \label{cond:equality_case_1} $z_1=z_2$;
    \item \label{cond:equality_case_2} there exist a constant vector $\vect{b}\gg\vez$, a function $\vect{c}\in\caC_\upp(\R\times\R^n,(\vez,\vei))$ 
    and a function $\vect{f}\in\caC_{\upp}(\R,\R^N)$ satisfying $\int_0^T\vect{f}\in\vspan(\veo)$ such that the entries of $\veL$ have the form:
    \begin{equation*}
        l_{i,j}[s]:(t,x)\mapsto
        \begin{cases}
            l_{i,i}[0](t,x)-sf_i(t) & \text{if }i=j, \\
            l_{i,j}[0](t,x)\left(\frac{b_j}{c_i(t,x)}\right)^s \upe^{s\left(\int_0^t f_j-\frac{t}{T}\int_0^T f_j\right)} & \text{if }i\neq j
        \end{cases}
    \end{equation*}
    and such that the function $\vect{c}$ satisfies, at all $(t,x)\in\clOmper$ and for each $i\in[N]$, 
    \[
        c_i(t,x)=b_i\upe^{\int_0^t f_i-\frac{t}{T}\int_0^T f_i}\quad\text{or}\quad\forall j\in[N]\backslash\{i\},\ l_{i,j}[0](t,x)=0.
    \]
\end{enumerate}
\end{prop}
\begin{proof}
We divide the proof into three steps: the concavity of $s\mapsto\lambda[s]$, the alternative between affinity or strict concavity,
the characterization of the affinity case.

\begin{proof}[Step 1: concavity]

Fix $s\in[0,1]$, set $z=(1-s)z_1+sz_2$ and, for all $(t,x)\in\R\times\R^n$, define the auxiliary matrix $\widetilde{\veL}[s](t,x)$ whose entries are:
\[
    \widetilde{l}_{i,j}[s](t,x)=
    \begin{cases}
        (1-s)l_{i,i}[0](t,x)+sl_{i,i}[1](t,x) & \text{if }i=j \\
        \left(l_{i,j}[0](t,x)\right)^{1-s}\left(l_{i,j}[1](t,x)\right)^s & \text{if }i\neq j
    \end{cases}
\]
(with $0^0=0$ by convention).
By construction, and by our convexity assumptions, $\veL[s]\leq\widetilde{\veL}[s]$ in $\R\times\R^n$. 
Hence, as a direct consequence of the min--max/max--min characterizations of the periodic principal eigenvalue
of Proposition \ref{prop:max--min_characterization_lambdaz}, we get: 
\begin{equation}
    \label{eq:concavity_eigenvalue_1}
    \lambda[s]\geq\lambda_{1,\upp}(\widetilde{\cbQ}[s]),
\end{equation}
where $\widetilde{\cbQ}[s]=\cbQ[s]+\veL[s]-\widetilde{\veL}[s]$.

Recall the notation $\upe_{z'}:x\mapsto\upe^{z'\cdot x}$ and note that, by definition of $\widetilde{\cbQ}[s]$, 
\[
    \frac{\left(\widetilde{\cbQ}[s](\upe_{-z}\veu)\right)_i}{\upe_{-z}u_i}=\frac{\caP_i u_i-\left(\widetilde{\veL}[s]\veu\right)_i}{u_i}\quad\text{for all }\veu\in\caC^{1,2}(\R\times\R^n,(\vez,\vei)).
\]
Hence there is a bijection between space-time periodic eigenfunctions of $\widetilde{\cbQ}[s]$ 
and time periodic eigenfunctions of $\dcbP-\widetilde{\veL}[s]$ whose product with $\upe_z$ is space periodic.

Let $\mu=\lambda_{1,\upp}(\widetilde{\cbQ}[0])$, $\nu=\lambda_{1,\upp}(\widetilde{\cbQ}[1])$ and $\upe_{-z_1}\veu$, $\upe_{-z_2}\vev$ two respectively associated space-time periodic positive eigenfunctions:
\[
    \widetilde{\cbQ}[0](\upe_{-z_1}\veu)=\mu\upe_{-z_1}\veu,\quad\widetilde{\cbQ}[1](\upe_{-z_2}\vev)=\nu\upe_{-z_2}\vev,
\]
\textit{i.e.}
\[
    \dcbP\veu-\widetilde{\veL}[0]\veu=\mu\veu,\quad\dcbP\vev-\widetilde{\veL}[1]\vev=\nu\vev.
\]
Define $\vew=\left(u_i^{1-s}v_i^s\right)_{i\in[N]}$. Since $\upe_{-z}w_i=(\upe_{-z_1}u_i)^{1-s}(\upe_{-z_2}v_i)^s$ for all $i\in[N]$, 
$\upe_{-z}\vew$ is space-time periodic and therefore we can use it as test function for $\widetilde{\cbQ}[s]$. 
Following Nadin \cite{Nadin_2007} for the expansion of the $\caP_i$ part and using the uniform ellipticity assumption
\ref{ass:ellipticity}, we find:
\begin{align*}
    \frac{\caP_i w_i-\left(\widetilde{\veL}[s]\vew\right)_i}{w_i} & = (1-s)\frac{\caP_i u_i}{u_i} +s\frac{\caP_i v_i}{v_i}
    +s(1-s)\left(\frac{\nabla u_i}{u_i}-\frac{\nabla v_i}{v_i}\right)\cdot A_i\left(\frac{\nabla u_i}{u_i}-\frac{\nabla v_i}{v_i}\right)\\
    &\quad - (1-s)\widetilde{l}_{i,i}[0]-s\widetilde{l}_{i,i}[1] 
    -\frac{1}{w_i}\sum_{j\in[N]\backslash\{i\}}(\widetilde{l}_{i,j}[0]u_j)^{1-s}(\widetilde{l}_{i,j}[1]v_j)^s \\
    & \geq (1-s)\frac{\caP_i u_i}{u_i} +s\frac{\caP_i v_i}{v_i}\\
    &\quad - (1-s)\widetilde{l}_{i,i}[0]-s\widetilde{l}_{i,i}[1] 
    -\frac{1}{w_i}\sum_{j\in[N]\backslash\{i\}}(\widetilde{l}_{i,j}[0]u_j)^{1-s}(\widetilde{l}_{i,j}[1]v_j)^s.
\end{align*}
Following Nussbaum \cite{Nussbaum_1986} and using the H\"{o}lder inequality, the equalities satisfied by $\veu$ and $\vev$ and
the inequality between arithmetic and geometric means, we get
\begin{equation}
    \label{eq:concavity_eigenvalue_components}
    \frac{\left(\widetilde{\cbQ}[s](\upe_{-z}\vew)\right)_i}{\upe_{-z}w_i}\geq(1-s)\mu+s\nu\quad\text{for all }i\in[N],
\end{equation}
and, eventually, the max--min characterization yields:
\begin{equation}
    \label{eq:concavity_eigenvalue_2}
    \lambda_{1,\upp}(\widetilde{\cbQ}[s])\geq (1-s)\mu + s\nu=(1-s)\lambda_{1,\upp}(\widetilde{\cbQ}[0])+s\lambda_{1,\upp}(\widetilde{\cbQ}[1]).
\end{equation}

Combining \eqref{eq:concavity_eigenvalue_1} and \eqref{eq:concavity_eigenvalue_2} and using the fact that $\cbQ[s]$ and $\widetilde{\cbQ}[s]$
coincide at $s=0$ and $s=1$, we find indeed the claimed concavity:
\begin{equation}
    \label{eq:concavity_eigenvalue}
    \lambda_{1,\upp}(\cbQ[s])\geq (1-s)\lambda_{1,\upp}(\cbQ[0])+s\lambda_{1,\upp}(\cbQ[1]).
\end{equation}
\end{proof}

\begin{proof}[Step 2: affinity or strict concavity]

Assume that $s\mapsto\lambda[s]$ is not strictly concave. This means that there exists
$s_0\in[0,1]$ such that \eqref{eq:concavity_eigenvalue} is an equality at $s=s_0$.

The equality in \eqref{eq:concavity_eigenvalue} at $s=s_0$ implies the equality in \eqref{eq:concavity_eigenvalue_1} at $s=s_0$, 
which in turn implies the equality $\veL[s_0]=\widetilde{\veL}[s_0]$ in $\R\times\R^n$. 
Since all $s\mapsto\widetilde{l}_{i,i}[s](t,x)$ are linear and all $s\mapsto l_{i,i}[s](t,x)$ are convex, 
$l_{i,i}[s](t,x)\leq \widetilde{l}_{i,i}[s](t,x)$ together
with the equality at $s=0$, $s=s_0$, $s=1$ imply $l_{i,i}=\widetilde{l}_{i,i}$ identically for all $i\in[N]$. Similarly, 
$l_{i,j}=\widetilde{l}_{i,j}$ identically for all $i,j\in[N]$. Hence, as functions of $(s,t,x)$,
$\veL=\widetilde{\veL}$ identically in $[0,1]\times\R\times\R^n$.

Similarly, the equality in \eqref{eq:concavity_eigenvalue} at $s=s_0$ implies the equality in
\eqref{eq:concavity_eigenvalue_2} at $s=s_0$, and then the max--min characterization (Proposition
\ref{prop:max--min_characterization_lambdaz}) implies equality in
\eqref{eq:concavity_eigenvalue_components} at $s=s_0$ for all $i\in[N]$ in $\R\times\R^n$. 
Then, this implies, for all $i\in[N]$:
\begin{itemize}
    \item $\nabla u_i/u_i=\nabla v_i/v_i$, that is there exists a function $a_i$ of the variable $t$ only such that $u_i(t,x)=a_i(t)v_i(t,x)$;
    \item for all $j\in[N]\backslash\{i\}$, there exists a positive function $c_i$ of $t$ and $x$ such that 
    $\widetilde{l}_{i,j}[0]u_j=c_i \widetilde{l}_{i,j}[1]v_j$ (equality in the H\"{o}lder inequality);
    \item $\frac{\caP_i u_i}{u_i}-\widetilde{l}_{i,i}[0]-\mu=\frac{\caP_i v_i}{v_i}-\widetilde{l}_{i,i}[1]-\nu$ 
    (equality in the inequality between geometric and arithmetic averages). 
\end{itemize}

Putting the two together, the equality in \eqref{eq:concavity_eigenvalue} at $s=s_0$ implies:
\begin{enumerate}[label=$(\text{Cond. }\arabic*')$]
    \item \label{cond:1'} $\veL=\widetilde{\veL}$ identically in $[0,1]\times\R\times\R^n$;
    \item \label{cond:2'} there exists a function $a_i$ of the variable $t$ only such that $u_i(t,x)=a_i(t)v_i(t,x)$;
    \item \label{cond:3'} for all $j\in[N]\backslash\{i\}$, there exists a positive function $c_i$ of $t$ and $x$ such that 
    $\widetilde{l}_{i,j}[0]a_j=c_i \widetilde{l}_{i,j}[1]$;
    \item \label{cond:4'} $\frac{\caP_i u_i}{u_i}-\widetilde{l}_{i,i}[0]-\mu=\frac{\caP_i v_i}{v_i}-\widetilde{l}_{i,i}[1]-\nu$. 
\end{enumerate}

These four conditions do not depend on $s_0$. Going back through Step 1, it appears that under these conditions, all inequalities are equalities.
Hence \eqref{eq:concavity_eigenvalue} is an equality at all $s\in[0,1]$, or in other words $s\mapsto\lambda[s]$ is affine. It will be useful in the
next step to note that this argument precisely shows that \ref{cond:1'}--\ref{cond:4'} are equivalent to the affinity of $s\mapsto\lambda[s]$.

\end{proof}

\begin{proof}[Step 3: necessary and sufficient conditions for affinity]

From Step 2, we know that $s\mapsto\lambda[s]$ is affine if and only if \ref{cond:1'}--\ref{cond:4'}.
Let us prove that this group of conditions is equivalent to the group \ref{cond:equality_case_1}--\ref{cond:equality_case_2}.

Note first that without loss of generality, we can assume that $\veu$ and $\vev$ are uniquely identified by the following normalizations:
\[
    \|u_1(0,\cdot)\|_{\mathcal{L}^\infty(\R^n,\R)} = 1,\quad\|v_1(0,\cdot)\|_{\mathcal{L}^\infty(\R^n,\R)} = 1.
\]

First, we prove that \ref{cond:1'}--\ref{cond:4'} imply \ref{cond:equality_case_1}--\ref{cond:equality_case_2}. 
From $u_i(t,x)=a_i(t)v_i(t,x)$, we deduce $z_1=z_2$ (recall that $\upe_{-z_1}\veu$ and $\upe_{-z_2}\vev$ 
are both space-time periodic) and $\frac{\caP_i u_i}{u_i}=\frac{\caP_i v_i}{v_i}+\frac{a_i'}{a_i}$.
The equality $\frac{\caP_i u_i}{u_i}-\widetilde{l}_{i,i}[0]-\mu=\frac{\caP_i v_i}{v_i}-\widetilde{l}_{i,i}[1]-\nu$ reads
$\frac{a_i'}{a_i}=\widetilde{l}_{i,i}[0]-\widetilde{l}_{i,i}[1]+\mu-\nu$, or in other words there exists $\vect{b}\in\R^N$ such that
\[
    a_i:t\mapsto b_i\exp\left(\int_0^t\left(\widetilde{l}_{i,i}[0](t',x)-\widetilde{l}_{i,i}[1](t',x)\right)\upd t'+(\mu-\nu)t\right).
\]
This directly implies that $f_i=\widetilde{l}_{i,i}[0]-\widetilde{l}_{i,i}[1]$ does not depend on $x$. Moreover, the 
positivity of both $u_i$ and $v_i$ implies $b_i>0$, the normalizations imply $b_1=1$, and the time periodicity 
implies that $f_i$ is periodic with average $\nu-\mu$, independent of $i$.
To characterize $\vect{c}$, we sum the $N-1$ equalities $l_{i,j}[0]a_j=c_il_{i,j}[1]$ for $j\in[N]\backslash\{i\}$ coming from
\ref{cond:3'} and rearrange terms as follows:
\begin{align*}
   0 & = \left(\sum_{j\in[N]\backslash\{i\}}l_{i,j}[1]v_j\right)c_i - \sum_{j\in[N]\backslash\{i\}}l_{i,j}[0]u_j \\
   & = \left((\veL[1]\vev)_i-l_{i,i}[1]v_i\right)c_i - \left((\veL[0]\veu)_i-l_{i,i}[0]u_i\right) \\
   & = \left(\frac{\caP_i v_i}{v_i}-l_{i,i}[1]-\nu\right)v_i c_i-\left(\frac{\caP_i u_i}{u_i}-l_{i,i}[0]-\mu\right)u_i
\end{align*}
Now, using \ref{cond:2'} and \ref{cond:4'}, we get
\begin{equation*}
   0 = \left(\frac{\caP_i v_i}{v_i}-l_{i,i}[1]-\nu\right)v_i(c_i-a_i) = \left(\sum_{j\in[N]\backslash\{i\}}l_{i,j}[1]v_j\right)(c_i-a_i).
\end{equation*}
By nonnegativity of each term in the sum, we deduce that, at each $(t,x)\in\clOmper$,
\[
    c_i(t,x)=a_i(t)\quad\text{or}\quad\forall j\in[N]\backslash\{i\},\ l_{i,j}[0](t,x)=0.
\]

Second, to verify that \ref{cond:equality_case_1}--\ref{cond:equality_case_2} imply \ref{cond:1'}--\ref{cond:4'}, it suffices to set
\[
    a_i:t\mapsto \frac{b_i}{b_1}\exp\left(\int_0^t f_i -\frac{t}{T}\int_0^T f_i\right),
\]
and to check $\frac{1}{T}\int_0^T f_i =\mu-\nu$ and $\veu=\vect{a}\circ\vev$.
Actually, $\widetilde{\veu}=\vect{a}\circ\vev$ satisfies
\begin{align*}
    \caP_i \widetilde{u}_i-\left(\veL[0]\widetilde{\veu}\right)_i & = a_i' v_i + a_i\caP_i v_i - \left(\veL[0](\vect{a}\circ\vev)\right)_i \\
    & =\left(f_i-\frac{1}{T}\int_0^T f_i\right)a_i v_i + a_i\left(\veL[1]\vev\right)_i +a_i \lambda_{1,\upp}(\cbQ[1])v_i - \left(\veL[0](\vect{a}\circ\vev)\right)_i \\
    & =\left(\lambda_{1,\upp}(\cbQ[1])-\frac{1}{T}\int_0^T f_i \right)\widetilde{u}_i +\sum_{j\in[N]\backslash\{i\}}\left(\left(l_{i,j}[1]a_i-l_{i,j}[0]a_j\right)v_j\right) \\
    & =\left(\lambda_{1,\upp}(\cbQ[1])-\frac{1}{T}\int_0^T f_i \right)\widetilde{u}_i +\left(a_i-c_i \right)\sum_{j\in[N]\backslash\{i\}}l_{i,j}[1]v_j \\
    & =\left(\lambda_{1,\upp}(\cbQ[1])-\frac{1}{T}\int_0^T f_i \right)\widetilde{u}_i.
\end{align*}
Since $\upe_{-z_2}\vev$ is space-time periodic, $\upe_{-z_1}\widetilde{\veu}=\upe_{-z_2}\vect{a}\circ\vev$ 
is also space-time periodic, whence by uniqueness
$\lambda_{1,\upp}(\cbQ[1])-T^{-1}\int_0^T f_i =\lambda_{1,\upp}(\cbQ[0])$ and $\upe_{-z_1}\widetilde{\veu}\in\vspan(\veu_{z_1})$. 
This exactly proves the existence of $C>0$ such that $C\widetilde{\veu}=\veu$, and, in view of the chosen normalizations on $\veu$ and $\vev$,
$C=1$, \textit{i.e.} $\widetilde{\veu}=\veu$. 
\end{proof}

The proof of the theorem is complete.
\end{proof}

\begin{rem}\label{rem:sharpness_of_NSC_for_equality_case}
When $\veL[0](t,x)$ is irreducible at all $(t,x)\in\clOmper$, the characterization of the function $\vect{c}$ in 
\ref{cond:equality_case_2} is immediately strengthened as: $c_i$ only depends on $t$ and $c_i:t\mapsto b_i\upe^{\int_0^t f_i-\frac{t}{T}\int_0^T f_i}$.

On the contrary, $\vect{c}$ cannot be uniquely determined when a line of $\veL[0]$ vanishes somewhere in $\clOmper$, and from
\ref{cond:equality_case_2} it is actually clear that if the $i$-th line of $\veL[0]$ vanishes in an open space-time ball 
$B\subset\clOmper$, then basically any nonnegative scalar function can be added to $a_i$ in $B$ and the resulting sum still 
forms an admissible $\vect{c}$.
This might come as a surprise, especially since pointwise irreducible matrices are dense\footnote{Just change $\veL$ into 
$\veL+\varepsilon\veo_{N\times N}$.} in the set of admissible matrices (namely, matrices in 
$\caC^{\delta/2,\delta}_\upp(\R\times\R^n,\R^{N\times N})$ satisfying \ref{ass:cooperative} and \ref{ass:irreducible}).
\end{rem}

\begin{cor}\label{cor:lambdaz_strictly_concave}
With the notations of Proposition \ref{prop:concavity_eigenvalue_z_L}, if $z_1\neq z_2$, then $s\in[0,1]\mapsto\lambda[s]$ is strictly concave. 
In particular, $z\mapsto\lambda_{1,z}$ is strictly concave.
\end{cor}

Very minor adaptations of the proof of Proposition \ref{prop:concavity_eigenvalue_z_L}, not detailed here, 
lead to the following analogous result in the Dirichlet case \footnote{The absence of $z$ actually makes the proof shorter.}.

\begin{prop}\label{prop:concavity_lambda1_Omega_L}
Let $\Omega\subset\R^n$ be a nonempty, bounded, smooth, open, connected set such that there exists $x_0\in\Omega$
satisfying $[x_0,x_0+L]\subset\Omega$.

Let 
\[
    \left(\veL[s]\right)_{s\in[0,1]}\in\left(\caC^{\delta/2,\delta}_\upp(\R\times\Omega,\R^{N\times N})\right)^{[0,1]}
\]
a family of matrices satisfying the same assumptions as $\veL$ (\textit{i.e.}, \ref{ass:cooperative}, \ref{ass:irreducible})
and such that, for all $(t,x)\in\R\times\Omega$ and $i\in[N]$,
\begin{enumerate}
    \item $s\mapsto l_{i,i}[s](t,x)$ is convex;
    \item for all $j\in[N]\backslash\{i\}$, $s\mapsto l_{i,j}[s](t,x)$ is either identically zero or log-convex.
\end{enumerate}

Then $s\in[0,1]\mapsto \lambda_{1,\upDir}(\Omega,\cbQ[s])$, where $\cbQ[s]$ is the operator $\cbQ$ with $\veL$ replaced $\veL[s]$, 
is affine or strictly concave and it is affine if and only if
there exist a constant vector $\vect{b}\gg\vez$, a function $\vect{c}\in\caC_{t-\upp}(\R\times\Omega,(\vez,\vei))$ 
and a function $\vect{f}\in\caC_{t-\upp}(\R,\R^N)$ satisfying $\int_0^T\vect{f}\in\vspan(\veo)$ such that the entries of $\veL$ have the form:
\begin{equation*}
    l_{i,j}[s]:(t,x)\mapsto
    \begin{cases}
        l_{i,i}[0](t,x)-sf_i(t) & \text{if }i=j, \\
        l_{i,j}[0](t,x)\left(\frac{b_j}{c_i(t,x)}\right)^s \upe^{s\left(\int_0^t f_j-\frac{t}{T}\int_0^T f_j\right)} & \text{if }i\neq j,
    \end{cases}
\end{equation*}
and such that the function $\vect{c}$ satisfies, at all $(t,x)\in\R\times\Omega$ and for each $i\in[N]$, 
\[
    c_i(t,x)=b_i\upe^{\int_0^t f_i-\frac{t}{T}\int_0^T f_i}\quad\text{or}\quad\forall j\in[N]\backslash\{i\},\ l_{i,j}[0](t,x)=0.
\]
\end{prop}

As a corollary, we obtain the concavity of $\lambda_1$ in arbitrary domains, namely Theorem \ref{thm:concavity_lambda1_L}.

\begin{cor}
Let $\Omega\subset\R^n$ be a nonempty open connected set such that there exists $x_0\in\Omega$ satisfying $[x_0,x_0+L]\subset\Omega$.

Let 
\[
    \left(\veL[s]\right)_{s\in[0,1]}\in\left(\caC^{\delta/2,\delta}_\upp(\R\times\Omega,\R^{N\times N})\right)^{[0,1]}
\]
a family of matrices satisfying the same assumptions as $\veL$ (\textit{i.e.}, \ref{ass:cooperative}, \ref{ass:irreducible}) 
and such that, for all $(t,x)\in\R\times\Omega$ and $i\in[N]$,
\begin{enumerate}
    \item $s\mapsto l_{i,i}[s](t,x)$ is convex;
    \item for all $j\in[N]\backslash\{i\}$, $s\mapsto l_{i,j}[s](t,x)$ is either identically zero or log-convex.
\end{enumerate}

Then the mapping $s\in[0,1]\mapsto \lambda_1(\Omega,\cbQ[s])$, where $\cbQ[s]$ is the operator $\cbQ$ with $\veL$ replaced $\veL[s]$, is concave.
\end{cor}

\begin{proof}
Just as in the proof of Proposition \ref{prop:eigenvalue_limit_Dirichlet}, we work with a sequence $(\Omega_k)_{k\in\N}$
of smooth, bounded, nonempty, open, connected subsets of $\Omega$ such that
\[
    [x_0,x_0+L]\subset\Omega_1,\quad\Omega_k\subset\Omega_{k+1},\quad\bigcup_{k\in\N}\Omega_k=\Omega.
\]
By virtue of Proposition \ref{prop:concavity_lambda1_Omega_L}, all $s\in[0,1]\mapsto\lambda_{1,\upDir}(\Omega_k,\cbQ[s])$ are concave.
By virtue of Proposition \ref{prop:eigenvalue_limit_Dirichlet}, $\lambda_{1,\upDir}(\Omega_k,\cbQ[s])\to\lambda_1(\Omega,\cbQ[s])$ as 
$k\to+\infty$, for all $s\in[0,1].$

The pointwise convergence of a sequence of concave functions on the compact set $[0,1]$ is automatically improved as uniform convergence 
in $[0,1]$, and the limit is concave on $[0,1]$ as well. This ends the proof.
\end{proof}

\begin{rem}
We will establish in the next section that $\lambda_1=\max_{z\in\R^n}\lambda_z$. However, the maximum of a family of concave functions
is in general not a concave function itself, so that this identity cannot be used to prove the concavity of $\lambda_1$. 
\end{rem}

\subsubsection{Relations between $\lambda_1$, $\lambda_1'$ and $\lambda_{1,z}$}

\begin{prop}\label{prop:existence_eigenfunction_exp_times_per}
There exists $z\in\R^n$ such that $\upe_z\veu_z$ is a generalized principal eigenfunction of $\cbQ$ associated with $\lambda_1$
and $\lambda_1=\lambda_{1,z}$.
\end{prop}

\begin{proof}
From Proposition \ref{prop:eigenvalue_limit_Dirichlet}, there exists a generalized principal eigenfunction 
$\veu\in\caC^{1,2}(\R\times\R^n,(\vez,\vei))$ associated with $\lambda_1$.

We first prove that there exists $z_1\in\R$ and a new generalized principal eigenfunction $\veu^1$ such that $(t,x)\mapsto\upe^{-z_1 x_1}\veu^1(t,x)$ 
is $L_1$-periodic with respect to $x_1$.

Define the translation $\tau:x\in\R^n\mapsto x+L_1 e_1$, where $e_1=(\delta_{1\alpha})_{\alpha\in[n]}$,
and denote $\veu^\tau:(t,x)\mapsto\veu(t,\tau(x))$ and 
$\vev=\left(u^\tau_i/u_i\right)_{i\in[N]}$. By virtue of the fully coupled Harnack inequality of Proposition \ref{prop:harnack_inequality} and
periodicity of the coefficients of $\cbQ$, $\vev$ is globally bounded.
Let 
\[
    z_1=L_1^{-1}\ln\left(\max_{i\in[N]}\sup_{(t,x)\in\R\times\R^n}v_i(t,x)\right).
\]
Recalling that $\veu$ and consequently $\vev$ are time periodic, 
there exists $\overline{i}\in[N]$ and $\left(t_k,x_k\right)_{k\in\N}\in([0,T]\times\R^n)^\N$ such that
$v_{\overline{i}}(t_k,x_k)\to\upe^{z_1L_1}$ as $k\to+\infty$. Moreover, there exists $\left(y_k\right)_{k\in\N}$
such that, for all $k\in\N$, $x_k-y_k\in L_1\mathbb{Z}\times\dots\times L_n\mathbb{Z}$. Up to extraction,
we assume that $(t_k,y_k)\to (t_\infty,y_\infty)\in\clOmper$.

Now, define, for all $k\in\N$,
\[
    \hat{\veu}_k:(t,x)\mapsto\frac{1}{u_{\overline{i}}(t_k,x_k)}\veu(t+t_k,x+x_k),
\]
\[
    \hat{\veu}^\tau_k:(t,x)\mapsto\hat{\veu}_k(t,\tau(x)),
\]
\[
    \vew_k:(t,x)\mapsto\upe^{z_1L_1}\hat{\veu}_k-\hat{\veu}_k^\tau.
\]

Once more by virtue of the Harnack inequality and
the periodicity of the coefficients of $\cbQ$, $\left(\hat{\veu}_k\right)_{k\in\N}$ is globally bounded.
By periodicity of the coefficients of $\cbQ$, it satisfies:
\[
    \cbQ(t+t_k,x+y_k)\hat{\veu}_k(t,x)=\lambda_1\hat{\veu}_k(t,x)\quad\text{for all }(t,x)\in\R\times\R^n,\ k\in\N.
\]
Therefore, by classical regularity estimates \cite{Lieberman_2005}, $\left(\hat{\veu}_k\right)_{k\in\N}$
converges up to a diagonal extraction to $\hat{\veu}_\infty\in\caC^{1,2}_{t-\upp}(\R\times\R^n,[\vez,\vei))$
which satisfies:
\[
    \cbQ(t+t_\infty,x+y_\infty)\hat{\veu}_\infty(t,x)=\lambda_1\hat{\veu}_\infty(t,x)\quad\text{for all }(t,x)\in\R\times\R^n,
\]
\textit{i.e.}
\[
    \cbQ(t,x)\hat{\veu}_\infty(t-t_\infty,x-y_\infty)=\lambda_1\hat{\veu}_\infty(t-t_\infty,x-y_\infty)\quad\text{for all }(t,x)\in\R\times\R^n.
\]
Moreover, $\hat{u}_{\overline{i},\infty}(0,0)=1$, whence $\hat{\veu}_\infty$ is nonzero.
By the strong maximum principle (see Proposition \ref{prop:strong_maximum_principle}), it is in fact positive.

We can now extend the family $(\vew_k)$ in $\N\cup\{\infty\}$ with 
$\vew_\infty=\upe^{z_1L_1}\hat{\veu}_\infty -\hat{\veu}_\infty^\tau$.
Since, for all $k\in\overline{\N}$,
\[
    \vew_k=\hat{\veu}_k\circ\left(\upe^{z_1L_1}\veo-\vev_k\right),\quad\text{where }\vev_k:(t,x)\mapsto\vev(t+t_k,x+x_k),
\]
we deduce by definition of $z_1$ that $\vew_\infty\geq\vez$ with $w_{\overline{i},\infty}(0,0)=0$. 
Moreover, $\vew_\infty$ satisfies the same equation than $\hat{\veu}_\infty$. Therefore,
by virtue of the strong maximum principle, $\vew_\infty$ is the zero function.
This exactly means that $\upe^{z_1L_1}\hat{\veu}_\infty=\hat{\veu}_\infty^\tau$.

It is now clear that $\veu^1:(t,x)\mapsto\hat{\veu}_\infty(t-t_\infty,x-y_\infty)$ is positive, 
time periodic, a solution of $\cbQ\veu^1=\lambda_1\veu^1$, and that the function
$(t,x)\mapsto\upe^{-z_1 x_1}\veu^1(t,x)$ is $L_1$-periodic with respect to $x_1$. 
Indeed, for any $(t,x)\in\R\times\R^n$,
\begin{align*}
    \upe^{-z_1 (x_1+L_1)}\veu^1(t,x+L_1 e_1) & =\upe^{-z_1 x_1}\upe^{-z_1 L_1}\veu^1(t,\tau(x)) \\
    & =\upe^{-z_1 x_1}\upe^{-z_1 L_1}\hat{\veu}_\infty^\tau(t-t_\infty,x-y_\infty) \\
    & =\upe^{-z_1 x_1}\hat{\veu}_\infty(t-t_\infty,x-y_\infty) \\
    & =\upe^{-z_1 x_1}\veu^1(t,x).
\end{align*}
The first part of the proof is done.

Next, we iterate this construction, replacing $\veu$ by $\veu^1$, in order to obtain a new
generalized principal eigenfunction $\veu^2$ such that 
$(t,x)\mapsto\upe^{-z_1 x_1}\upe^{-z_2 x_2}\veu^2(t,x)$ is $L_1$-periodic with respect to $x_1$ and 
$L_2$-periodic with respect to $x_2$. Iterating again, we finally obtain $z\in\R^n$ and $\veu^n\in\caC^{1,2}_{t-\upp}(\R\times\R^n,(\vez,\vei))$
such that $\veu^n$ is a generalized principal eigenfunction associated with $\lambda_1$ and such
that $\upe_{-z}\veu^n$ is space periodic. The uniqueness of the eigenpair $(\lambda_{1,z},\veu_z)$,
up to multiplication of $\veu_z$ by a constant, yields finally $\lambda_1=\lambda_{1,z}$ and
$\upe_{-z}\veu^n\in\vspan(\veu_z)$.
\end{proof}

\begin{cor}\label{cor:lambda1_max_lambdaz}
The generalized principal eigenvalue $\lambda_1$ satisfies:
\begin{equation}\label{eq:lambda1_max_lambdaz}
    \lambda_1 = \max_{z\in\R^n}\lambda_{1,z}
\end{equation}
and there exists a unique $z\in\R^n$ such that $\lambda_1=\lambda_{1,z}$.
\end{cor}

\begin{proof}
Proposition \ref{prop:existence_eigenfunction_exp_times_per} already shows that $\lambda_1$ is in the image of $z\mapsto\lambda_{1,z}$
and Corollary \ref{cor:lambdaz_strictly_concave} already shows that $z\mapsto\lambda_{1,z}$ is strictly concave. Thus
it only remains to show $\lambda_1\geq\sup_{z\in\R^n}\lambda_{1,z}$. This is actually obvious, since the
equality $\cbQ(\upe_z\veu_z)=\lambda_{1,z}\upe_z\veu_z$ (which is just the definition of the eigenpair $(\lambda_{1,z},\veu_z)$) directly
implies, in view of the definition of $\lambda_1$, the inequality $\lambda_1\geq\lambda_{1,z}$.
\end{proof}

\begin{rem}
Let
\[
    E=\left\{ \lambda\in\R\ |\ \exists \veu\in\caC^{1,2}_{t-\upp}(\R\times\R^n,(\vez,\vei))\ \cbQ\veu=\lambda\veu \right\}
\]
and denote $\Lambda\subset\R$ the image of $z\in\R^n\mapsto\lambda_{1,z}$.
From the equality $\cbQ\upe_z\veu_z=\lambda_{1,z}\upe_z\veu_z$, the following set inclusions hold true:
\[
\Lambda\subset E\subset\left\{ \lambda\in\R\ |\ \exists \veu\in\caC^{1,2}_{t-\upp}(\R\times\R^n,(\vez,\vei))\ \cbQ\veu\geq\lambda\veu \right\}.
\]
By strict concavity, $\Lambda=(-\infty,\max\lambda_{1,z}]$, and since $\lambda_1=\max\lambda_{1,z}$ is by definition the supremum of the larger set above, 
all inclusions above are actually set equalities.

This shows in particular that the set $E$ of eigenvalues of $\cbQ$ acting on the set $\caC^{1,2}_{t-\upp}(\R\times\R^n,(\vez,\vei))$ 
is $(-\infty,\lambda_1]$. This is of course in striking contrast with the case of smooth bounded domains,
where the Krein--Rutman theorem can be applied and the principal eigenvalue is unique.
For the same result in the elliptic case with general spatial heterogeneities, refer to Berestycki--Rossi \cite[Theorem 1.4]{Berestycki_Ros_1}
(scalar setting) and Arapostathis--Biswas--Pradhan \cite[Theorem 1.2]{Arapostathis_Biswas_Pradhan_2020} (cooperative vector setting).
\end{rem}

\begin{prop}\label{prop:lambda1prime_lambda0}
The generalized principal eigenvalue $\lambda_1'$ satisfies:
\begin{equation}\label{eq:lambda1prime_lambda0}
    \lambda_1'=\lambda_{1,0}.
\end{equation}
\end{prop}

\begin{proof}
Since $\veu_0=\upe_0\veu_0$ is globally bounded, we can use it as test function in the definition of $\lambda_1'$ and obtain
$\lambda_1'\leq \lambda_{1,0}$. 

Now, we assume by contradiction that this inequality is actually strict, so that by definition of $\lambda_1'$, there exists
$\mu\in(\lambda_1',\lambda_{1,0})$ and $\veu\in\mathcal{W}^{1,\infty}\cap\caC^{1,2}_{t-\upp}(\R\times\R^n,(\vez,\vei))$ such 
that $\cbQ\veu\leq\mu\veu$.

We can now define 
\[
    \kappa^\star = \inf\left\{ \kappa>0\ |\ \kappa\veu_0-\veu\gg\vez \right\}
\]
and study the sign of $\vev=\kappa^\star\veu_0-\veu$. This function satisfies 
\[
\cbQ\vev\geq(\lambda_{1,0}-\mu)\kappa^\star\veu_0+\mu\vev,
\]
is time periodic and nonnegative, and by optimality there exists 
$\left((t_k,x_k)\right)_{k\in\N}\in\left([0,T]\times\R^n\right)^\N$ and $\underline{i}\in[N]$ such that 
\[
    v_{\underline{i}}(t_k,x_k)\to 0\quad\text{as }k\to+\infty.
\]
Moreover, there exists $(y_k)_{k\in\N}$ such that, for all $k\in\N$, $x_k-y_k\in L_1\mathbb{Z}\times\dots\times L_n\mathbb{Z}$. Up to
extraction, we assume that $(t_k,y_k)\to(t_\infty,y_\infty)\in\clOmper$.

Define 
\[
    \vev_k:(t,x)\mapsto\vev(t+t_k,x+x_k)\quad\text{for all }k\in[N].
\]
By standard regularity estimates \cite{Lieberman_2005}, $\left(\vev_k\right)_{k\in\N}$ converges
up to a diagonal extraction to a function $\vev_\infty\in\mathcal{L}^\infty\cap\caC^{1,2}_{t-\upp}(\R\times\R^n,[\vez,\vei))$ 
which satisfies $v_{\underline{i},\infty}(0,0)=0$ and, for all $(t,x)\in\R\times\R^n$,
\[
    (\cbQ-\mu)(t+t_\infty,x+y_\infty)\vev_\infty(t,x)\geq(\lambda_{1,0}-\mu)\kappa^\star\min_{i\in[N]}\min_{\clOmper}\left(u_{0,i}\right)\veo\gg\vez.
\]
By virtue of the strong maximum principle (see Proposition \ref{prop:strong_maximum_principle}), $\vev_\infty=\vez$, 
but then this contradicts the preceding inequality. This ends the proof.
\end{proof}

\begin{rem}\label{rem:counter-example_evenness_without_advection}
    It is natural to investigate the equality between $\lambda_1=\max_{z\in\R^n}\lambda_{1,z}$ 
    and $\lambda_1'=\lambda_{1,0}$. The scalar counter-example with constant coefficients 
$\mathcal{Q}=\partial_t-\partial_{xx}+q\partial_x-l$ shows that both outcomes are possible, since 
$\lambda_{1,z}=z(q-z)-l$ is maximal at $z=0$ if and only if $q=0$. Identifying precise conditions 
for the maximality at $z=0$ becomes 
then one of our main goals. A very recent contribution by Griette and Matano \cite[Proposition 4.1]{Griette_Matano_2021} shows that in
the vector setting, the absence of advection is not enough.

Their two-dimensional counter-example in one-dimensional space is:
\begin{equation*}
    \cbQ=\partial_t-\partial_x\left(\diag\begin{pmatrix}a_1 \\ a_2\end{pmatrix}\partial_x\right)-
    \begin{pmatrix}
        r_1-\frac{1}{\varepsilon}p & \frac{1}{\varepsilon}(1-p) \\
        \frac{1}{\varepsilon}p & r_2-\frac{1}{\varepsilon}(1-p)
    \end{pmatrix}
\end{equation*}
with $a_1$, $a_2$, $r_1$, $r_2$ and $p$ periodic functions of $x$.
As $\varepsilon\to 0$, locally uniformly with respect to $z$,
\begin{equation*}
    \lambda_{1,z}(\cbQ)\to\lambda_{1,z}\left(-\partial_x(a\partial_x)+q\partial_x-(r-q')\right)
\end{equation*}
with 
\begin{equation*}
    a=(1-p)a_1+pa_2,\quad r=(1-p)r_1+pr_2,\quad q=(a_1-a_2)p'.
\end{equation*}
Under the condition $\int_0^{L_1}q/a\neq 0$, the limit is not maximal at $z=0$ \cite[Appendix A]{Griette_Matano_2021}, whence $\lambda_{1,z}(\cbQ)$
is also not maximal at $z=0$ when $\varepsilon$ is sufficiently small. For more details, we refer to \cite{Griette_Matano_2021}.
\end{rem}

\subsubsection{Rough estimates}

Here we state rough upper and lower estimates that will be used later on in the proofs. The more precise estimates of
Subsection \ref{sec:theorems_explicit_formulas}, that use special assumptions on the coefficients of $\cbQ$,
will be proved later.

Using the cooperativity assumption \ref{ass:cooperative}, the min--max characterization of Proposition
\ref{prop:max--min_characterization_lambdaz}, 
the equalities $\lambda_1=\max\lambda_{1,z}$ and $\lambda_1'=\lambda_{1,0}$ of Corollary \ref{cor:lambda1_max_lambdaz} and 
Proposition \ref{prop:lambda1prime_lambda0} respectively, and the corresponding scalar results \cite{Nadin_2007}, 
we deduce the following corollary which relates the generalized principal eigenvalues of the operator $\cbQ$ to the generalized principal 
eigenvalues of the scalar operators $\caP_i-l_{i,i}$.

\begin{cor}\label{cor:comparison_eigenvalues_decoupled_equations}
For all $z\in\R^n$,
\[
    \lambda_{1,z}(\cbQ)\leq\min_{i\in[N]}\lambda_{1,z}(\caP_i-l_{i,i}).
\]
Consequently, 
\[
\lambda_1(\cbQ)\leq\min_{i\in[N]}\lambda_1(\caP_i-l_{i,i})\quad\text{and}\quad\lambda_1'(\cbQ)\leq\min_{i\in[N]}\lambda_1'(\caP_i-l_{i,i}).
\]
\end{cor}

\begin{rem}
Rougher but more explicit estimates can subsequently be derived by considering constant test functions in the min--max characterization of
$\lambda_{1,z}(\caP_i-l_{i,i})$ and the discrete Cauchy--Schwarz inequality:
\begin{align*}
    \lambda_{1,z}(\caP_i-l_{i,i}) & \leq \max_{\clOmper}\left(-z\cdot A_i z -\nabla\cdot\left(A_i z\right)+q_i\cdot z-l_{i,i}\right) \\ 
    & \leq -\min_{\clOmper}(z\cdot A_i z)+\max_{\clOmper}(-\nabla\cdot(A_i z))+\max_{\clOmper}(q_i\cdot z)-\min_{\clOmper}l_{i,i} \\
    & \leq -\underline{l}_{i,i}-\min_{\clOmper}\min_{y\in\Sn}(y\cdot A_i y)|z|^2 +\max_{\clOmper}\sum_{\alpha=1}^n \partial_\alpha\left(\sum_{\beta=1}^n -A^i_{\alpha,\beta}z_\beta\right)+\max_{\clOmper}(|q_i|)|z| \\
    & \leq -\underline{l}_{i,i}-\min_{\clOmper}\min_{y\in\Sn}(y\cdot A_i y)|z|^2 +\sum_{\alpha=1}^n \max_{\clOmper}\sum_{\beta=1}^n \partial_\alpha (-A^i_{\alpha,\beta})z_\beta +\||q_i|\||z| \\
    & \leq -\underline{l}_{i,i}-\min_{\clOmper}\min_{y\in\Sn}(y\cdot A_i y)|z|^2
    +\sum_{\alpha=1}^n\left\|\left(\sum_{\beta=1}^n|\partial_\alpha A^i_{\alpha,\beta}|^2\right)^{1/2}\right\||z| 
    +\||q_i|\||z|,
\end{align*}
where $\underline{\veL}$ is defined in \ref{ass:cooperative} and the notation $\|\cdot\|$ refers to the norm
in the space $\mathcal{L}^\infty(\R\times\R^n,\R)$.
\end{rem}

Another way to obtain rough upper and lower estimates consists in using 
\[
    \underline{\veL}+\left(\min_{\clOmper}\min_{y\in\Sn}(y\cdot A_i y)|z|^2-K|z|\right)\vect{I}\leq\veL+\diag\left(z\cdot A_i z+\nabla\cdot(A_i z)-q_i\cdot z\right)
\]
and
\[
    \veL+\diag\left(z\cdot A_i z+\nabla\cdot(A_i z)-q_i\cdot z\right)\leq\overline{\veL}+\left(\max_{\clOmper}\max_{y\in\Sn}(y\cdot A_i y)|z|^2+K|z|\right)\vect{I},
\]
where 
\[
    K=\sum_{\alpha=1}^n\left\|\left(\sum_{\beta=1}^n|\partial_\alpha A^i_{\alpha,\beta}|^2\right)^{1/2}\right\|+\||q_i|\|.
\]
Although $\underline{\veL}$ might not be irreducible, its Perron--Frobenius eigenvalue is still well-defined by continuous 
extension; it admits nonnnegative nonzero eigenvectors that can be used as sub-solutions. Using Perron--Frobenius
eigenvectors as test functions in the min--max and max--min formulas of Proposition \ref{prop:max--min_characterization_lambdaz},
we deduce the following corollary.

\begin{cor}\label{cor:comparison_eigenvalues_underline_overline_L}
    Let 
    \[
        A=\max_{i\in[N]}\left(\max\left(\max_{\clOmper}\max_{y\in\Sn}(y\cdot A_i y),\frac{1}{\min_{\clOmper}\min_{y\in\Sn}(y\cdot A_i y)}\right)\right)>0,
    \]
    \[
        B=\max_{i\in[N]}\left(\sum_{\alpha=1}^n\left(\sum_{\beta=1}^n\max_{\clOmper}|\partial_\alpha A^i_{\alpha,\beta}|^2\right)^{1/2}+\max_{\clOmper}|q_i|\right)\geq 0.
    \]
    
    Then, for all $z\in\R^n$, 
    \[
        -\lambda_{\upPF}(\overline{\veL})-A|z|^2-B|z|
        \leq \lambda_{1,z}
        \leq -\lambda_{\upPF}(\underline{\veL})-\frac{1}{A}|z|^2+B|z|.
    \]
\end{cor}

\subsection{Asymptotic dependence: proof of Theorems \ref{thm:continuity_eigenvalue_L}--\ref{thm:limits_eigenvalue_time_frequency}}
Proposition \ref{prop:concavity_eigenvalue_z_L} already proves Theorem \ref{thm:concavity_eigenvalue_L}. Below, we prove
the remaining theorems on coefficient dependence.

\subsubsection{Continuity: proof of Theorem \ref{thm:continuity_eigenvalue_L}}

We begin with the proof of Theorem \ref{thm:continuity_eigenvalue_L}, whose statement is recalled below.

\begin{prop}\label{prop:continuity_eigenvalue_L}
Let $\veL^\triangle\in\caC_\upp^{\delta/2,\delta}(\R\times\R^n,\R^{N\times N})$ be a block upper triangular essentially nonnegative matrix. 
Let $N'\in[N]$ and $(N_k)_{k\in[N'-1]}$ such that 
\[
    N_0=0< 1\leq N_1\leq N_2\leq\dots\leq N_{N'-1}\leq N_{N'}=N
\]
and such that 
\[
    (l^\triangle_{i,j})_{(i,j)\in([N_k]\backslash[N_{k-1}])^2}
\]
is the $k$-th diagonal block of $\veL^\triangle$ (with the convention $[0]=\emptyset$). Assume
\[
    \left(\max_{(t,x)\in\clOmper}l^\triangle_{i,j}(t,x)\right)_{(i,j)\in([N_k]\backslash[N_{k-1}])^2}\quad\text{is irreducible for all }k\in[N'].
\]

Let
\[
    \cbQ_k= \diag(\caP_i)_{i\in[N_k]\backslash[N_{k-1}]}-(l^\triangle_{i,j})_{(i,j)\in([N_k]\backslash[N_{k-1}])^2}\quad\text{for all }k\in[N'].
\]

Then, as $\veL\to\veL^\triangle$ in $\caC_\upp^{\delta/2,\delta}(\R\times\R^n,\R^{N\times N})$, 
\begin{equation*}
    \lambda_{1,z}(\cbQ)\to\min_{k\in[N']}\lambda_{1,z}\left(\cbQ_k\right)\quad\text{for all }z\in\R^n,
\end{equation*}
\begin{equation*}
    \lambda_1(\cbQ)\to\max_{z\in\R^n}\min_{k\in[N']}\lambda_{1,z}(\cbQ_{k})\leq\min_{k\in[N']}\lambda_1(\cbQ_{k}).
\end{equation*}
\end{prop}

\begin{proof}
\begin{proof}[Step 1: the special case z=0]
Let $(\veL_p)_{p\in\N}$ be a sequence of matrices satisfying \ref{ass:cooperative}, \ref{ass:irreducible} and that 
converges to $\veL^\triangle$ in $\caC_\upp^{\delta/2,\delta}(\R\times\R^n,\R^{N\times N})$. Denote $\cbQ_p=\diag(\caP_i)-\veL_p$
and $\cbQ_{k,p}=\diag(\caP_i)_{i\in[N_k]\backslash[N_{k-1}]}-(l_{p,i,j})_{(i,j)\in([N_k]\backslash[N_{k-1}])^2}$.

Since
\[
    0\leq \veL_p\leq \left(\sup_{p\in\N}\max_{(t,x)\in\clOmper}l_{p,i,j}(t,x)\right),
\]
we can derive from the max--min and min--max characterizations of $\lambda_1'(\veL_p)$ (see Proposition
\ref{prop:max--min_characterization_lambdaz}) uniform bounds on $\left( \lambda_1'(\veL_p) \right)_{p\in\N}$. Therefore
up to extraction this sequence converges to a limit $\lambda\in\R$. Similarly, up to extraction, the associated generalized principal
eigenfunction $\veu_{p}\in\caC^{1,2}_\upp(\R\times\R^n,(\vez,\vei))$ with normalization $|\veu_{p}(0,0)|=1$ converges to a nonnegative nonzero 
limit $\veu\in\caC^{1,2}_\upp(\R\times\R^n,[\vez,\vei))$ satisfying the same normalization and satisfying
\[
    \diag(\caP_i)\veu-\veL^\triangle\veu=\lambda\veu.
\]

Note that for each $k\in[N']$,
\begin{align*}
    \cbQ_{k}(u_i)_{i\in[N_k]\backslash[N_{k-1}]} & =\lambda(u_i)_{i\in[N_k]\backslash[N_{k-1}]}+\left(\sum_{j\in[N_{k-1}]\cup[N]\backslash[N_k]}l^\triangle_{i,j}u_j\right)_{i\in[N_k]\backslash[N_{k-1}]} \\
    & \geq \lambda(u_i)_{i\in[N_k]\backslash[N_{k-1}]}.
\end{align*}
Therefore, from the strong maximum principle of Proposition \ref{prop:strong_maximum_principle}, either $(u_i)_{i\in[N_k]\backslash[N_{k-1}]}=\vez$
or $(u_i)_{i\in[N_k]\backslash[N_{k-1}]}\gg\vez$. For all $k\in[N']$ such that $(u_i)_{i\in[N_k]\backslash[N_{k-1}]}\gg\vez$,
it follows from the characterization of $\lambda_{1,\upp}(\cbQ_{k})$ (see Proposition \ref{prop:max--min_characterization_lambdaz})
that $\lambda\leq\lambda_{1,\upp}(\cbQ_{k})$. Since $\veu$ is nonzero, there exists at least one such $k$. Let $I\subset[N']$ be the set
of all such $k$ and let $J=[N']\backslash I$.

If $N'\in I$, then from the special block upper triangular form of $\cbQ$, $\lambda=\lambda_1'(\cbQ_{N'})$. Otherwise, 
there exists $k\in[N'-1]\cap I$. It follows then from a classical inductive argument that there exists indeed $k\in[N']$ 
such that $\lambda=\lambda_1'(\cbQ_k)$.

Now, assume by contradiction that there exists $k'\in[N']$ such that $\lambda>\lambda_1'(\cbQ_{k'})$. Let $\eta=\lambda-\lambda_1'(\cbQ_{k'})>0$.
Let $\veu_{k'}$ be a periodic principal eigenfunction associated with $\lambda_{1,\upp}(\cbQ_{k'})$. Let $\underline{\veu}$ be defined as 
\[
    \underline{u}_i=
    \begin{cases}
        u_{k',i-N_{k'-1}} & \text{if }i\in[N_{k'}]\backslash[N_{k'-1}], \\
        0 & \text{otherwise}.
    \end{cases}
\]
Then, for all $i\in[N]$,
\[
    (\cbQ_{p}\underline{\veu})_i=
    \begin{cases}
        ((\cbQ_{k',p}-\cbQ_{k'})\underline{\veu})_i+\lambda_1'(\cbQ_{k'})\underline{u}_i &  \text{if }i\in[N_{k'}]\backslash[N_{k'-1}], \\
        \displaystyle-\sum_{j\in[N_{k'}]\backslash[N_{k'-1}]}l_{p,i,j}\underline{u}_j & \text{otherwise}.
    \end{cases}
\]
On one hand, 
\[
    -\sum_{j\in[N_{k'}]\backslash[N_{k'-1}]}l_{p,i,j}\underline{u}_j\leq 0=\underline{u}_i\quad\text{for all }i\notin[N_{k'}]\backslash[N_{k'-1}].
\]
On the other hand, by convergence of $\veL_p$ and the Harnack inequality of Proposition \ref{prop:harnack_inequality} applied to the 
fully coupled operator $\cbQ_{k'}$, we can assume that $p\in\N$ is so large that, for all $i\in[N_{k'}]\backslash[N_{k'-1}]$,
\[
    ((\cbQ_{k',p}-\cbQ_{k'})\underline{\veu})_i=((\veL_{k',p}-\veL^\triangle_{k'})\underline{\veu})_i\leq \frac{\eta}{2}\underline{u}_i,
\]
where $\veL_{k',p}=\left(l_{p,i,j}\right)_{(i,j)\in([N_{k'}]\backslash[N_{k'-1}])^2}$ and
$\veL^\triangle_{k'}=\left(l^\triangle_{i,j}\right)_{(i,j)\in([N_{k'}]\backslash[N_{k'-1}])^2}$. Hence
\[
    \cbQ_{p}\underline{\veu}\leq \left(\lambda_1'(\cbQ_{k'})+\frac{\eta}{2}\right)\underline{\veu}=\left(\lambda-\frac{\eta}{2}\right)\underline{\veu}.
\]
If $\lambda_1'(\cbQ_p)>\lambda-\frac{\eta}{2}$, then we can study $\kappa^\star\veu_{p}-\underline{\veu}$ with
\[
    \kappa^\star=\frac{\displaystyle\max_{i\in[N]}\max_{(t,x)\in\clOmper}\underline{u}_i(t,x)}{\displaystyle\min_{i\in[N]}\min_{(t,x)\in\clOmper}u_{p,i}(t,x)}>0
\]
and, by full coupling of $\cbQ_{p}$, deduce a contradiction from the strong maximum principle. 
Hence $\lambda_1'(\cbQ_{p})\leq\lambda-\frac{\eta}{2}$. But now, assuming in
addition that $p$ is so large that $\lambda_1'(\cbQ_p)>\lambda-\frac{\eta}{3}$, we find a contradiction. Therefore, for all $k'\in[N']$,
$\lambda\leq\lambda_1'(\cbQ_{k'})$, or in other words:
\[
    \lambda\leq\min_{k'\in[N']}\lambda_1'(\cbQ_{k'}).
\]
Combining this with $\lambda=\lambda_1'(\cbQ_k)$, we deduce that the preceding inequality is an equality.

This argument shows that any convergent subsequence of the sequence $\left(\lambda_1'(\veL_p)\right)_{p\in\N}$ converges to 
$\min_{k\in[N']}\lambda_1'(\cbQ_{k})$. The conclusion follows.
\end{proof}

\begin{proof}[Step 2: the general case $z\in\R^n$]
In view of
\[
    \lambda_{1,z}(\cbQ)=\lambda_1'(\cbQ_z)=\lambda_1'\left(\cbQ-\diag\left((A_i+A_i^\upT)z\cdot\nabla+z\cdot A_i z+\nabla\cdot\left(A_i z\right)-q_i\cdot z\right)\right),
\]
in order to prove the convergence of $\lambda_{1,z}$ for any $z\in\R^n$, we only have to apply the preceding step to the operator $\cbQ_z$. 
\end{proof}

\begin{proof}[Step 3: convergence of $\lambda_1=\max_{z\in\R^n}\lambda_{1,z}$]
Since all $z\mapsto\lambda_{1,z}(\cbQ_p)$, $p\in\N$, are concave, the pointwise convergence is automatically improved to locally uniform convergence. 

On one hand, recall from Corollary \ref{cor:comparison_eigenvalues_underline_overline_L} and
the ellipticity assumption \ref{ass:ellipticity} that there exists $A>0$, $B\geq 0$ and $C\in\R$ such that, for all $p\in\N$,
\[
    \lambda_{1,z}(\cbQ_p)\leq -A|z|^2+B|z|+C\quad\text{for all }p\in\N.
\]

On the other hand, for all $p\in\N$, $\lambda_1(\cbQ_p)\geq\lambda_1'(\cbQ_p)$. In particular,
$\lambda_1(\cbQ_p)\geq\inf_{p\in\N}\lambda_1'(\cbQ_p)$ and this lower bound is finite by virtue of Step 1 above.

Consequently, for all $p\in\N$, the point $z_p$ where the maximum is achieved (which is indeed uniquely defined, 
see Corollary \ref{cor:lambda1_max_lambdaz}) is necessarily in the set $Z$ defined as: 
\[
    Z=\left\{ z\in\R^n\ |\ \inf_{p\in\N}\lambda_1'(\cbQ_p)\leq -A|z|^2+B|z|+C\right\}.
\]
This set is compact.

To conclude, from the already established equality:
\[
    \lim_{p\to+\infty}\lambda_{1,z}(\cbQ_p)=\min_{k\in[N']}\lambda_{1,z}(\cbQ_{k}),
\]
and from the uniform convergence in $Z$ and the concavity in $\R^n$, we deduce
\begin{align*}
    \lim_{p\to+\infty}\lambda_1(\cbQ_p) & =\lim_{p\to+\infty}\max_{z\in Z}\lambda_{1,z}(\cbQ_p) \\
    & =\max_{z\in Z}\lim_{p\to+\infty}\lambda_{1,z}(\cbQ_p) \\
    & =\max_{z\in\R^n}\lim_{p\to+\infty}\lambda_{1,z}(\cbQ_p) \\
    & =\max_{z\in\R^n}\min_{k\in[N']}\lambda_{1,z}(\cbQ_{k}).
\end{align*}

Finally, from the inequality $\lambda_{1,z}(\cbQ_k)\leq\lambda_1(\cbQ_k)$ for all $k$ and $z$, it follows that 
\[
    \min_{k\in[N']}\lambda_{1,z}(\cbQ_k)\leq\min_{k\in[N']}\lambda_1(\cbQ_k),
\]
whence 
\[
    \max_{z\in\R^n}\min_{k\in[N']}\lambda_{1,z}(\cbQ_{k})\leq\min_{k\in[N']}\lambda_1(\cbQ_{k}).
\]
\end{proof}
This ends the proof.
\end{proof}

\begin{rem}\label{rem:continuity_weak_topology}
It should be noted here that the only part of this proof that is seemingly specific to the case of H\"{o}lder-continuous coefficients is the Harnack inequality provided by Proposition \ref{prop:harnack_inequality}. However, this is merely used to show that $\kappa^*\in (0,\infty)$. In the case of coefficients in $\mathcal{L}^\infty$, a simple reasoning by contradiction shows that $\kappa^*\in(0, \infty)$ is still true. A minor adaptation of the arguments then shows the continuity of $\lambda_1'$ for the weak-$\star$ $\mathcal{L}^\infty$ convergence of coefficients.
\end{rem}

\begin{rem}
We will use repeatedly the arguments of Step 2 and Step 3 above in what follows, in order to deduce the convergence of 
$\lambda_{1,z}$ and $\lambda_1$ when the convergence of $\lambda_1'$ has been established.

Note however that the estimate $\lambda_{1,z}(\cbQ)\leq -A|z|^2+B|z|+C$ with $A>0$ and $B\geq 0$ of Corollary 
\ref{cor:comparison_eigenvalues_underline_overline_L} becomes useless when the diffusion 
matrices $A_i$ vanish. This is consistent with the fact that, in Theorem \ref{thm:continuity_eigenvalue_diffusion_advection}, 
the convergence of $\lambda_1$ is in general false.
\end{rem}

\begin{rem}\label{rem:counter-example_lambda1_reducible}
The inequality 
\[
    \max_{z\in\R^n}\min_{k\in[N']}\lambda_{1,z}(\cbQ_{k})\leq\min_{k\in[N']}\lambda_1(\cbQ_{k})
\]
is strict in some cases.
Consider for instance the following space-time homogeneous, one-dimensional, two-component counter-example:
\[
    \cbQ=\cbQ_\varepsilon=\diag\left(\begin{pmatrix}
        \partial_t u_1-\partial_{xx}u_1 \\
        \partial_t u_2-\partial_{xx}u_2+2\partial_x u_2-1 \\
    \end{pmatrix}\right)-\varepsilon\begin{pmatrix}
        -1 & 1 \\
        1 & -1
    \end{pmatrix},
\]
where $\varepsilon>0$. The operator is diagonal if $\varepsilon=0$. By standard reduction 
(\textit{cf.} \eqref{eq:reduction_lambdaz_space-time_homogeneous}), the two scalar operators on the diagonal, 
$\cbQ_1=\partial_t-\partial_{xx}$ and $\cbQ_2=\partial_t-\partial_{xx}+2\partial_x-1$, 
satisfy $\lambda_{1,z}(\cbQ_1)=-z^2$ and $\lambda_{1,z}(\cbQ_2)=-(z-1)^2$. In particular, $\lambda_1(\cbQ_1)=\lambda_1(\cbQ_2)=0$. 
However, the function $z\mapsto\min(-z^2,-(z-1)^2)$ coincides with 
\[
    z\mapsto
    \begin{cases}
        -(z-1)^2 & \text{if }z<1/2, \\
        -z^2 & \text{if }z\geq 1/2,
    \end{cases}
\]
whose maximal value is $-1/4<0$, which is attained at $1/2$.
\end{rem}

\subsubsection{Vanishing diffusion and advection: proof of Theorem \ref{thm:continuity_eigenvalue_diffusion_advection}}

Below, we prove Theorem \ref{thm:continuity_eigenvalue_diffusion_advection} on vanishing diffusion and advection rates.

Following the statement of the theorem, we fix a function $\vect{f}\in\caC^1\left([0,+\infty),[\vez,\vei)\right)$ 
such that $\vect{f}^{-1}(\{\vez\})=\{0\}$ and $\vect{f}'(0)\neq\vez$, as well as
a family $((q_i^\varepsilon)_{i\in[N]})_{\varepsilon\geq 0}$ such that, for all $\varepsilon\geq 0$, 
$(q_i^\varepsilon)_{i\in[N]}\in\caC^{\delta/2,\delta}_\upp(\R\times\R^n,\R^n)$ and such that
$(q_i^\varepsilon)_{i\in[N]}\to (q_i^0)_{i\in[N]}$ in $\caC^{\delta/2,\delta}_\upp(\R\times\R^n,\R^n)$ as $\varepsilon\to 0$.
Note that, by assumption, $\vect{f}'(0)>\vez$.

We denote $\cbQ_{\varepsilon}$ the operator $\cbQ$ where $(A_i)_{i\in[N]}$, $(q_i)_{i\in[N]}$ 
and $\veL$ are replaced respectively by
$(f_i(\varepsilon)^2 A_i)_{i\in[N]}$, $(f_i(\varepsilon)q_i^\varepsilon)_{i\in[N]}$ and
a parameterized matrix $\veL^\varepsilon$ that still satisfies
\ref{ass:cooperative}--\ref{ass:smooth_periodic} and that converges uniformly to $\veL^0=\veL$ as $\varepsilon\to 0$.
We use the (slightly abusing) notations $A_i(x):t\mapsto A_i(t,x)$, $q_i^\varepsilon(x):t\mapsto q_i^\varepsilon(t,x)$ and $\veL^\varepsilon(x):t\mapsto\veL^\varepsilon(t,x)$.

Recall that, for any $z\in\R^n$,
\begin{align*}
    \upe_{-z}\cbQ_\varepsilon\upe_z = &\ \partial_t -\diag\left(f_i(\varepsilon)^2\nabla\cdot\left(A_i\nabla\right)\right) \\
    & -\diag\left(f_i(\varepsilon)\left(2f_i(\varepsilon)A_i z-q_i^{\varepsilon}\right)\cdot\nabla\right) \\
    & -\diag\left(\left(f_i(\varepsilon)^2 z\cdot A_iz +f_i(\varepsilon)^2\nabla\cdot(A_i z) - f_i(\varepsilon)q_i^\varepsilon\cdot z\right)\right)-\veL^\varepsilon
\end{align*}

Therefore the convergence in the case where the coefficients of $\cbQ_\varepsilon$ do not depend on space
is a straightforward consequence of \eqref{eq:reduction_lambdaz_space_homogeneous} and of the continuity of the periodic
principal eigenvalue (see Theorem \ref{thm:continuity_eigenvalue_L}).

The above expansion of $\upe_{-z}\cbQ_\varepsilon\upe_z$ also shows that, thanks to the
replacement of $\veL$ by $\veL^\varepsilon$, this operator has again the form of
$\cbQ_\varepsilon$.
Therefore, by virtue of the above expansion and of the equality, for all $x\in[0,L]$,
\[
    \lambda_{1,\upp}\left(\frac{\upd}{\upd t}-\veL(x)\right)=\lambda_1'\left(\partial_t-\diag(f_i'(0)^2\nabla\cdot(A_i(x)\nabla)-f_i'(0) q_i^0(x)\cdot\nabla)-\veL(x)\right),
\]
the following result is actually sufficient to prove the remaining part of Theorem \ref{thm:continuity_eigenvalue_diffusion_advection}.

\begin{prop}\label{prop:limit_eigenvalue_diffusion_advection}
For any $x\in[0,L]$, let
\[
    \widetilde{\cbQ}(x) = \partial_t-\diag\left(f_i'(0)^2\nabla\cdot(A_i(x)\nabla)-f_i'(0) q_i^0(x)\cdot\nabla\right)-\veL(x).
\]

Then the generalized principal eigenvalue $\lambda_1'(\cbQ_\varepsilon)$ satisfies:
\[
    \min_{x\in[0,L]}\lambda_{1,\upp}\left(\frac{\upd}{\upd t}-\veL(x)\right)\leq 
    \liminf_{\substack{\varepsilon\to 0\\\varepsilon>0}}\lambda_1'(\cbQ_\varepsilon)\leq
    \limsup_{\substack{\varepsilon\to 0\\\varepsilon>0}}\lambda_1'(\cbQ_\varepsilon)\leq
    \min_{x\in[0,L]}\lambda_1(\widetilde{\cbQ}(x)).
\]
\end{prop}

\begin{proof}
    The proof is done in three steps.

    \begin{proof}[Step 1: the pointwise irreducibility of every $\veL^\varepsilon$ can be assumed without loss of generality]
    Assume the result has been proved provided every $\veL^\varepsilon(t,x)$ is irreducible at all $(t,x)\in\clOmper$.

Define
\[
    \veL^\varepsilon:s\in[0,+\infty)\mapsto\veL^\varepsilon+(\upe^s-1)\veo_{N\times N}-(\upe^s-1)\vect{I}.
\]
Obviously, $\veL^\varepsilon(0)=\veL^\varepsilon$ and, for all $s\in(0,+\infty)$, $\veL^\varepsilon(s,t,x)$ is irreducible at 
all $(t,x)\in\clOmper$. Moreover, by virtue of
Propositions \ref{prop:concavity_eigenvalue_z_L}, \ref{prop:max--min_characterization_lambdaz} and \ref{prop:continuity_eigenvalue_L}, 
the periodic principal eigenvalue $\lambda_1'(\varepsilon,s)$ associated with the operator
\[
    \cbQ_{\varepsilon,s}=\partial_t-\diag(f_i(\varepsilon)^2\nabla\cdot(A_i\nabla)-f_i(\varepsilon)q_i^\varepsilon\cdot\nabla)-\veL^\varepsilon(s)
\]
is, as a function of $s$, continuous in $[0,+\infty)$, decreasing in $[0,+\infty)$, strictly concave in $[0,+\infty)$. 

In particular, by decreasing monotonicity, for any $s>0$,
\[
    \liminf_{\substack{\varepsilon\to 0\\\varepsilon>0}}\lambda_1'(\cbQ_{\varepsilon,0})\geq\liminf_{\substack{\varepsilon\to 0\\\varepsilon>0}}\lambda_1'(\cbQ_{\varepsilon,s})\geq \min_{x\in[0,L]}\lambda_{1,\upp}\left(\frac{\upd}{\upd t}-\veL(s,x)\right).
\]
Passing to the limit $s\to 0$ in the right-hand side, which is continuous indeed with respect to $s$, shows the first inequality:
\[
    \min_{x\in[0,L]}\lambda_{1,\upp}\left(\frac{\upd}{\upd t}-\veL(x)\right)\leq 
    \liminf_{\substack{\varepsilon\to 0\\\varepsilon>0}}\lambda_1'(\cbQ_\varepsilon).
\]

Similarly, defining for all $s\in(0,+\infty)$ and $x\in[0,L]$ the operator
\[
    \widetilde{\cbQ}_s(x) = \partial_t-\diag\left(f_i'(0)^2\nabla\cdot(A_i(x)\nabla)-f_i'(0) q_i^0(x)\cdot\nabla\right)-\veL(s,x),
\]
we find by monotonicity
\[
    \limsup_{\substack{\varepsilon\to 0\\\varepsilon>0}}\lambda_1'(\cbQ_{\varepsilon,s})\leq     \min_{x\in[0,L]}\lambda_1(\widetilde{\cbQ}_s(x))\leq\min_{x\in[0,L]}\lambda_1(\widetilde{\cbQ}(x)).
\]
To pass to the limit $s\to 0$ in the left-hand side, we use
the concavity as follows. Let $\gamma>0$ and $s\in(0,1)$. There exists $\varepsilon_{s,\gamma}>0$ such that
\[
    \lambda_1'(\cbQ_{\varepsilon,s})\leq\min_{x\in[0,L]}\lambda_1(\widetilde{\cbQ}(x))+\gamma\quad\text{for all }\varepsilon\in(0,\varepsilon_{s,\gamma}).
\]
Let $\varepsilon\in(0,\varepsilon_{s,\gamma})$. By concavity, using $s<1$,
\[
    \lambda_1'(\varepsilon,2)-\lambda_1'(\varepsilon,1)\leq\lim_{\substack{s''<s\\s''\to s}}\frac{\lambda_1'(\varepsilon,s'')-\lambda_1'(\varepsilon,s)}{s''-s}\leq 0
\]
and for all $s'\in(0,s)$, 
\[
    \lambda_1'(\cbQ_{\varepsilon,s'})\leq\lambda_1'(\cbQ_{\varepsilon,s})+\left(\lim_{\substack{s''<s\\s''\to s}}\frac{\lambda_1'(\varepsilon,s'')-\lambda_1'(\varepsilon,s)}{s''-s}\right)(s'-s)
\]
\textit{i.e.}
\[
    \lambda_1'(\cbQ_{\varepsilon,s'})+\left(\lim_{\substack{s''<s\\s''\to s}}\frac{\lambda_1'(\varepsilon,s'')-\lambda_1'(\varepsilon,s)}{s''-s}\right)(s-s')\leq\lambda_1'(\cbQ_{\varepsilon,s}).
\]
Hence
\[
    \lambda_1'(\cbQ_{\varepsilon,s'})+\left(\lim_{\substack{s''<s\\s''\to s}}\frac{\lambda_1'(\varepsilon,s'')-\lambda_1'(\varepsilon,s)}{s''-s}\right)(s-s')\leq\min_{x\in[0,L]}\lambda_1(\widetilde{\cbQ}(x))+\gamma
\]
and consequently
\[
    \lambda_1'(\cbQ_{\varepsilon,s'})+\left(\lambda_1'(\varepsilon,2)-\lambda_1'(\varepsilon,1)\right)(s-s')\leq\min_{x\in[0,L]}\lambda_1(\widetilde{\cbQ}(x))+\gamma.
\]
Passing to the limit $s'\to 0$, we get:
\[
    \lambda_1'(\cbQ_{\varepsilon,0})+\left(\lambda_1'(\varepsilon,2)-\lambda_1'(\varepsilon,1)\right)s\leq\min_{x\in[0,L]}\lambda_1(\widetilde{\cbQ}(x))+\gamma
\]
and this inequality holds for any $\gamma>0$, $s\in(0,1)$ and $\varepsilon\in(0,\varepsilon_{s,\gamma})$.

There exists a sequence $(\varepsilon_k)_{k\in\N}$ such that, as $k\to+\infty$,
$\varepsilon_k\to 0$ and
\[
    \lambda_1'(\cbQ_{\varepsilon_k,0})\to\limsup_{\substack{\varepsilon\to 0\\\varepsilon>0}}\lambda_1'(\cbQ_{\varepsilon,0}).
\]
Passing to the limit along this sequence, we obtain:
\[
    \limsup_{\substack{\varepsilon\to 0\\\varepsilon>0}}\lambda_1'(\cbQ_{\varepsilon,0})+\liminf_{k\to+\infty}\left(\lambda_1'(\varepsilon_k,2)-\lambda_1'(\varepsilon_k,1)\right)s\leq\min_{x\in[0,L]}\lambda_1(\widetilde{\cbQ}(x))+\gamma.
\]
Finally, passing to the limit $s\to 0$ and $\gamma\to 0$, we obtain the second inequality:
\[
    \limsup_{\substack{\varepsilon\to 0\\\varepsilon>0}}\lambda_1'(\cbQ_{\varepsilon,0})\leq \min_{x\in[0,L]}\lambda_1(\widetilde{\cbQ}(x)).
\]

This ends the proof of this step.
\end{proof}

Consequently, in the following steps, we assume without loss of generality that every $\veL^\varepsilon(t,x)$ is indeed 
irreducible at all $(t,x)\in\clOmper$. 

In order to ease the reading, we assume without loss of generality (up to a spatial translation) that the minimum of $x\mapsto\lambda_{1,\upp}\left(\frac{\upd}{\upd t}-\veL(x)\right)$ is attained at $x=0$.

\begin{proof}[Step 2: $\liminf_{\varepsilon\to 0}\lambda_1'(\cbQ_\varepsilon)\geq \lambda_{1,\upp}\left(\frac{\upd}{\upd t}-\veL(0)\right)$]
    For each $\varepsilon>0$, let $\veu_\varepsilon$ be the space-time periodic generalized principal eigenfunction 
    associated with $\lambda_1'(\cbQ_\varepsilon)$ and normalized by 
    \[
        \max_{i\in[N]}\max_{\clOmper}u_{i,\varepsilon}=1.
    \]
    Let $(t_\varepsilon,x_\varepsilon)\in\clOmper$ such that $\max_{i\in[N]}u_{i,\varepsilon}(t_\varepsilon,x_\varepsilon)=1$. 
    By compactness, there exists a sequence $(\varepsilon_k)_{k\in\N}$ that converges to $0$ as $k\to\infty$ and such that:
    \begin{enumerate}
        \item $(t_k,x_k)=(t_{\varepsilon_k},x_{\varepsilon_k})$ converges in $\clOmper$ to a limit $(t_\infty,x_\infty)$;
        \item $\lambda_k=\lambda_1'(\cbQ_{\varepsilon_k})$ converges in 
        $[-\lambda_{\upPF}(\overline{\veL}),-\lambda_{\upPF}(\underline{\veL})]$ to a limit $\lambda_\infty$ 
        (see Corollary \ref{cor:comparison_eigenvalues_underline_overline_L}).
    \end{enumerate}
    
    Let $g_i:\varepsilon\in(0,+\infty)\mapsto f_i(\varepsilon)/\varepsilon$. For each $k\in\N$, define 
    \[
        \vev_k:(t,x)\mapsto\veu_{\varepsilon_k}(t,\varepsilon_k x+x_k),
    \]
    \[
        (A_{i,k})_{i\in[N]}:(t,x)\mapsto \left(g_i(\varepsilon_k)^2 A_i(t,\varepsilon_k x+x_k)\right)_{i\in[N]},
    \]
    \[
        (q_{i,k})_{i\in[N]}:(t,x)\mapsto \left(g_i(\varepsilon_k)q_i^{\varepsilon_k}(t,\varepsilon_k x+x_k)\right)_{i\in[N]},
    \]
    \[
        \veL_k:(t,x)\mapsto \veL^{\varepsilon_k}(t,\varepsilon_k x+x_k),
    \]
    \[
        \cbQ_k = \partial_t - \diag\left(\nabla\cdot(A_{i,k}\nabla)-q_{i,k}\cdot\nabla\right))-\veL_k.
    \]
    The function $\vev_k$ is space-time periodic with periodicity cell $[0,T]\times[0,L/\varepsilon_k]$. It satisfies
    \[
        \cbQ_k\vev_k = \lambda_k\vev_k\quad\text{in }\R\times\R^n.
    \]
    
    Denote, for each $i\in[N]$, 
    \[
        \sigma_i=f_i'(0)=\lim_{\varepsilon\to 0}g_i(\varepsilon)\geq 0.
    \]
    Let $I\subset[N]$ such that $\sigma_i>0$ if and only if $i\in I$. By assumption on $\vect{f}$, $I$ is nonempty.

    The coefficients of the above system, $(A_{i,k})$, $(q_{i,k})$ and $\veL_k$,
    converge locally uniformly to their respective space homogeneous $[0,T]$-time periodic limits
    $(\sigma_i^2 A_i(x_\infty))_{i\in[N]}$, $(\sigma_i q_i^0(x_\infty))_{i\in[N]}$, $\veL(x_\infty)$.
    
    Note that the limiting operator, 
    \[
        \cbQ_\infty = \partial_t -\diag(\nabla\cdot(\sigma_i^2 A_i(x_\infty)\nabla)-\sigma_i q_i^0(x_\infty)\cdot\nabla)-\veL(x_\infty),
    \]
    is a non-degenerate parabolic operator only in the special case $I=[N]$ and is in general a 
    parabolic--ordinary operator. Hence the standard parabolic estimates cannot be used to pass 
    directly to the limit in the solution $\vev_k$.
    
    Since $\vev_k\leq\veo$ for all $k\in\N$, by virtue of the Banach--Alaoglu theorem, the sequence $(\vev_k)_{k\in\N}$ 
    converges up to extraction to a weak-$\star$ limit $\vev_\infty\in\mathcal{L}^\infty_\upp([0,T],\mathcal{L}^\infty(\R^n))$. 
    Up to other extractions, by standard parabolic estimates \cite{Lieberman_2005}, 
    all components with index $i\in I$ (there is at least one)
    converge in a Sobolev stronger sense that preserves the fact that they are weak solutions.
    We introduce, for every $k\leq +\infty$, the adjoint operator:
    \[
        \cbQ_k^\star = -\partial_t -\diag(\nabla\cdot(A_{i,k}\nabla)+q_{i,k}\cdot\nabla+\nabla\cdot q_{i,k})-\veL_k^\upT.
    \]
    By passing to the limit in 
    \[
        \lambda_k\int_{[0,T]\times\R^n}\vect{\varphi}^\upT\vev_k=\int_{[0,T]\times\R^n}\vev_k^\upT(\cbQ_k^\star\vect{\varphi}),
    \]
    for any test function $\vect{\varphi}\in\caC^1_\upp(\R,\caC^2_c(\R^n,[\vez,\vei)))$, and then integrating 
    by parts in space-time, we find that:
    \begin{itemize}
        \item if $i\in I$, then $v_{i,\infty}$ is a 
        bounded weak solution of a parabolic equation with bounded continuous 
        coefficients and bounded forcing term, whence it benefits from bootstrapped parabolic regularity, 
        and in particular it is continuous;
        \item if $i\in [N]\backslash I$, then $v_{i,\infty}$ is a 
        bounded weak solution of an ordinary differential equation with bounded 
        continuous coefficients and bounded forcing term, whence it is actually 
        Lipschitz-continuous.        
    \end{itemize}
    Consequently, $\vev_\infty$ is continuous.   
    Iterating, we find that it is actually a classical solution, nonnegative, time periodic 
    and globally bounded by $\veo$, of 
    \[
        \cbQ_\infty\vev_\infty=\lambda_\infty\vev_\infty\quad\text{in }\R\times\R^n.
    \]

    Let us show now that it is a nonzero solution. Actually, let us show that it is positive.
    Up to permutations, we can assume that there exists $N_0\in[N]$ such that $I=[N_0]$. Since 
    \[
        \veL_\infty\geq 
        \begin{pmatrix}
            \left(l_{i,j,\infty}\right)_{i,j\in[N_0]} & 0 \\
            0 & \left(l_{i,j,\infty}\right)_{i,j\in[N]\backslash[N_0]}
        \end{pmatrix}
    \]
    and both diagonal blocks in the right-hand side are fully coupled (\textit{cf.} Step 1),
    we can apply the comparison principle for cooperative systems of parabolic partial differential equations
    of Proposition \ref{prop:strong_maximum_principle} on $(v_{i,\infty})_{i\in[N_0]}$ 
    and a form of comparison principle for cooperative systems of ordinary differential equations
    \cite{Slomczynski_1993} on $(v_{i,\infty})_{i\in[N]\backslash[N_0]}$ to deduce
    that they are both either positive or zero. 
    Subsequently, the full coupling of $\veL_\infty$ (\textit{cf.} Step 1) implies that
    $\vev_\infty$ is either positive or zero.
    The standard parabolic estimates \cite{Lieberman_2005} on $(v_{i,k})_{i\in[N_0]}$ imply that if 
    \[
        \inf_{k\in\N}\max_{i\in[N_0]}v_{i,k}(t_\infty,0)>0,
    \]
    then $\max_{i\in[N_0]}v_{i,\infty}(t_\infty,0)>0$ and subsequently $\vev_\infty$ is positive. 
    If, on the contrary, 
    \[
        \inf_{k\in\N}\max_{i\in[N_0]}v_{i,k}(t_\infty,0)=0,
    \]
    then by the parabolic strong comparison principle applied to each uniformly parabolic operator $\cbQ_k$,
    it actually means that, up to extraction, $(v_{i,k})_{i\in[N_0]}$ converges locally uniformly
    to $\vez$ as $k\to+\infty$. Due to the normalization $\max_{j\in[N]}v_{j,k}(t_k,0)=1$,
    there exists $\overline{j}\in[N]\backslash[N_0]$ such that, up to extraction, 
    \[
        v_{\overline{j},k}(t_\infty,0)\to 1\quad\text{as }k\to+\infty.
    \]
    Moreover, due to the full coupling of $\veL(x_\infty)$, there exists $\overline{i}\in[N_0]$ such that 
    \[
        l_{\overline{i},\overline{j}}(t_\infty,x_\infty)>0.
    \]
    Evaluating at $(t,x)=(t_\infty,0)$ and passing to the limit $k\to+\infty$ in the equality
    \[
        \partial_t v_{\overline{i},k}-\nabla\cdot(A_{\overline{i},k}\nabla v_{\overline{i},k})+q_{\overline{i},k}\cdot\nabla v_{\overline{i},k}-\sum_{j\in[N_0]}l_{\overline{i},j,k}v_{j,k}-\lambda_k v_{\overline{i},k}=\sum_{j\in[N]\backslash[N_0]}l_{\overline{i},j,k}v_{j,k},
    \]
    the left-hand side converges to $0$ (by standard parabolic regularity estimates)
    whereas the right-hand side admits a limit inferior bounded below by 
    $l_{\overline{i},\overline{j}}(t_\infty,x_\infty)>0$. This contradiction means that 
    \[
        \inf_{k\in\N}\max_{i\in[N_0]}v_{i,k}(t_\infty,0)>0,
    \]
    whence $\vev_\infty$ is positive.
    
    Now, consider $\vev\in\caC^1_\upp(\R,\R^N)$ the time periodic principal eigenfunction satisfying
    \[
        \vev'=\veL(x_\infty)\vev+\lambda_{1,\upp}\left(\frac{\upd}{\upd t}-\veL(x_\infty)\right)\vev,\quad\min_{i\in[N]}\min_{[0,T]}v_i=1.
    \]
    Since it is spatially homogeneous, it is also a solution of 
    \[
        \cbQ_\infty\vev=\lambda_{1,\upp}\left(\frac{\upd}{\upd t}-\veL(x_\infty)\right)\vev.
    \]

    A comparison argument (distinguishing, as before, the parabolic part and the ordinary part 
    of the operator $\cbQ_\infty$) shows that if 
    $\lambda_{1,\upp}\left(\frac{\upd}{\upd t}-\veL(x_\infty)\right)\geq\lambda_\infty$, 
    then $\vev$ and $\vev_\infty$ are proportional, and then the preceding 
    inequality is an equality. In other words, 
    \[
        \lambda_\infty\geq \lambda_{1,\upp}\left(\frac{\upd}{\upd t}-\veL(x_\infty)\right).
    \]

    Since this applies to any accumulation point of $(x_\varepsilon,\lambda_1'(\cbQ_{\varepsilon}))_{\varepsilon>0}$, the conclusion
    of this step follows:
    \[
        \liminf_{\varepsilon\to 0}\lambda_1'(\cbQ_\varepsilon)\geq \min_{x\in[0,L]}\lambda_{1,\upp}\left(\frac{\upd}{\upd t}-\veL(x)\right)=\lambda_{1,\upp}\left(\frac{\upd}{\upd t}-\veL(0)\right).
    \]
\end{proof}

\begin{proof}[Step 3: $\limsup_{\varepsilon\to 0}\lambda_1'(\cbQ_\varepsilon)\leq     \min_{x\in[0,L]}\lambda_1(\widetilde{\cbQ}(x))$]    
    Let us fix arbitrarily $x_0\in[0,L]$ and prove that 
    \[
        \limsup_{\varepsilon\to 0}\lambda_1'(\cbQ_\varepsilon)\leq\lambda_1(\widetilde{\cbQ}(x_0)).
    \]

    We repeat the construction of Step 2 with the (important) modification $(t_\varepsilon,x_\varepsilon)=(0,x_0)$ 
    for each $\varepsilon>0$. 
    We use, for each $k\in\N$ and up to $k=+\infty$ when it makes sense, the same notations 
    $\lambda_k$, $t_k$, $x_k$, $\vev_k$, $\sigma_i$, $I$, $A_{i,k}$, $q_{i,k}$, $\veL_k$, $\cbQ_k$.
    With the new spatial shift, $\cbQ_\infty=\widetilde{\cbQ}(x_0)$.

    Since $(0,x_0)\in\clOmper$ is \textit{a priori} not in $\operatorname{argmax}(\max_{i\in[N]}v_{i,k})$, it is now
    unclear if $\vev_k$ remains positive at the limit $k\to+\infty$. Thus, instead of directly passing to the limit, 
    we are going to construct a sub-solution appropriate for every finite $k$.

    Let $R>0$ and, for each $k\in\N$, define $\mu_{R,k}=\lambda_{1,\upDir}(\cbQ_k,B(x_0,R))$. The solution
    $\underline{\vev}_k\in\caC^{1,2}_{t-\upp}(\R\times B(x_0,R),\R^N)\cap\caC^{0,1}_{t-\upp}(\R\times\overline{B(x_0,R)},\R^N)$
    of the Dirichlet periodic--parabolic problem:
    \[
        \begin{cases}
            \cbQ_k\underline{\vev}_k=\mu_{R,k}\underline{\vev}_k & \text{in }\R\times B(x_0,R) \\
            \underline{\vev}_k=\vez & \text{on }\R\times\partial B(x_0,R) \\
            \displaystyle\max_{i\in[N]}\max_{(t,x)\in[0,T]\times\overline{B(x_0,R)}}\underline{v}_{i,k}(t,x)=1
        \end{cases}
    \]
    can be multiplied by an appropriately small positive constant and extended in
    $\R\times\R^n$ by setting $\underline{\vev}_k=\vez$ in $\R\times(\R^n\backslash B(x_0,R))$ in
    order to show, thanks to a standard comparison argument, that 
    \begin{equation}\label{eq:vanishing_diffusion_upper_estimate_eigenvalue}
        \lambda_1'(\cbQ_k)\leq \mu_{R,k}.
    \end{equation}

    Now, let us verify that $\mu_{R,k}\to\lambda_{1,\upDir}(\widetilde{\cbQ}(x_0),B(x_0,R))$ as $k\to+\infty$.
    We point out that when $\min_{i\in[N]}\sigma_i>0$, this convergence is obvious, by locally uniform convergence of 
    the coefficients $(A_{i,k})_{i\in[N]}$, $(q_{i,k})_{i\in[N]}$ and $\veL_k$ and by continuity of the Dirichlet 
    principal eigenvalue. Hence we only have to verify that this continuity remains true when some, but not all, equations of the system become ordinary differential equations.

    Repeating the compactness and bootstrap argument of Step 2, we obtain after passing to the limit a
    classical solution $\vev_\infty$ of the Dirichlet periodic--parabolic problem:
    \[
        \begin{cases}
            \widetilde{\cbQ}(x_0)\vev_\infty=\mu_{R,\infty}\vev_\infty & \text{in }\R\times B(x_0,R) \\
            \vev_\infty=\vez & \text{on }\R\times\partial B(x_0,R) \\
            \displaystyle\max_{i\in[N]}\max_{(t,x)\in[0,T]\times\overline{B(x_0,R)}}v_{i,\infty}(t,x)=1,
        \end{cases}
    \]
    where $\mu_{R,\infty}$ is an accumulation point of $(\mu_{R,k})_{k\in\N}$. To prove by a comparison argument the uniqueness of the Dirichlet periodic--parabolic principal
    eigenpair, so that $\mu_{R,\infty}=\lambda_{1,\upDir}(\widetilde{\cbQ}(x_0),B(x_0,R))$, we need
    to control the behavior of the solution close to $\partial B(x_0,R)$. More precisely,
    it suffices to prove that for every Dirichlet principal eigenfunction $\vev$, there exists
    $C>0$ such that, for each $i\in[N]$,
    \begin{equation}\label{eq:vanishing_diffusion_sub-solution_Lipschitz_Hopf}
        0<\frac{1}{C}\leq\frac{v_i(t,x)}{||x-x_0|-R|}\leq C\quad\text{for all }(t,x)\in[0,T]\times B(x_0,R).
    \end{equation}

    For each $i\in I$, \eqref{eq:vanishing_diffusion_sub-solution_Lipschitz_Hopf} follows directly
    from the parabolic regularity up to the boundary and the Hopf lemma. It remains to prove
    \eqref{eq:vanishing_diffusion_sub-solution_Lipschitz_Hopf} for each $i\in [N]\backslash I$.

    As in Step 2, assume up to permutations that $I=[N_0]$ with $N_0\in[N]$. If $N_0=N$, then 
    \eqref{eq:vanishing_diffusion_sub-solution_Lipschitz_Hopf} is true for all $i\in[N]$.
    If $N_0<N$, then $1\leq N_0\leq N-1$ and we can write $\veL(x_0)$ in block form:
    \[
        \veL(t,x_0)=
        \begin{pmatrix}
            \vect{A}(t) & \vect{B}(t) \\
            \vect{C}(t) & \vect{D}(t)
        \end{pmatrix}
        \quad\text{for all }t\in[0,T]
    \]
    with 
    \[
        \vect{A}(t)\in\R^{N_0\times N_0},\ \vect{B}(t)\in\R^{N_0\times (N-N_0)},\ \vect{C}(t)\in\R^{(N-N_0) \times N_0},\ \vect{D}(t)\in\R^{(N-N_0)\times (N-N_0)}.
    \] 
    By Step 1, we can assume without loss of generality that, for all $t\in[0,T]$, $\vect{A}(t)$ is essentially positive, $\vect{B}(t)$ 
    is positive, $\vect{C}(t)$ is positive and $\vect{D}(t)$ is essentially positive. Hence, decomposing the principal eigenfunction $\vev$ as 
    \[
        \vev=
        \begin{pmatrix}
            \vew \\
            \widetilde{\vew}
        \end{pmatrix}
        \quad\text{with }\vew\in\R^{N_0},\ \widetilde{\vew}\in\R^{N-N_0},
    \]
    we find
    \[
        \partial_t\widetilde{\vew}=\vect{D}\widetilde{\vew}+\vect{C}\vew\quad\text{in }[0,T]\times B(x_0,R)
    \]
    On one hand, this immediately leads to the Lipschitz-continuity of $\widetilde{\vew}$ at
    the boundary of the spatial domain, namely the upper estimate of \eqref{eq:vanishing_diffusion_sub-solution_Lipschitz_Hopf} for each $i\in[N]\backslash[N_0]$.
    On the other hand, using the inequality
    \[
        \partial_t\widetilde{\vew}=\vect{D}\widetilde{\vew}+\vect{C}\vew\geq\diag(d_{i,i})_{i\in[N-N_0]}\widetilde{\vew}+\vect{C}\vew \quad\text{in }[0,T]\times B(x_0,R)
    \]
    multiplied by 
    \[
        \vect{E}:t\mapsto\diag\left(\upe^{-\int_0^t d_{i,i}(\tau)\upd \tau}\right)_{i\in[N-N_0]}
    \]
    and integrated in $[0,t]$ for an arbitrary $t\in(0,T]$, we obtain for all $x\in B(x_0,R)$
    \begin{align*}
        \widetilde{\vew}(t,x) & \geq\vect{E}(t)^{-1}\widetilde{\vew}(0,x)+\vect{E}(t)^{-1}\int_0^t\vect{E}(t')\vect{C}(t')\vew(t',x)\upd t' \\
        & \geq \vect{E}(t)^{-1}\int_0^t\vect{E}(t')\vect{C}(t')\vew(t',x)\upd t'
    \end{align*}
    Since $\vew$ satisfies \eqref{eq:vanishing_diffusion_sub-solution_Lipschitz_Hopf} component-wise,
    it follows that, up to increasing $C$, for each $i\in[N]\backslash[N_0]$,
    \[
        0<\frac{1}{C}\leq\frac{\widetilde{w}_i(t,x)}{||x-x_0|-R|}\quad\text{for all }(t,x)\in[0,T]\times B(x_0,R).
    \]
    Therefore for each $i\in[N]\backslash[N_0]$ the lower estimate of \eqref{eq:vanishing_diffusion_sub-solution_Lipschitz_Hopf} is also satisfied. 

    Hence \eqref{eq:vanishing_diffusion_sub-solution_Lipschitz_Hopf} is satisfied for each $i\in[N]$.
    
    Now, take two principal eigenpairs $(\lambda,\vev)$ and $(\widetilde{\lambda},\widetilde{\vev})$, assume
    for instance $\lambda\geq\widetilde{\lambda}$, and let us prove by comparison that $\lambda=\widetilde{\lambda}$
    and $\vev=\kappa \widetilde{\vev}$ for some $\kappa>0$.
    By \eqref{eq:vanishing_diffusion_sub-solution_Lipschitz_Hopf}, there exists $\kappa>0$ such that
    $\kappa\widetilde{\vev}\leq\vev$. Hence $\vew=\vev-\kappa\widetilde{\vev}$ is nonnegative and satisfies:
    \[
        \widetilde{\cbQ}(x_0)\vew = \widetilde{\lambda}\vew+(\lambda-\widetilde{\lambda})\vev \geq \widetilde{\lambda}\vew\quad\text{in }\R\times B(x_0,R).
    \]
    Verifying that the Hopf-type lower estimate of \eqref{eq:vanishing_diffusion_sub-solution_Lipschitz_Hopf} remains true for 
    super-solutions (the argument is the same), we find that $\vew$ is either positive or zero. 
    Increasing continuously $\kappa$, we deduce from the
    comparison principle an optimal $\kappa^\star$ such that $\vev=\kappa^\star\widetilde{\vev}$.
    It follows subsequently, from the system satisfied by $\vew$, that $\lambda=\widetilde{\lambda}$.
    
    Hence the Dirichlet principal eigenvalue is unique indeed and $\mu_{R,k}\to\lambda_{1,\upDir}(\widetilde{\cbQ}(x_0))$.
    Passing to the limit $k\to+\infty$ in \eqref{eq:vanishing_diffusion_upper_estimate_eigenvalue},
    we deduce:
    \[
        \limsup_{\varepsilon\to 0}\lambda_1'(\cbQ_\varepsilon)\leq\lambda_{1,\upDir}(\widetilde{\cbQ}(x_0),B(x_0,R)).
    \]

    Since this is true for all $R>0$, it only remains to verify that
    \[
            \lambda_{1,\upDir}(\widetilde{\cbQ}(x_0),B(x_0,R))\to \lambda_1(\widetilde{\cbQ}(x_0))\quad\text{as }R\to+\infty.
    \]
    Again, this is classical if $N_0=N$ and only requires work if $N_0<N$. 
    
    In fact, it can be verified that, in order to adapt the proof of Proposition 
    \ref{prop:eigenvalue_limit_Dirichlet} to an operator which is degenerate parabolic
    but still satisfies the strong comparison principle, we only need to be able to extract a 
    convergent subsequence of Dirichlet principal eigenfunctions $\vev_R$ as $R\to+\infty$ with 
    nonzero limit.
    In view of the compactness procedure detailed in Step 2, we actually only need
    to verify that the pointwise normalization $\max_{i\in[N]}v_{i,R}(0,x_0)=1$ implies that,
    for any $R_0>0$, there exists $C>0$ such that, for any $R>R_0$, 
    \[
        \vev_R\leq C\veo\quad\text{in }[0,T]\times B(x_0,R_0).
    \]
    As a matter of fact, such local bounds can be derived from parabolic interior regularity
    estimates \cite{Lieberman_2005} applied to components with index $i\in[N_0]$ 
    combined with standard regularity results for parameterized
    ordinary differential equations applied to components with index $i\in[N]\backslash[N_0]$. 
    This is standard and not detailed.
\end{proof}

The proof is now ended.

\end{proof}

\begin{rem}\label{rem:counter-example_vanishing_diffusion}
    Contrarily to what was claimed by Nadin in his work on the scalar case \cite[Theorem 3.6]{Nadin_2007},
    the convergence to $\min_{x\in[0,L]}\lambda_{1,\upp}\left(\frac{\upd}{\upd t}-\veL(x)\right)$ cannot be true in general,
    even if $N=1$.
    Indeed, if:
    \begin{itemize}
        \item $f_i(\varepsilon)=\varepsilon$ for each $i\in[N]$;
        \item the coefficients of $\cbQ_\varepsilon$ do not depend on time;
        \item $\veL$ is symmetric;
        \item and there exists $Q\in\caC^2(\R^n,\R)$ such that $\int_{[0,L]}\nabla Q=0$ and 
        \[
            (\varepsilon^2 A_i)^{-1}(\varepsilon q_i^\varepsilon)=\nabla Q\quad\text{for each }i\in[N];
        \]
    \end{itemize}
    then by the variational formula of Theorem \ref{thm:lambdaz_time_homogeneous}, 
    \[
        \lambda_1(\cbQ_{\varepsilon})=\lambda_1'(\cbQ_{\varepsilon})=\lambda_1'(-\varepsilon^2\diag(\nabla\cdot(A_i\nabla))-\veL_Q)
    \]
    with
    \begin{align*}
        \veL_Q & = \veL + \diag\left(\frac12\nabla\cdot(\varepsilon^2 A_i\nabla Q)-\frac14\nabla Q\cdot \varepsilon^2 A_i\nabla Q\right) \\
        & = \veL + \diag\left(\frac12\varepsilon\nabla\cdot(q_i)-\frac14 A_i^{-1}q_i\cdot q_i\right).
    \end{align*}
    If the convergence of $\lambda_1'(\cbQ_\varepsilon)$ to $\min_{x\in[0,L]}\lambda_{1,\upp}\left(\frac{\upd}{\upd t}-\veL(x\right))$ was true,
    then by passing to the limit in the equality $\lambda_1'(\cbQ_{\varepsilon})=\lambda_1'(-\varepsilon^2\diag(\nabla\cdot(A_i\nabla))-\veL_Q)$, 
    we would obtain
    \[
        \min_{x\in[0,L]}\lambda_{1,\upp}\left(\frac{\upd}{\upd t}-\veL(x)\right)=\min_{x\in[0,L]}\lambda_{1,\upp}\left(\frac{\upd}{\upd t}-\veL(x)+\frac14\diag(|A_i^{-1/2}q_i|^2)(x)\right),
    \]
    which yields an obvious contradiction after a careful choice of $\veL$, $(A_i)_{i\in[N]}$, $(q_i)_{i\in[N]}$
    -- for instance, in spatial dimension $n=1$, 
    \[
        A_i=1,\quad q_i:x_1\mapsto\cos\left(\frac{2\pi}{L_1}x_1\right),\quad \veL:x_1\mapsto-\frac14\sin\left(\frac{2\pi}{L_1}x_1\right)^2\vect{I}+\vect{M}
    \]
    with $\vect{M}$ the discrete Laplacian defined in \eqref{eq:discrete_Lap}.

    The mistake in Nadin's proof \cite{Nadin_2007} can be corrected with the additional assumption
    $\lambda_1(\widetilde{\cbQ})=\lambda_1'(\widetilde{\cbQ})$. In this sense, our result and its proof 
    provide as a by-product a correction of the scalar counterpart in \cite{Nadin_2007}.
\end{rem}

\subsubsection{Large diffusion: proof of Theorem \ref{thm:limit_eigenvalue_large_diffusion}}

We now prove Theorem \ref{thm:limit_eigenvalue_large_diffusion}.

\begin{prop}\label{prop:limit_eigenvalue_large_diffusion}
Let 
\[
    \left((\langle A_i\rangle,\langle q_i\rangle)_{i\in[N]},\langle\veL\rangle\right):t\mapsto
    \frac{1}{|[0,L]|}\int_{[0,L]}\left((A_i,q_i)_{i\in[N]},\veL\right)(t,x)\upd x
\]
and, for all $\vect{d}\in(\vez,\vei)$, let $\cbQ_{\vect{d}}$ be the operator $\cbQ$ with $(A_i)_{i\in[N]}$ replaced by $(d_i A_i)_{i\in[N]}$.

Then 
\[
    \lim_{\min_{i\in[N]}d_i\to+\infty}\lambda_{1,\upp}(\cbQ_{\vect{d}})=\lambda_{1,\upp}\left(\partial_t-\langle\veL\rangle\right).
\]
\end{prop}

\begin{proof}
Let $\vect{d}\gg\vez$ and $\veu_{\vect{d}}$ be the periodic principal eigenfunction associated with $\lambda_1'(\cbQ_{\vect{d}})$ and
normalized by 
\[
    \frac{1}{T|[0,L]|}\int_{\clOmper}|\veu_{\vect{d}}|^2=1.
\]

Below, we assume $\caC^1$ space regularity of the coefficients $(q_i)_{i\in[N]}$.
The proof in the general case with mere H\"{o}lder-continuity is not detailed -- it can be deduced by a standard 
regularization procedure and the continuity of periodic principal eigenvalues.

Multiplying $(\cbQ_{\vect{d}}\veu_{\vect{d}})_i-\lambda_1'(\cbQ_{\vect{d}})u_{\vect{d},i}$ by $u_{\vect{d},i}$ and
then integrating over $\clOmper$, we find for each $i\in[N]$:
\[
    d_i\int_{\clOmper}\nabla u_{\vect{d},i}\cdot A_i\nabla u_{\vect{d},i} =\int_{\clOmper}\left(\left(\frac{\nabla\cdot q_i}{2}+\lambda_1'(\cbQ_{\vect{d}})\right)u_{\vect{d},i}^2+\sum_{i=1}^N l_{i,j}u_{\vect{d},i}u_{\vect{d},j}\right).
\]
Recall from Corollary \ref{cor:comparison_eigenvalues_underline_overline_L} the estimate
$-\lambda_{\upPF}(\overline{\veL})\leq\lambda_1'(\cbQ_{\vect{d}})\leq-\lambda_{\upPF}(\underline{\veL})$,
whence, by the Young inequality $|u_{\vect{d},i}u_{\vect{d},j}|\leq\frac12\left(|u_{\vect{d},i}|^2+|u_{\vect{d},j}|^2\right)$,
there exists a constant $K>0$ independent of $\vect{d}$ such that
\[
    0\leq \sum_{i=1}^N\int_{\clOmper}\nabla u_{\vect{d},i}\cdot A_i\nabla u_{\vect{d},i}\leq \frac{K}{\displaystyle\min_{i\in[N]}d_i}.
\]
Consequently, 
\[
    \frac{1}{T}\int_0^T\left(\sum_{i=1}^N\int_{[0,L]}\nabla u_{\vect{d},i}\cdot A_i\nabla u_{\vect{d},i}\right)\to 0\quad\text{as }\min_{i\in[N]}d_i\to+\infty.
\]

Let $\langle \veu_{\vect{d}}\rangle:t\mapsto\frac{1}{|[0,L]|}\int_{[0,L]}\veu_{\vect{d}}(t,x)\upd x$ and
$\vev_{\vect{d}}=\veu_{\vect{d}}-\langle\veu_{\vect{d}}\rangle$. By the Poincaré inequality, there exists another constant $K'>0$ such that
\[
    \sum_{i=1}^N\int_{[0,L]}\nabla u_{\vect{d},i}\cdot A_i\nabla u_{\vect{d},i}=\sum_{i=1}^N\int_{[0,L]}\nabla v_{\vect{d},i}\cdot A_i\nabla v_{\vect{d},i}\geq K'\int_{[0,L]}|\vev_{\vect{d}}|^2.
\]
Since the average in $[0,T]$ of the nonnegative function on the left-hand side converges to $0$ as $\min_{i\in[N]}d_i\to+\infty$,
so does the average in $[0,T]$ of $\int_{[0,L]}|\vev_{\vect{d}}|^2$, whence $\vev_{\vect{d}}$ itself converges to $\vez$ almost everywhere
(up to extraction of a subsequence).

Since, for each $i\in[N]$,
\begin{align*}
    \int_0^T\langle \veu_{\vect{d}}\rangle_i & =|[0,L]|^{-1}\int_{\clOmper}u_{\vect{d},i} \\
    & \leq \left(\frac{T}{|[0,L]|}\right)^{1/2}\left(\int_{\clOmper}u_{\vect{d},i}^2\right)^{1/2} \\
    & \leq \left(\frac{T}{|[0,L]|}\right)^{1/2}\left(\int_{\clOmper}\sum_{i=1}^N u_{\vect{d},i}^2\right)^{1/2}=T,
\end{align*}
$\langle\veu_{\vect{d}}\rangle_i$ is bounded in $\mathcal{L}^1([0,T])$ uniformly with respect to $\vect{d}$.

Integrating $\cbQ_{\vect{d}}\veu_{\vect{d}}=\lambda_1'(\cbQ_{\vect{d}})\veu_{\vect{d}}$ over $[0,L]$ and dividing by $|[0,L]|$, we find:
\begin{align*}
    \partial_t\langle \veu_{\vect{d}}\rangle & = |[0,L]|^{-1}\int_{[0,L]}\left(\diag(\nabla\cdot q_i)\veu_{\vect{d}}\right)+|[0,L]|^{-1}\int_{[0,L]}\left(\veL\veu_{\vect{d}}\right)+\lambda_1'(\cbQ_{\vect{d}})\langle \veu_{\vect{d}}\rangle \\
    & = |[0,L]|^{-1}\int_{[0,L]}\diag(\nabla\cdot q_i)\langle \veu_{\vect{d}}\rangle+|[0,L]|^{-1}\int_{[0,L]}\veL\langle \veu_{\vect{d}}\rangle+\lambda_1'(\cbQ_{\vect{d}})\langle \veu_{\vect{d}}\rangle \\
    & \quad + |[0,L]|^{-1}\int_{[0,L]}\left(\left(\diag(\nabla\cdot q_i)+\veL\right)\vev_{\vect{d}}\right) \\
    & = \langle\veL\rangle\langle\veu_{\vect{d}}\rangle+\lambda_1'(\cbQ_{\vect{d}})\langle \veu_{\vect{d}}\rangle + |[0,L]|^{-1}\int_{[0,L]}\left(\left(\diag(\nabla\cdot q_i)+\veL\right)\vev_{\vect{d}}\right).
\end{align*}

Now, by the discrete and continuous Cauchy--Schwarz inequalities, denoting $\|\cdot\|$ the norm in $\mathcal{L}^\infty(\clOmper,\R)$,
\begin{align*}
    \left|\int_{[0,L]}\left(\left(\nabla\cdot q_i\right)v_{\vect{d},i}+\sum_{j=1}^N l_{i,j} v_{\vect{d},j}\right)\right|
    & \leq \left(\|\nabla\cdot q_i\|+\max_{j\in[N]}\|l_{i,j}\|\right)\int_{[0,L]}\sum_{j=1}^N|v_{\vect{d},j}| \\
    & \leq \left(\|\nabla\cdot q_i\|+\max_{j\in[N]}\|l_{i,j}\|\right)\int_{[0,L]}\sqrt{N}\left(\sum_{j=1}^N|v_{\vect{d},j}|^2\right)^{\frac12} \\
    & \leq \left(\|\nabla\cdot q_i\|+\max_{j\in[N]}\|l_{i,j}\|\right)\sqrt{N|[0,L]|}\int_{[0,L]}|\vev_{\vect{d}}|^2.
\end{align*}

Since $\int_{[0,L]}|\vev_{\vect{d}}|^2$ term converges to $0$ as $\min_{i\in[N]}d_i\to+\infty$, each component of 
$\partial_t\langle\veu_{\vect{d}}\rangle$ is bounded in $\mathcal{L}^1([0,T])$ uniformly with respect to $\vect{d}$. 
Hence each component of $\langle\veu_{\vect{d}}\rangle$ is 
bounded uniformly in $\mathcal{W}^{1,1}([0,T])$, and then via the fundamental theorem of calculus it is bounded uniformly in $\mathcal{L}^\infty([0,T])$, whence it is bounded uniformly in the space of functions of bounded variation $\mathcal{BV}([0,T])$.
By compactness of the embedding $\mathcal W^{1,1}\hookrightarrow \mathcal L^1$, each component of $\langle\veu_{\vect{d}}\rangle$ converges up to extraction in $\mathcal{L}^1([0,T])$. Using the equation
\[
    \partial_t\langle \veu_{\vect{d}}\rangle = \langle\veL\rangle\langle\veu_{\vect{d}}\rangle+\lambda_1'(\cbQ_{\vect{d}})\langle \veu_{\vect{d}}\rangle + |[0,L]|^{-1}\int_{[0,L]}\left(\left(\diag(\nabla\cdot q_i)+\veL\right)\vev_{\vect{d}}\right)
\]
and assuming up to another extraction that $\lambda_1'(\cbQ_{\vect{d}})\to\lambda\in\R$,
so does each component of $\partial_t\langle\veu_{\vect{d}}\rangle$. Denoting by $\veu_\infty$ the limit of $\langle\veu_{\vect{d}}\rangle$, we deduce
that the limit of $\partial_t\langle\veu_{\vect{d}}\rangle$ is, in distributional sense, the derivative of $\veu_\infty$, so that
$\veu_\infty$ satisfies
\[
    \partial_t\veu_\infty=\langle\veL\rangle\veu_\infty+\lambda\veu_\infty\quad\text{in }(\mathcal{L}^\infty)'(\R).
\]
Again by virtue of the fundamental theorem of calculus, each component of $\veu_\infty$ is actually in $\mathcal{L}^\infty([0,T])$, and 
now from the equation it appears that so does each component of $\partial_t\veu_\infty$. Therefore $\veu_\infty$ is in fact 
Lipschitz-continuous, and using again the equation it is $\caC^1$. Since it is periodic, nonnegative (by almost everywhere convergence, 
up to another extraction) and nonzero (if on the contrary it was zero, then $\veu_{\vect{d}}$ would converge to $0$ almost everywhere and 
this would contradict the normalization on $\veu_{\vect{d}}$) and since the operator $\partial_t-\langle\veL\rangle$ is fully coupled in 
$[0,T]$ by \ref{ass:irreducible}, we deduce by uniqueness of the classical solution that
\[
    \lambda=\lambda_{1,\upp}(\partial_t-\langle\veL\rangle).
\]
By uniqueness, the family $(\lambda_1'(\cbQ_{\vect{d}}))_{\vect{d}\gg\vez}$ has a unique accumulation point and thus:
\[
    \lim_{\min_{i\in[N]}d_i\to+\infty}\lambda_1'(\cbQ_{\vect{d}})=\lambda_{1,\upp}(\partial_t-\langle\veL\rangle).
\]
\end{proof}

\begin{rem}\label{rem:correction_Nadin_large_diffusion_lambdaz}
    Contrarily to what was claimed by Nadin in \cite[Theorem 3.6]{Nadin_2007}, the large diffusion limit
    of the family $(\lambda_{1,z})_{z\in\R^n}$ and of the generalized principal eigenvalue 
    $\lambda_1$ cannot be directly deduced from the above proof. Indeed, the large parameter
    $\vect{d}$ appears in the zeroth order term of the operator $\cbQ_z$ and makes the eigenvalue
    $\lambda_{1,z}$ blow-up to $-\infty$ as $\min d_i\to+\infty$, see Corollary 
    \ref{cor:comparison_eigenvalues_underline_overline_L}.
\end{rem}

\begin{rem}\label{rem:counter-example_comparison_small_large_diffusion}
Contrarily to the scalar setting \cite{Nadin_2007} or the special cases where $\partial_t-\veL(x)$ admits a space-time homogeneous
periodic principal eigenfunction, it is in general false that for systems without diffusion and advection, 
spatial average and periodic principal eigenvalue commute, namely
\[
    \lambda_{1,\upp}(\partial_t-\langle\veL\rangle)\neq\langle x\mapsto\lambda_{1,\upp}(\partial_t-\veL(x))\rangle.
\]
In fact, even the inequality
\[
    \lambda_{1,\upp}(\partial_t-\langle\veL\rangle)\geq\min_{x\in[0,L]}\lambda_{1,\upp}(\partial_t-\veL(x))
\]
is false in general, as shown by the following very simple time homogeneous one-dimensional counter-example
\[
    \veL:(t,x)\mapsto
    \begin{cases}
        \begin{pmatrix}1 & 1 \\ 0 & 1\end{pmatrix} & \text{if }x\in[0,L_1/2]+L_1\mathbb{Z}, \\
        \begin{pmatrix}1 & 0 \\ 1 & 1\end{pmatrix} & \text{if }x\in[L_1/2,L_1]+L_1\mathbb{Z}.
    \end{cases}
\]
In a time homogeneous setting, 
\[
    \lambda_{1,\upp}(\partial_t-\langle\veL\rangle)=-\lambda_{\upPF}(\langle\veL\rangle),\quad\min_{x\in[0,L]}\lambda_{1,\upp}(\partial_t-\veL(x))=-\max_{x\in[0,L]}\lambda_{\upPF}(\veL(x)).
\]
With the counter-example above, these two quantities turn out to be respectively $-\frac{3}{2}$ and $-1$: the averaged matrix has a larger
Perron--Frobenius eigenvalue than the matrix at any point in space. In other words, considering for instance the operator 
$\partial_t -d\Delta-\veL$, the limit $d\to 0$ of the periodic principal eigenvalue is larger than the limit $d\to+\infty$.
This is in sharp contrast with the variational formula of Theorem \ref{thm:lambdaz_time_homogeneous}, which indicates a nondecreasing
dependence on $d$ but does not apply here due to the asymmetry of $\veL$. Of course $\veL$ in this counter-example is not continuous 
and therefore does not satisfy \ref{ass:smooth_periodic}; however, any smooth sufficiently precise approximation of $\veL$ will give 
the same conclusion, by continuity of the Perron--Frobenius eigenvalue. As a side result, this counter-example also shows that the variational
formula of Theorem \ref{thm:lambdaz_time_homogeneous} does not hold if only $\langle\veL\rangle$ is symmetric, namely if the pointwise 
symmetry assumption is replaced by an assumption of symmetry on average.

Let us also point out that this counter-example is used to prove Corollary \ref{cor:asymptotics_small_large_spatial_frequency}. More precisely, the map $L_1\mapsto\lambda_{1,\upp}(\cbQ_{T,L_1})$ is
\begin{itemize}
    \item constant if for instance the coefficients of $\cbQ_{T,L_1}$ do not depend on space;
    \item decreasing if for instance the coefficients of $\cbQ_{T,L_1}$ do not depend on time
    and the operator is self-adjoint and spatially heterogeneous, so that the variational formula of Theorem \ref{thm:lambdaz_time_homogeneous} applies and periodic principal eigenfunctions are spatially heterogeneous and therefore have nonzero gradients;
    \item neither if for instance $\cbQ_{T,L_1}=\frac{1}{T}\partial_t-\frac{1}{L_1^2}\Delta-\veL$ with $\veL$ the above counter-example, so that the limit $L_1\to +\infty$ of the periodic principal eigenvalue is larger than the limit $L_1\to 0$.
\end{itemize}
\end{rem}

\subsubsection{Small and large time frequency: proof of Theorem \ref{thm:limits_eigenvalue_time_frequency}}

Now we turn to the proof of Theorem \ref{thm:limits_eigenvalue_time_frequency}. Denoting by $\cbQ_{\omega}$ the operator $\cbQ$ with
$\partial_t$ replaced by $\omega\partial_t$, we first prove the small frequency limit $\omega\to 0$ in
Proposition \ref{prop:small_time_frequency_limit}, then the high frequency one $\omega\to +\infty$ in Proposition 
\ref{prop:high_time_frequency_limit}. 

\begin{prop}\label{prop:small_time_frequency_limit}
For all $z\in\R^n$,
\[
    \lim_{\substack{\omega>0\\\omega\to 0}}\lambda_{1,z}(\cbQ_{\omega})=\frac{1}{T}\int_0^T\lambda_{1,z}\left(-\diag(\nabla\cdot(A_i(t)\nabla)-q_i(t)\cdot\nabla)-\veL(t)\right)\upd t,
\]
\[
    \lim_{\substack{\omega>0\\\omega\to 0}}\lambda_1(\cbQ_{\omega})=\frac{1}{T}\int_0^T\lambda_{1}\left(-\diag(\nabla\cdot(A_i(t)\nabla)-q_i(t)\cdot\nabla)-\veL(t)\right)\upd t,
\]
where, with a slight abuse of notation, for all $t\in[0,T]$,
\[
    ((A_i(t),q_i(t))_{i\in[N]},\veL(t)):x\mapsto((A_i(t,x),q_i(t,x))_{i\in[N]},\veL(t,x)).
\]
\end{prop}

\begin{proof}
It is sufficient to prove only the case $z=0$, since we can deduce the general case for $\lambda_{1,z}$ by applying the result to the operator $\cbQ_z$, and then we can deduce the result for $\lambda_1$ by
applying the same argument as in the proof of Proposition \ref{prop:continuity_eigenvalue_L}, using
Corollary \ref{cor:comparison_eigenvalues_decoupled_equations} and the strict concativity of $z\mapsto\lambda_{1,z}$.

The proof requires two steps.

\begin{proof}[Step 1: the pointwise irreducibility of $\veL$ can be assumed without loss of generality]
Assume the limit has been proved provided $\veL(t,x)$ is irreducible at all $(t,x)\in\clOmper$.

Define
\[
    \veL:s\in[0,+\infty)\mapsto\veL+(\upe^s-1)\veo_{N\times N}-(\upe^s-1)\vect{I}.
\]
Obviously, $\veL(0)=\veL$ and, for all $s\in(0,+\infty)$, $\veL(s,t,x)$ is irreducible at all $(t,x)\in\clOmper$. Moreover, by virtue of
Propositions \ref{prop:concavity_eigenvalue_z_L}, \ref{prop:max--min_characterization_lambdaz} and \ref{prop:continuity_eigenvalue_L}, 
the periodic principal eigenvalue $\lambda_1'(\omega,s)$ associated with the operator
\[
    \cbQ_{\omega,s}=\omega\partial_t-\diag(\nabla\cdot(A_i\nabla)-q_i\cdot\nabla)-\veL(s)
\]
is, as a function of $s$, continuous in $[0,+\infty)$, decreasing in $[0,+\infty)$, strictly concave in $[0,+\infty)$. 
By concavity, for all $\omega\in(0,1]$, the one-sided derivatives of $s\mapsto\lambda_1'(\omega,s)$ are well-defined
and satisfy:
\begin{align*}
    \lambda_1'(\omega,2)-\lambda_1'(\omega,1) & < \lim_{\substack{s'<1\\s'\to 1}}\frac{\lambda_1'(\omega,s')-\lambda_1'(\omega,1)}{s'-1} \\
    & \leq\lim_{\substack{s'<s\\s'\to s}}\frac{\lambda_1'(\omega,s')-\lambda_1'(\omega,s)}{s'-s} \\
    & \leq \lim_{\substack{s'>s\\s'\to s}}\frac{\lambda_1'(\omega,s')-\lambda_1'(\omega,s)}{s'-s} \\
    & \leq \lim_{\substack{s'>0\\s'\to 0}}\frac{\lambda_1'(\omega,s')-\lambda_1'(\omega,0)}{s'} \\
    & \leq 0. 
\end{align*}
By assumption, $\veL(s)$ being irreducible when $s>0$, $\lambda_1'(\omega,1)$ and $\lambda_1'(\omega,2)$ both converge as $\omega\to 0$, 
whence
\[
    \inf_{\omega\in(0,1]}\lambda_1'(\omega,2)-\lambda_1'(\omega,1)>-\infty.
\]
Therefore the family $\left(s\in[0,1]\mapsto\lambda_1'(\omega,s)\right)_{\omega\in(0,1]}$ is uniformly Lipschitz-continuous,
and \textit{a fortiori} equicontinuous. 
By virtue of the Arzel\`{a}--Ascoli theorem, it is relatively compact in $\caC([0,1])$.

Let $\lambda\in\caC([0,1])$ be any accumulation point of the family as $\omega\to 0$. Since we assumed the pointwise convergence
\[
    \lim_{\substack{\omega>0\\\omega\to 0}}\lambda_1'(\cbQ_{\omega,s})=\frac{1}{T}\int_0^T\lambda_1'\left(-\diag(\nabla\cdot(A_i(t)\nabla)-q_i(t)\cdot\nabla)-\veL(s,t)\right)\upd t,
\]
when $s>0$, it follows that $\lambda$ coincides in $(0,1]$ with 
\[
    s\mapsto\frac{1}{T}\int_0^T\lambda_{1,z}\left(-\diag(\nabla\cdot(A_i(t)\nabla)-q_i(t)\cdot\nabla)-\veL(s,t)\right)\upd t.
\]
By continuity of $\lambda$ and of the above function (due to Proposition \ref{prop:continuity_eigenvalue_L}), they also 
coincide at $s=0$. Hence there is a unique accumulation point for the sequence, whence the whole family
$\left(s\mapsto\lambda_1'(\omega,s)\right)_{\omega\in(0,1]}$ converges uniformly to the above function as $\omega\to 0$. 
This implies the pointwise convergence at $s=0$, and this ends the proof of this step.
\end{proof}

In the following step we assume, without loss of generality, that $\veL(t,x)$ is indeed irreducible at all $(t,x)\in\clOmper$.

\begin{proof}[Step 2: the proof in the pointwise irreducible case]
Let $\omega>0$ and $\varepsilon>0$.

Let 
\[
    \lambda:t\mapsto\lambda_1'\left(-\diag(\nabla\cdot(A_i(t)\nabla)-q_i(t)\cdot\nabla)-\veL(t)\right).
\]
By virtue of the pointwise irreducibility of $\veL$, which implies the full coupling of all operators in the family 
\[
    \left(-\diag(\nabla\cdot(A_i(t)\nabla)-q_i(t)\cdot\nabla)-\veL(t)\right)_{t\in[0,T]},
\]
the function $\lambda$ is continuous and periodic. 
Let $\vev:\R\times\R^n\to\R^N$ be the function such that, for any
$t\in[0,T]$, $x\in[0,L]\mapsto\vev(t,x)$ is
the periodic principal eigenfunction of $-\diag(\nabla\cdot(A_i(t)\nabla)-q_i(t)\cdot\nabla)-\veL(t)$,
with an appropriate normalization ensuring the continuity of $\vev$ in time, \textit{e.g.} $\max_{i\in[N]}v_i(t,0)=1$. 

Assuming sufficient time regularity of the coefficients $(A_i)_{i\in[N]}$, $(q_i)_{i\in[N]}$, $\veL$, we deduce from classical
regularity estimates \cite{Lieberman_2005} that $\vev\in\caC^{1,2}_\upp(\R\times\R^n,(\vez,\vei))$.
The proof in the general case with mere H\"{o}lder-continuity in time is not detailed -- it can be deduced by a standard 
regularization procedure and the continuity of periodic principal eigenvalues, \textit{cf.} Step 1.

Since $\vev$ is $\caC^1$ with respect to time, so is $(\ln v_i)_{i\in[N]}$, and therefore $(\partial_t v_i/v_i)_{i\in[N]}$
is globally bounded in $\clOmper$. Hence there exists $K>0$, independent of $\omega$, such that $-K\vev\leq\partial_t \vev\leq K\vev$.
Provided $\omega\leq\frac{\varepsilon}{K}$,
\[
     \left(-\varepsilon+\lambda\right)\vev\leq\cbQ_{\omega}\vev\leq\left(\varepsilon+\lambda\right)\vev
     \quad\text{in }\clOmper.
\]
Let 
\[
    v:t\mapsto\exp\left(\frac{1}{\omega}\left(\frac{t}{T}\int_0^T\lambda(t')\upd t'-\int_0^t\lambda(t')\upd t'\right)\right)
\]
which is positive, periodic and satisfies $\omega v'=\left(\frac{1}{T}\int_0^T\lambda -\lambda\right)v$. Then
\[
     \left(-\varepsilon+\frac{1}{T}\int_0^T\lambda\right)v\vev\leq\omega v'\vev+v\cbQ_{\omega}\vev\leq\left(\varepsilon+\frac{1}{T}\int_0^T\lambda\right)v\vev.
\]
Since $\omega v'\vev+v\cbQ_{\omega}\vev=\cbQ_{\omega}(v\vev)$ and $v\vev\in\caC^{1,2}_\upp(\R\times\R^n,(\vez,\vei))$, this shows that
$v\vev$ can be used both as a super-solution and as a sub-solution to derive from Proposition \ref{prop:max--min_characterization_lambdaz}
the following inequalities:
\[
    \frac{1}{T}\int_0^T\lambda(t)\upd t-\varepsilon\leq\lambda_1'(\cbQ_{\omega})\leq\frac{1}{T}\int_0^T\lambda(t)\upd t+\varepsilon.
\]

Passing to the limit $\varepsilon\to 0$ ends the proof of this step.
\end{proof}

Putting the two steps together, the claim is proved.
\end{proof}

Next we prove the limit $\omega\to+\infty$.

\begin{prop}\label{prop:high_time_frequency_limit}
For all $z\in\R^n$,
\[
    \lim_{\omega\to+\infty}\lambda_{1,z}(\cbQ_{\omega})=\lambda_{1,z}\left(-\diag(\nabla\cdot(\hat{A}_i\nabla)-\hat{q}_i\cdot\nabla)-\hat{\veL}\right),
\]
\[
    \lim_{\omega\to+\infty}\lambda_1(\cbQ_{\omega})=\lambda_1\left(-\diag(\nabla\cdot(\hat{A}_i\nabla)-\hat{q}_i\cdot\nabla)-\hat{\veL}\right),
\]
where
\[
    \left((\hat{A}_i,\hat{q}_i)_{i\in[N]},\hat{\veL}\right):x\mapsto
    \frac{1}{T}\int_0^T\left((A_i,q_i)_{i\in[N]},\veL\right)(t,x)\upd t.
\]
\end{prop}

\begin{proof}
Similarly to the proof of Proposition \ref{prop:small_time_frequency_limit}, it is sufficient to prove only the case $z=0$.

By virtue of Corollary \ref{cor:comparison_eigenvalues_underline_overline_L},
$-\lambda_\upPF(\overline{\veL})\leq\lambda_1'(\cbQ_\omega)\leq-\lambda_\upPF(\underline{\veL})$,
whence there exists a sequence $(\omega_k)_{k\in\N}$ and $\lambda_\infty\in\R$ such that, as $k\to+\infty$, $\omega_k\to+\infty$ and
$\lambda_k=\lambda_1'(\cbQ_{\omega_k})\to\lambda_\infty$.

Let $\veu_k\in\caC^{1,2}_\upp(\R\times\R^n,(\vez,\vei))$ be the unique generalized principal eigenfunction associated with 
$\lambda_k$ satisfying the normalization $\int_{\clOmper}|\veu_k|^2=1$.

Multiplying $(\cbQ_{\omega_k}\veu_k)_i-\lambda_k u_{k,i}$ by $u_{k,i}$, integrating by parts over $\clOmper$, and using
the Cauchy--Schwarz inequality $\int_{\clOmper}u_{k,j}u_{k,i}\leq\|u_{k,j}\|_{\mathcal{L}^2(\clOmper)}\|u_{k,i}\|_{\mathcal{L}^2(\clOmper)}$,
we obtain the uniform boundedness of $(\nabla u_{k,i})_{k\in\N}$ in $\mathcal{L}^2(\clOmper)$ for each $i\in[N]$, just as in Nadin
\cite[Proof of Theorem 3.10]{Nadin_2007}.

From now on, we assume $\caC^1$ regularity in time of the coefficients $(A_i)_{i\in[N]}$.
The proof in the general case with mere H\"{o}lder-continuity in time is not detailed -- it can be deduced from the $\caC^1$ case 
by a standard regularization procedure and the continuity of periodic principal eigenvalues.

By integration by parts in time, for each $i\in[N]$ and $k\in\N$, 
\[
    \int_{\clOmper}A_i\nabla u_{k,i}\cdot\nabla\partial_t u_{k,i}=-\int_{\clOmper}\partial_t A_i\nabla u_{k,i}\cdot\nabla u_{k,i}-\int_{\clOmper}A_i\nabla\partial_t u_{k,i}\cdot\nabla u_{k,i},
\]
whence, by virtue of the symmetry of $A_i$, the following identity holds:
\[
    \int_{\clOmper}A_i\nabla u_{k,i}\cdot\nabla\partial_t u_{k,i}=-\frac12\int_{\clOmper}\partial_t A_i\nabla u_{k,i}\cdot\nabla u_{k,i}.
\]
From this identity and a space-time integration by parts of $\left((\cbQ_{\omega_k}\veu_k)_i-\lambda_k u_{k,i}\right)\partial_t u_{k,i}$,
we deduce
\begin{align*}
    \omega_k\int_{\clOmper}(\partial_t u_{k,i})^2 = & \frac12\int_{\clOmper}\partial_t A_i\nabla u_{k,i}\cdot\nabla u_{k,i} \\
    & -\int_{\clOmper}(q_i\cdot\nabla u_{k,i})\partial_t u_{k,i} \\
    & +\sum_{j=1}^N\int_{\clOmper}l_{i,j}u_{k,j}\partial_t u_{k,i}.
\end{align*}
By the Cauchy--Schwarz inequality and the Young inequality, there exist $A>0$ and $B>0$, that only depend 
on $\mathcal{L}^\infty$ bounds on $(A_i)_{i\in[N]}$, $(q_i)_{i\in[N]}$ and $\veL$, such that, for each $i\in[N]$,
\begin{equation*}
    (\omega_k-A)\|\partial_t u_{k,i}\|_{\mathcal{L}^2(\clOmper)}^2\leq B\|\nabla u_{k,i}\|_{\mathcal{L}^2(\clOmper)}^2.
\end{equation*}
Therefore, by finiteness of $\sup_{k\in\N}\|\nabla u_{k,i}\|_{\mathcal{L}^2(\clOmper)}$, 
\[
    \|\partial_t u_{k,i}\|_{\mathcal{L}^2(\clOmper)}\to 0\quad\text{as }k\to+\infty.
\]

Hence, for each $i\in[N]$, $(u_{k,i})_k$, $(\partial_t u_{k,i})_k$ and $(\nabla u_{k,i})_k$ are all 
uniformly bounded in $\mathcal{L}^2(\clOmper)$, with $\|\partial_t u_{k,i}\|_{\mathcal{L}^2(\clOmper)}\to 0$ as well. Therefore, up to 
extraction of a subsequence, $u_{k,i}$ converges in $\mathcal{L}^2(\clOmper)$ to a limit $u_{\infty,i}$ and $\nabla u_{k,i}$ and
$\partial_t u_{k,i}$ converge weakly in $\mathcal{L}^2(\clOmper)$ to limits $\nabla u_{\infty,i}$ and
$\partial_t u_{\infty,i}$ respectively. 
By weak lower-semicontinuity of the norm in $\mathcal{L}^2(\clOmper)$, the convergence $\partial_t u_{k,i}\to 0$ occurs in fact 
in the sense of the strong convergence in $\mathcal{L}^2(\clOmper)$. 

Let $\hat{\veu}_k:x\mapsto\frac{1}{T}\int_0^T\veu_k(t,x)\upd t$ and $\vev_k=\veu_k-\hat{\veu}_k$. By the Poincaré inequality,
there exists a constant $K>0$ such that, for each $i\in[N]$,
\[
    \int_0^T(\partial_t u_{k,i})^2=\int_0^T(\partial_t v_{k,i})^2\geq K\int_0^T v_{k,i}^2,
\]
whence
\[
    \|v_{k,i}\|_{\mathcal{L}^2(\clOmper)}\to 0\quad\text{as }k\to+\infty.
\]
Also, since
\begin{align*}
    \int_{[0,L]}\left|\hat{u}_{k,i}-u_{\infty,i}\right| & =\int_{[0,L]}\left|\frac{1}{T}\int_0^T u_{k,i}-\frac{1}{T}\int_0^T u_{\infty,i}\right| \\
    & \leq\frac{1}{T}\int_{\clOmper}|u_{k,i}-u_{\infty,i}| \\
    & \leq \sqrt{\frac{|[0,L]|}{T}}\left(\int_{\clOmper}(u_{k,i}-u_{\infty,i})^2\right)^{1/2}
\end{align*}
for each $i\in[N]$, $\hat{\veu}_{k}$ converges to $\veu_{\infty}$ in $\mathcal{L}^1([0,L])$. Similarly, for any test function
$\varphi\in\mathcal{L}^2_\upp(\R^n)$,
\[
    \left|\int_{[0,L]}(\nabla\hat{u}_{k,i}-\nabla u_{\infty,i})\varphi\right|\leq\frac{1}{T}\left|\int_{\clOmper}(\nabla u_{k,i}-\nabla u_{\infty,i})\varphi\right|,
\]
so that $\nabla \hat{\veu}_k\rightharpoonup \nabla\veu_\infty$ in $\mathcal{L}^2_\upp(\R^n)$.

Integrating for any $k\in\N$ the quantity $(\cbQ_{\omega_k}\veu_k)_i-\lambda_k u_{k,i}$ in $[0,T]$ and dividing by $T$, we deduce
\begin{align*}
    0 & = \frac{1}{T}\nabla\cdot\left(\int_0^T (A_i\nabla u_{k,i})\right)-\frac{1}{T}\int_0^T(q_i\cdot\nabla u_{k,i})+\frac{1}{T}\sum_{j=1}^N\int_0^T l_{i,j}u_{k,j}+\lambda_k\hat{u}_{k,i} \\
    & = \nabla\cdot(\hat{A}_i\nabla\hat{u}_{k,i})-\hat{q}_i\cdot\nabla\hat{u}_{k,i}+\sum_{i=1}^N\hat{l}_{i,j}\hat{u}_{k,j}+\lambda_k\hat{u}_{k,i} \\
    & \quad + \frac{1}{T}\nabla\cdot\left(\int_0^T (A_i\nabla v_{k,i})\right)-\frac{1}{T}\int_0^T(q_i\cdot\nabla v_{k,i})+\frac{1}{T}\sum_{j=1}^N\int_0^T l_{i,j}v_{k,j}.
\end{align*}
Testing this identity against a test function in $\caC^2_\upp(\R^n)$ and using 
the convergence of $(\vev_k)_k$ to $\vez$ in $\mathcal{L}^2(\clOmper)$ as well as the convergence of $(\hat{\veu}_k)_k$ to $\veu_\infty$ in
$\mathcal{L}^1(\clOmper)$, we deduce that $x\mapsto\veu_\infty(x)$ is a weak solution in the dual
of $\caC^2_\upp(\R^n)$ of
\[
    \diag(\nabla\cdot(\hat{A}_i\nabla)-\hat{q}_i\cdot\nabla)\veu_\infty+\hat{\veL}\veu_{\infty}+\lambda\veu_\infty=\vez.
\]
By density, this remains true with test functions in $\mathcal{H}^1_\upp(\R^n)$, or in other words
$\veu_\infty$ is a weak solution on $\mathcal{H}^{-1}_\upp(\R^n)$. By elliptic regularity \cite{Gilbarg_Trudin}, $\veu_\infty\in\mathcal{H}^1_\upp(\R^n)$ is in fact
a classical solution, in $\caC^2_\upp(\R^n)$. Since $\veu_\infty$ is nonnegative and satisfies the normalization
$\int_{\clOmper}|\veu_\infty|^2=|[0,L]|\int_0^T|\veu_\infty|^2=1$, it is nonnegative nonzero, and then positive by the maximum principle
(the elliptic operator under consideration is fully coupled in $[0,L]$ by \ref{ass:irreducible}), whence it is a generalized principal 
eigenfunction associated with $\lambda_1'(-\diag(\nabla\cdot(\hat{A}_i\nabla)-\hat{q}_i\cdot\nabla)-\hat{\veL})$. Thus
$\lambda=\lambda_1'(-\diag(\nabla\cdot(\hat{A}_i\nabla)-\hat{q}_i\cdot\nabla)-\hat{\veL})$. 

As a conclusion, the accumulation point of $(\lambda_k)_{k\in\N}$
is unique and therefore the whole sequence converges. Subsequently, the whole family $(\lambda_1'(\cbQ_\omega))_{\omega>0}$ converges.
\end{proof}

\subsection{Formulas and estimates in special cases: proof of Theorems \ref{thm:lambdaz_space_homogeneous}--\ref{thm:lambdaz_ari-geo_time_averages}}

We begin this subsection by recalling that space, time or space-time homogeneous coefficients in $\cbQ$
lead to the reduced formulas
\eqref{eq:reduction_lambdaz_space_homogeneous}, \eqref{eq:reduction_lambdaz_time_homogeneous}, 
\eqref{eq:reduction_lambdaz_space-time_homogeneous}, respectively.

\subsubsection{Formulas for operators with space homogeneity: proof of Theorem \ref{thm:lambdaz_space_homogeneous}}

Recall the notations $\hat{A}_i$, $\hat{q}_i$, $\hat{\veL}$ for the averages in time, 
$\langle A_i\rangle$, $\langle q_i\rangle$, $\langle\veL\rangle$ for the averages in space and
$\langle \hat{A}_i\rangle$, $\langle \hat{q}_i\rangle$, $\langle\hat{\veL}\rangle$ for the averages in space-time.

\begin{prop}\label{prop:functions_of_time_only}
Let $z\in\R^n$. If
\begin{enumerate}
    \item $A_1$, $q_1$ and $\veL$ do not depend on $x$,
    \item there exists a constant positive vector $\veu\in(\vez,\vei)$ such that $\veu$ is a Perron--Frobenius eigenvector of $\veL(t)$ 
    for all $t\in\R$,
    \item either $z=0$ or $(A_1,q_1)=(A_2,q_2)=\dots=(A_N,q_N)$,
\end{enumerate}
then 
\[
    \lambda_{1,z}=-z\cdot \hat{A}_1z+\hat{q}_1\cdot z-\lambda_{\upPF}(\hat{\veL}).
\]
\end{prop}

\begin{proof}
First, writing the equality satisfied by $\veu$ and taking the average in time, we obtain 
$\frac{1}{T}\int_0^T\lambda_\upPF(\veL(t))\upd t=\lambda_\upPF(\hat{\veL})$. Note that $\hat{\veL}$ is irreducible.

Next, let $f:t\mapsto -z\cdot A_1(t)z+q_1(t)\cdot z-\lambda_{\upPF}(\veL(t))$. 
By uniqueness of the periodic principal eigenpair of $\cbQ$, it suffices to verify that
the space-independent function
\[
    (t,x)\mapsto\exp\left(-\int_0^t f(t')\upd t'+\frac{t}{T}\int_0^T f\right)\veu
\]
is a $\caC^{1,2}$, periodic, positive eigenfunction of $\cbQ_z$ associated with the eigenvalue $\frac{1}{T}\int_0^T f(t)\upd t$.
The continuity of $t\mapsto\lambda_{\upPF}(\veL(t))$ follows from \ref{ass:smooth_periodic} and the continuity of 
the Perron--Frobenius eigenvalue as function of the entries of the matrix.
\end{proof}

\begin{cor}
If
\begin{enumerate}
    \item $A_1$, $q_1$ and $\veL$ do not depend on $x$,
    \item there exists a constant positive vector $\veu\in(\vez,\vei)$ such that $\veu$ is a Perron--Frobenius eigenvector of $\veL(t)$ 
    for all $t\in\R$,
\end{enumerate}
then 
\[
    \lambda_1'=-\lambda_{\upPF}(\hat{\veL}).
\]

Furthermore, if $(A_1,q_1)=(A_2,q_2)=\dots=(A_N,q_N)$, then $\lambda_1=\lambda_1'$ if and only if $\hat{q}_1=0$.
\end{cor}

\begin{rem}\label{rem:counter-example_average_eigenvalue}
Although we do not know if the second condition in the statement is truly optimal, we know that the first condition alone cannot be
sufficient. Indeed, simple counter-examples exist. 

For instance, consider in dimension $N=2$ the matrix
\[
\veL:t\mapsto\begin{pmatrix} 0 & \eta(t) \\ \eta(t-T/2) & 0\end{pmatrix}
\]
where $\eta$ is the continuous $T$-periodic function that coincides on $[0,T]$ with $t\mapsto\max\left(\sin\left(\frac{2\pi}{T}t\right),0\right)$.

Even though $\veL(t)$ is actually always reducible, its Perron--Frobenius eigenvalue, understood as the continuous extension 
of the Perron--Frobenius eigenvalue to essentially nonnegative matrices, is $0$ for all $t\in[0,T]$, it is always a geometrically simple 
eigenvalue and its unit Perron--Frobenius eigenvector is $(1,0)^\upT$ in $(0,T/2)$ and $(0,1)^\upT$ in $(T/2,T)$.
The matrix $\hat{\veL}$ is symmetric and admits $\veo$ as Perron--Frobenius eigenvector and $1/\pi$ as Perron--Frobenius eigenvalue.

Due to the uniqueness of the periodic principal eigenfunction and the symmetries of $\veL$, the periodic principal eigenfunction
of $\partial_t-\Delta-\veL$, namely that of $\frac{\upd}{\upd t}-\veL$, necessarily
has the form $\veu:t\mapsto (u(t),u(t-T/2))^\upT$. Moreover, we can choose to normalize it with $u(0)=1$. 
The scalar function $u$ satisfies $u'(t)=\eta(t)u(t-T/2)+\lambda_1' u(t)$ for all $t\in[0,T]$, \textit{i.e.} 
$u'(t)=\sin\left(\frac{2\pi}{T}t\right)u(t-T/2)+\lambda_1' u(t)$ for all $t\in[0,T/2]$ and 
$u'(t)=\lambda_1' u(t)$ for all $t\in[T/2,T]$. 
It follows that $u$ satisfies:
\[
    u(t)\upe^{-\lambda_1't}=1+\int_0^t\upe^{-\lambda_1't'}\sin\left(\frac{2\pi}{T}t'\right)u(t'-T/2)\upd t'\quad\text{for all }t\in[0,T/2],
\]
\[
    u(t)\upe^{-\lambda_1'(t-T/2)}=u(T/2)\quad\text{for all }t\in[T/2,T].
\]
Since $u(T/2)=u(T)\upe^{-\lambda_1' (T-T/2)}=u(0)\upe^{-\lambda_1' T/2}=\upe^{-\lambda_1' T/2}$
and since, for all $t'\in[0,T/2]$, $u(t'-T/2)=u(t'+T/2)=u(T/2)\upe^{\lambda_1'(t'+T/2-T/2)}=\upe^{\lambda_1'(t'-T/2)}$,
the first equality, for $t\in[0,T/2]$, is simplified as 
\begin{align*}
    u(t) & =\upe^{\lambda_1' t}+\upe^{\lambda_1'(t-T/2})\int_0^t\sin\left(\frac{2\pi}{T}t'\right)\upd t' \\
    & =\upe^{\lambda_1' t}+\upe^{\lambda_1'(t-T/2})\left(-\frac{T}{2\pi}\right)\left(\cos\left(\frac{2\pi}{T}t\right)-1\right)
\end{align*}
whence, evaluating at $t=T/2$,
\[
    \upe^{-\lambda_1' T/2}=u(T/2)=\upe^{\lambda_1' T/2}+\frac{T}{\pi}
\]
\textit{i.e.} $\sinh\left(\frac{\lambda_1' T}{2}\right)=-\frac{T}{2\pi}$, \textit{i.e.} 
\[
    \lambda_1'=\frac{2}{T}\sinh^{-1}\left(-\frac{T}{2\pi}\right)=-\frac{2}{T}\ln\left(\frac{T}{2\pi}+\sqrt{\frac{T^2}{4\pi^2}+1}\right).
\]
On one hand, the above equality shows that $\lambda_1'<0$, independently of the value of $T$. On the other hand, it is easily 
verified that $1$ is not in the image of $\tau\in(0,+\infty)\mapsto\frac{\ln(\tau+\sqrt{\tau^2+1})}{\tau}$, whence 
$\lambda_1'\neq-\frac{1}{\pi}$, also independently of the value of $T$.

Therefore this counter-example shows that in general, $\lambda_{1,\upp}\left(\frac{\upd}{\upd t}-\veL\right)$ coincides neither 
with $-\frac{1}{T}\int_0^T\lambda_\upPF(\veL)=0$ nor with $-\lambda_\upPF(\hat{\veL})=-\frac{1}{\pi}$.

Note also that this counter-example is consistent with the formulas for asymptotics $T\to 0$ and $T\to+\infty$ of Theorem \ref{thm:limits_eigenvalue_time_frequency}, that predict $\lambda_1'(T)\to \frac{1}{T}\int_0^T-\lambda_{\upPF}(\veL)$ as $T\to\infty$
and $\lambda_1'(T)\to -\lambda_{\upPF}\left(\hat{\veL}\right)$ as $T\to 0$.
\end{rem}

\begin{rem}
Consider a diagonal perturbation of $\veL$ of the form $\veL_\nu=\veL-\nu\vect{I}$ with $\nu>0$. 
At some arbitrary time $t_0\in[0,T]$, consider the ``frozen in time'' system of ordinary differential equations 
$\veu'(t)=\veL_\nu(t_0)\veu(t)$: its periodic principal eigenvalue is $-\lambda_\upPF(\veL_\nu(t_0))=\nu>0$. However, the periodic principal eigenvalue of the ``unfrozen'' nonautonomous
system $\veu'(t)=\veL_\nu(t)\veu(t)$ is $\lambda_1'+\nu$, that remains negative provided
$\nu>0$ is small enough. Therefore, although the trajectories of the ``frozen'' system
converge exponentially fast to $\vez$, the nonnegative nonzero solutions of the ``unfrozen'' system
diverge from $\vez$ exponentially fast. In this sense, the stability properties of the two systems are unrelated.

This fact should not surprise readers familiar with nonautonomous dynamical systems, since the existence of such counter-examples, 
relying strongly upon the non-symmetry, is classical. We highlight it here for other readers.
\end{rem}

\subsubsection{Formulas for operators with time homogeneity: proof of Theorem \ref{thm:lambdaz_time_homogeneous}}

Now we turn to the proof of Theorem \ref{thm:lambdaz_time_homogeneous}. We will use the following well-known property
concerning variational formulas in the self-adjoint elliptic case. 
\begin{prop}\label{prop:variational_formulas_when_self-adjoint}
If
\begin{enumerate}
    \item $(A_i)_{i\in[N]}$ and $\veL$ do not depend on $t$,
    \item $\veL(x)$ is symmetric for all $x\in\R^n$,
    \item $q_1=q_2=\dots=q_N=0$,
\end{enumerate}
then the periodic principal eigenvalue of $\vect{\mathcal{L}}=\diag\left(\nabla\cdot(A_i\nabla)\right)+\veL$ satisfies:
\[
    \lambda_1'(-\vect{\mathcal{L}})=\min_{\veu\in\caC^2_{\upp}(\R^n,\R^N)\backslash\{\vez\}}\frac{\displaystyle\int_{[0,L]}\left(\sum_{i=1}^N \nabla u_i\cdot A_i\nabla u_i-\veu^\upT\veL\veu\right)}{\displaystyle\int_{[0,L]}|\veu|^2}.
\]
and, for any nonempty bounded smooth open set $\Omega$, the Dirichlet principal eigenvalue satisfies:
\[
    \lambda_{1}(-\vect{\mathcal{L}},\Omega) = \min_{\veu\in\caC^1_0(\Omega,\R^N)\backslash\{\vez\}}\frac{\displaystyle\int_{\Omega}\left(\sum_{i=1}^N \nabla u_i\cdot A_i\nabla u_i-\veu^\upT\veL\veu\right)}{\displaystyle\int_{\Omega}|\veu|^2}.
\]
\end{prop}

These formulas being recalled, we are in a position to prove Theorem \ref{thm:lambdaz_time_homogeneous}. We first focus on the 
case $z=0$, and will then deduce the general case as a consequence of this one. Also, since the statement of Theorem 
\ref{thm:lambdaz_time_homogeneous} involves the inverses of the diffusion matrices $A_i$, we recall that their invertibility follows
from their uniform ellipticity \ref{ass:ellipticity}. Moreover, the periodicity and regularity of the inverses follows from
the periodicity and regularity of the diffusion matrices \ref{ass:smooth_periodic}).

\begin{prop}\label{prop:equality_between_lambda1_and_lambda1prime_when_self-adjoint}
If
\begin{enumerate}
    \item $(A_i)_{i\in[N]}$ and $\veL$ do not depend on $t$,
    \item $\veL(x)$ is symmetric for all $x\in\R^n$, 
    \item there exists $Q\in\caC^2(\R^n,\R)$ such that $\int_{[0,L]}\nabla Q=0$ and 
    \[
        A_1^{-1}q_1=A_2^{-1}q_2=\dots=A_N^{-1}q_N=\nabla Q,
    \]
\end{enumerate}
then
\[
    \lambda_1=\lambda_1'=\min_{\veu\in\caC^2_{\upp}(\R^n,\R^N)\backslash\{\vez\}}\frac{\displaystyle\int_{[0,L]}\left(\sum_{i=1}^N \nabla u_i\cdot A_i\nabla u_i-\veu^\upT\veL_Q\veu\right)}{\displaystyle\int_{[0,L]}|\veu|^2},
\]
where
\[
    \veL_Q =\veL+\frac{1}{4}\diag\left(2\nabla\cdot q_i-A_i^{-1}q_i\cdot q_i\right) =\veL+\frac{1}{4}\diag\left(2\nabla\cdot(A_i\nabla Q)-\nabla Q\cdot A_i\nabla Q\right).
\]
\end{prop}

\begin{proof}
Let $\vect{\mathcal{L}}=\diag\left(\nabla\cdot(A_i\nabla)-q_i\cdot\nabla\right)+\veL$. 

\begin{proof}[Step 1: the case $q_1=q_2=\dots=q_N=0$, \textit{i.e.} $Q$ constant]

By \eqref{eq:reduction_lambdaz_time_homogeneous}, we only have to prove the statement for the elliptic operator $\vect{\mathcal{L}}$.
Also, we already know that $\lambda_1'=\lambda_{1,0}\leq \lambda_1=\max_{z\in\R^n}\lambda_{1,z}$.

Following Berestycki--Rossi \cite{Berestycki_Ros_1}, we consider an even function $\chi\in\caC^\infty(\R,[0,1])$ such that 
$\chi=0$ in $\R\backslash[-1,1]$ and $\chi(0)=1$. Next, we construct a family of radial smooth cut-off functions $\left(\chi_R\right)_{R>1}$ 
such that, for each $R>1$, $\chi_R=1$ in $B_{R-1}$ and $\chi_R(x)=\chi(|x|-(R-1))$ if $x\in\R^n\backslash B_{R-1}$, 
where $B_{R-1}=B(0,R-1)$. By construction, the family $\left(\|\chi_R\|_{\caC^\infty(\R^n,\R)}\right)_{R>1}$ 
is bounded.

Let $R>1$ and denote $\mu_R$ the Dirichlet principal eigenvalue $\lambda_{1,\upDir}\left(-\vect{\mathcal{L}},B_R\right)$,
where $B_R=B(0,R)$. By Proposition \ref{prop:variational_formulas_when_self-adjoint}, 
\[
\mu_R= \min_{\veu\in\caC^1_0(B_R,\R^N)\backslash\{\vez\}}\frac{\int_{B_R}\left(\sum_{i=1}^N\nabla u_i\cdot A_i\nabla u_i-\veu^\upT\veL\veu\right)}{\int_{B_R}|\veu|^2}.
\]
Taking $\chi_R\veu_0$ as test function (recall that $\veu_0$ is a positive periodic 
principal eigenfunction of $\vect{\mathcal{L}}$) and using the equality $-\vect{\mathcal{L}}\veu_0=\lambda_1'\veu_0$ satisfied
pointwise in $B_{R-1}$, we find:
\begin{align*}
    \mu_R & \leq \frac{\int_{B_{R-1}}\chi_R\veu_0^\upT(-\vect{\mathcal{L}})(\chi_R\veu_0)
    + \int_{B_R\backslash B_{R-1}}\chi_R\veu_0^\upT(-\vect{\mathcal{L}})(\chi_R\veu_0)}{\int_{B_R}\chi_R^2|\veu_0|^2} \\
    & = \frac{\int_{B_{R-1}}\veu_0^\upT(-\vect{\mathcal{L}})\veu_0
    - \int_{B_R\backslash B_{R-1}}\chi_R\veu_0^\upT\vect{\mathcal{L}}(\chi_R\veu_0)}{\int_{B_R}\chi_R^2|\veu_0|^2} \\
    & = \frac{\lambda_1'\int_{B_{R-1}}|\veu_0|^2
    - \int_{B_R\backslash B_{R-1}}\chi_R\veu_0^\upT\vect{\mathcal{L}}(\chi_R\veu_0)}{\int_{B_R}\chi_R^2|\veu_0|^2} \\
    & = \frac{\lambda_1'\int_{B_{R-1}}\chi_R^2|\veu_0|^2
    - \int_{B_R\backslash B_{R-1}}\chi_R\veu_0^\upT\vect{\mathcal{L}}(\chi_R\veu_0)}{\int_{B_R}\chi_R^2|\veu_0|^2} \\
    & = \lambda_1'-\frac{\lambda_1'\int_{B_R\backslash B_{R-1}}\chi_R^2|\veu_0|^2+\int_{B_R\backslash B_{R-1}}\chi_R\veu_0^\upT\vect{\mathcal{L}}(\chi_R\veu_0)}{\int_{B_R}\chi_R^2|\veu_0|^2}.
\end{align*}
Using now the triangle inequality, the inequality $\chi_R\geq (\chi_R)_{|B_{R-1}}$, the discrete Cauchy--Schwarz inequality and then pointwise upper and lower bounds, we find:
\begin{align*}
    \mu_R & \leq \lambda_1'+\frac{|\lambda_1'|\int_{B_R\backslash B_{R-1}}\chi_R^2|\veu_0|^2+\int_{B_R\backslash B_{R-1}}|\chi_R\veu_0^\upT\vect{\mathcal{L}}(\chi_R\veu_0)|}{\int_{B_{R-1}}\chi_R^2|\veu_0|^2} \\
    & \leq \lambda_1'+\frac{\int_{B_R\backslash B_{R-1}}\left(|\lambda_1'|\chi_R^2|\veu_0|^2+|\chi_R\veu_0^\upT\vect{\mathcal{L}}(\chi_R\veu_0)|\right)}{\int_{B_{R-1}}|\veu_0|^2} \\
    & \leq \lambda_1'+\frac{\int_{B_R\backslash B_{R-1}}\left(|\lambda_1'|\chi_R^2|\veu_0|^2+|\chi_R\veu_0||\vect{\mathcal{L}}(\chi_R\veu_0)|\right)}{\displaystyle\min_{x\in\clOmper}|\veu_0(x)|^2\int_{B_{R-1}}1} \\
    & \leq \lambda_1' + \frac{|\lambda_1'|\|\chi_R\veu_0\|^2+\|\chi_R\veu_0\|
    \|\vect{\mathcal{L}}(\chi_R\veu_0)\|}{\displaystyle\min_{x\in\clOmper}|\veu_0(x)|^2}\frac{\int_{B_R\backslash B_{R-1}}1}{\int_{B_{R-1}}1}
\end{align*}
where the norm $\|\ \|$ is defined as $\|\vev\|=\sup_{x\in\R^n}|\vev(x)|$ (appropriate for $\caC(\R^n,\R^N)$).
Thus, from the boundedness of the operator $\vect{\mathcal{L}}:\caC^2(\R^n,\R^N)\to\caC(\R^n,\R^N)$ and the boundedness in
$\caC^2(\R^n,\R^N)$ of the family $\left(\chi_R\veu_0\right)_{R>1}$, there exists a constant $K>0$, independent of $R$, such that
\[
    \mu_R\leq \lambda_1'+K\frac{R^{n-1}}{(R-1)^n},
\]
and, passing to the limit $R\to+\infty$, we deduce finally $\lambda_1\leq \lambda_1'$ 
(the proof of the convergence of the Dirichlet principal eigenvalues in balls of increasing radius 
to the generalized principal eigenvalue $\lambda_1$ for the elliptic operator $-\vect{\mathcal{L}}$ is done exactly as in
the parabolic case, see Proposition \ref{prop:eigenvalue_limit_Dirichlet}). Hence $\lambda_1'=\lambda_1$.

The conclusion of this step follows from Proposition \ref{prop:variational_formulas_when_self-adjoint}.
\end{proof}

\begin{proof}[Step 2: the general case]
From now on, for all $i\in[N]$, $q_i=A_i\nabla Q$ with $\int_{[0,L]}\nabla Q=0$. 
Following Berestycki--Hamel--Rossi \cite{Berestycki_Hamel_Rossi}, the idea is to change variables to reduce this case to the previous one.

Preliminarily, we check that $Q\in\caC^2(\R^n,\R)$ is necessarily space periodic. Fix $\alpha\in[n]$.
The function $x\mapsto Q(x+L_\alpha e_\alpha)-Q(x)$, where $e_\alpha=(\delta_{\alpha\beta})_{\beta\in[n]}$, is constant, since
\[
\nabla(Q(x+L_\alpha e_\alpha)-Q(x))=A_1^{-1}(x+L_\alpha e_\alpha)q_1(x+L_\alpha e_\alpha)-A_1^{-1}(x)q_1(x)=0.
\]
Then
\begin{align*}
Q(L_\alpha e_\alpha)-Q(0) & =\left(|[0,L]|\right)^{-1}\int_{[0,L]}Q(x+L_\alpha e_\alpha)-Q(x)\upd x \\
& =\left(|[0,L]|\right)^{-1}\int_0^{L_\alpha}\int_{[0,L]}\frac{\partial Q}{\partial x_\alpha}(x+se_\alpha)\upd x\upd s\\
& = 0
\end{align*}
Hence $Q$ is indeed periodic with respect to $x_\alpha$, and then with respect to $x$.

Then, introducing for any $\veu\in\caC^2(\R^n,\R^N)$ the transformation 
\[
    \vev:x\mapsto\exp(Q(x)/2)\veu(x))
\]
and following \cite{Berestycki_Hamel_Rossi}, we get:
\begin{align*}
-(\vect{\mathcal{L}}\vev)_i & =\upe^{Q/2}\left[-(\vect{\mathcal{L}}\veu)_i-\frac{1}{2}\left(u_i\nabla\cdot (A_i\nabla Q)+2\nabla u_i\cdot(A_i\nabla Q)+\frac{1}{2}u_i\nabla Q\cdot(A_i\nabla Q)-u_i q_i\cdot\nabla Q\right)\right] \\
& =\upe^{Q/2}\left[-(\vect{\mathcal{L}}\veu)_i-\frac{1}{2}\left(u_i\nabla\cdot q_i+2\nabla u_i\cdot q_i-\frac{1}{2}u_i\nabla Q\cdot q_i\right)\right]\\
& =\upe^{Q/2}\left[-\nabla\cdot(A_i\nabla u_i)-(\veL\veu)_i-\frac{1}{2}\left(\nabla\cdot q_i-\frac{1}{2}\nabla Q\cdot q_i\right)u_i\right].
\end{align*}
Therefore $\vev$ is an eigenfunction of $-\vect{\mathcal{L}}$ if and only if $\veu$ is an eigenfunction of the new periodic elliptic
operator:
\[
-\vect{\mathcal{L}}_Q=-\diag\left(\nabla\cdot(A_i\nabla)\right)-\veL_Q.
\]
Consequently, the periodic and Dirichlet principal eigenvalues coincide: for instance, with $\veu$ a periodic principal eigenfunction of 
$-\vect{\mathcal{L}}_Q$, $\vev$ satisfies $-\vect{\mathcal{L}}\vev=\lambda_1'(-\vect{\mathcal{L}}_Q)\vev$, and
since $\vev=\upe^{Q/2}\veu$ is periodic, it is then (by uniqueness) a periodic principal eigenfunction of $-\vect{\mathcal{L}}$, whence 
$\lambda_1'(-\vect{\mathcal{L}})=\lambda_1'(-\vect{\mathcal{L}}_Q)$. 

By virtue of Step 1,
\[
    \lim_{R\to+\infty}\lambda_1(-\vect{\mathcal{L}}_Q,B_R)=\lambda_1'(-\vect{\mathcal{L}}_Q)
\]
and consequently
\[
    \lim_{R\to+\infty}\lambda_1(-\vect{\mathcal{L}},B_R)=\lambda_1'(-\vect{\mathcal{L}}).
\]
Therefore
\[
    \lambda_1(-\vect{\mathcal{L}})=\lambda_1'(-\vect{\mathcal{L}})=\lambda_1'(-\vect{\mathcal{L}}_Q).
\]
The conclusion of this step follows from the variational formula for the operator $-\vect{\mathcal{L}}_Q$, see
Proposition \ref{prop:variational_formulas_when_self-adjoint}.
\end{proof}

The proof is ended.
\end{proof}

\begin{rem}\label{rem:counter-example_variational_formula}
The symmetry assumption on $\veL$ is crucial, both for the equality between $\lambda_1$ and $\lambda_1'$ (as explained above in Remark
\ref{rem:counter-example_evenness_without_advection}) and for the equality between $\lambda_1'$ and the minimized integral.

Denote
\[
    R=\min_{\veu\in\caC^2_\upp(\R^n,\R^N)\backslash\{\vez\}}\frac{\int_{[0,L]}\left(\sum_{i=1}^N|\nabla u_i|^2-\veu^\upT\veL\veu\right)}{\int_{[0,L]}|\veu|^2},
\]
which is the quotient appearing in the variational formula in the special case $q_i=0$ and $A_i=\textup{Id}$ for each $i\in[N]$.

It is well-known that for a general non-symmetric square matrix, the maximum of the Rayleigh quotient needs not coincide with the dominant
eigenvalue. More precisely, the maximum of the Rayleigh quotient of a matrix $\veL$ coincides with the dominant eigenvalue of the symmetric
part $\frac12(\veL+\veL^\upT)$. Similarly, $R$ is the periodic principal eigenvalue of the symmetrized operator 
$-\Delta-\frac12(\veL+\veL^\upT)$.

Therefore, using a constant irreducible non-symmetric matrix 
\[
\veL=\begin{pmatrix} 1 & 1 \\ \varepsilon & 1\end{pmatrix}\quad\text{with }\varepsilon>0,
\]
we obtain a counter-example of the equality between $\lambda_1'=-1-\sqrt{\varepsilon}$ and $R=-(3+\varepsilon)/2$.
\end{rem}

Subsequently, replacing $(q_i)_{i\in[N]}$ by $(q_i-2A_iz)_{i\in[N]}$ and $\veL$ by 
$\veL+\diag(z\cdot A_i z+\nabla\cdot\left(A_i z\right)-q_i\cdot z)$, we obtain the following corollary, which
is the full statement of Theorem \ref{thm:lambdaz_time_homogeneous}.

\begin{cor}
If
\begin{enumerate}
    \item $(A_i)_{i\in[N]}$ and $\veL$ do not depend on $t$,
    \item $\veL(x)$ is symmetric for all $x\in\R^n$, 
    \item there exists $z\in\R^n$ and $Q\in\caC^2(\R^n,\R)$ such that $\int_{[0,L]}\nabla Q=0$ and 
    \[
        A_1^{-1}q_1=A_2^{-1}q_2=\dots=A_N^{-1}q_N=2z+\nabla Q,
    \]
\end{enumerate}
then
\[
    \lambda_1=\lambda_{1,z}=\min_{\veu\in\caC^2_{\upp}(\R^n,\R^N)\backslash\{\vez\}}\frac{\displaystyle\int_{[0,L]}\left(\sum_{i=1}^N\nabla u_i\cdot A_i\nabla u_i-\veu^\upT\veL_{Q,z}\veu\right)}{\displaystyle\int_{[0,L]}|\veu|^2},
\]
where
\[
    \veL_{Q,z}=\veL_Q+\diag\left(\nabla\cdot\left(A_i z\right)-z\cdot A_i(z+\nabla Q)\right)
\]
and $\veL_Q$ is defined as in the statement of Proposition \ref{prop:equality_between_lambda1_and_lambda1prime_when_self-adjoint}.
\end{cor}

With no symmetry assumption on $\veL$ and more general advection terms, we can still compare $\lambda_{1,z}$ with the variational formula.

\begin{cor}
Let $z\in\R^n$. If
\begin{enumerate}
    \item $(A_i,q_i)_{i\in[N]}$ and $\veL$ do not depend on $t$,
    \item for all $i\in[N]$, $q_i\in\caC^1_{\upp}(\R^n,\R^n)$ and $\nabla\cdot(q_i-2A_iz)\leq 0$,
    \item for all $i\in[N]$, $q_i\cdot z\geq 0$,
\end{enumerate}
then
\[
    \lambda_{1,z}\geq\lambda_{1,z}\left(\partial_t-\diag(\nabla\cdot(A_i\nabla))-\frac12(\veL+\veL^\upT)\right).
\]
\end{cor}
\begin{proof}
By time homogeneity of the coefficients, the periodic principal eigenfunction of the parabolic operator $\cbQ_z$ is time homogeneous.
Taking the scalar product between the periodic principal eigenfunction $\veu_z\in\caC^2_\upp(\R^n,(\vez,\vei))$ associated with 
$\lambda_{1,z}$ and $\cbQ_z\veu_z$ and then integrating in $[0,L]$, we get immediately:
\begin{align*}
    \lambda_{1,z}\int_{[0,L]}|\veu_z|^2 & =\sum_{i=1}^N\int_{[0,L]}\nabla u_{z,i}\cdot A_i\nabla u_{z,i}-\int_{[0,L]}\veu_z^\upT\veL\veu_z \\
    & \quad -\sum_{i=1}^N\int_{[0,L]}\left(\frac12\nabla\cdot(q_i-2A_iz)+(z\cdot A_iz) +\nabla\cdot(A_iz)-(q_i\cdot z)\right)u_{z,i}^2.
\end{align*}
The conclusion follows from $\veu_z^\upT\veL\veu_z=\veu_z^\upT\frac12(\veL+\veL^\upT)\veu_z$, 
the sign assumptions on $\nabla\cdot(q_i-2A_iz)$ and $q_i\cdot z$ and the variational formula of Theorem
\ref{thm:lambdaz_time_homogeneous} applied to the operator $\partial_t-\diag(\nabla\cdot(A_i\nabla))-\frac12(\veL+\veL^\upT)$.
\end{proof}

\begin{rem}
In the nonnegative square matrix context, the inequality $\lambda_\upPF(\veL)\leq\lambda_\upPF\left(\frac12(\veL+\veL^\upT)\right)$ 
is a consequence of a theorem by Levinger which states that $t\in[0,1]\mapsto\lambda_\upPF\left(t\veL+(1-t)\veL^\upT)\right)$ is
nondecreasing in $[0,1/2]$, nonincreasing in $[1/2,1]$, and that the function is constant if and only if the unit Perron--Frobenius
eigenvectors of $\veL$ and $\veL^\upT$ coincide. There are many works on this theorem and on its extension to Banach spaces. We 
refer for instance to the recent paper of Altenberg--Cohen \cite{Altenberg_Cohen_2020} and references therein.
\end{rem}

\begin{rem}
The second and third assumptions are obviously satisfied if $q_i$ is divergence-free and $z=0$, but it is also interesting to consider
for instance the case $z\neq 0$ with shear flows $q_i:x\mapsto(\alpha_i(x_2,\dots,x_n),0,\dots,0)^\upT$ with $\alpha_i$ of constant sign.
In biological applications (climate change at constant speed towards the north, fish populations living in a river, etc.) or when
studying planar spreading, such shear flows appear naturally.
\end{rem}

\begin{rem}
    We emphasize that the estimate of the above corollary on $\lambda_{1,z}$ is a lower estimate, contrarily to most estimates
    in this work which are upper estimates (see Subsection \ref{sec:theorems_explicit_formulas}).
\end{rem}

Taking as a test function in the variational formula of Theorem \ref{thm:lambdaz_time_homogeneous} any constant eigenvector of $\veL_{Q,z}(x)$, 
we obtain that the average of the corresponding eigenvalue is smaller than or equal to $-\lambda_1=-\lambda_{1,z}$. In particular, 
noting that a constant Perron--Frobenius eigenvector implies 
$\frac{1}{|[0,L]|}\displaystyle\int_{[0,L]}\lambda_\upPF(\veL_{Q,z}(x))\upd x=\lambda_\upPF(\langle\veL_{Q,z}\rangle)$, the 
following corollary holds.

\begin{cor}\label{cor:self-adjoint_case_and_average_value}
If
\begin{enumerate}
    \item $(A_i)_{i\in[N]}$ and $\veL$ do not depend on $t$,
    \item $\veL(x)$ is symmetric for all $x\in\R^n$, 
    \item there exists $z\in\R^n$ and $Q\in\caC^2(\R^n,\R)$ such that $\int_{[0,L]}\nabla Q=0$ and 
    \[
        A_1^{-1}q_1=A_2^{-1}q_2=\dots=A_N^{-1}q_N=2z+\nabla Q,
    \]
    \item there exists a constant positive vector $\veu\in(\vez,\vei)$ such that $\veu$ is a Perron--Frobenius eigenvector of $\veL_{Q,z}(x)$ 
    for all $x\in\R^n$,
\end{enumerate}
then
\[
    \lambda_1=\lambda_{1,z}\leq -\lambda_\upPF(\langle\veL_{Q,z}\rangle).
\]
\end{cor}

\begin{rem}\label{rem:counter-example_self-adjoint_average_value}
Again, denote
\[
    R=\min_{\veu\in\caC^2_\upp(\R^n,\R^N)\backslash\{\vez\}}\frac{\int_{[0,L]}\left(\sum_{i=1}^N|\nabla u_i|^2-\veu^\upT\veL\veu\right)}{\int_{[0,L]}|\veu|^2}.
\]
Let us construct a counter-example where all the conditions of the statement are satisfied but where, due to heterogeneities in $\veL(x)$,
\[
    R < -\lambda_\upPF(\langle\veL\rangle)=-\frac{1}{|[0,L]|}\displaystyle\int_{[0,L]}\lambda_\upPF\left(\veL(x)\right)\upd x.
\]
(The existence of such counter-examples in the scalar setting is well-known, we provide a vector counter-example just for the sake of completeness.)

In a spirit similar to that of Remark \ref{rem:counter-example_average_eigenvalue}, we set
\[
\veL:x\mapsto\begin{pmatrix} 1 & \eta(x_1) \\ \eta(x_1) & 1\end{pmatrix}
\]
where $\eta$ is the continuous $L_1$-periodic function that coincides on $[0,L_1]$ with $x_1\mapsto\max(L_1/4-|x_1-L_1/4|,0)$.

For all $x\in[0,L]$, $\lambda_\upPF(\veL(x))=1+\eta(x_1)$ with constant eigenvector $\veo$, whence
\[
    R \leq -\frac{1}{L_1}\int_0^{L_1}\left(1+\eta(x_1)\right)\upd x_1 = -1 -\frac{L_1}{16}.
\]
Considering test functions of the form $\veu(x)=u(x_1)\veo$, we get
\[
R\leq\min_{u\in\caC^2_\upp(\R), \|u\|_{\mathcal{L}^2}=1}\int_0^{L_1}\left(|u'(x_1)|^2-\left(1+\eta(x_1)\right)u(x_1)^2\right)\upd x_1
\]
Testing against (a $\caC^2$ approximation of) $u:x_1\mapsto\frac{8\sqrt{3}}{L_1^{3/2}}\eta(x_1)$, we find:
\begin{align*}
R & \leq \frac{96}{L_1^3}\int_0^{L_1}\left(1-\eta(x_1)^2-\eta(x_1)^3\right)\upd x_1 \\
& = -1+\frac{96}{L_1^2}-\frac{3L_1}{16} \\
& < -1 -\frac{L_1}{16}\quad\text{if }L_1>768^{1/3}.
\end{align*}
\end{rem}

\begin{rem}
More interestingly, in the vector setting, the inequality 
\[
-R\geq \frac{1}{|[0,L]|}\int_{[0,L]}\lambda_\upPF(\veL)
\]
might not be satisfied if the fourth assumption of the statement, regarding the existence of a constant positive eigenvector, fails.

To verify this claim, we consider the counter-example $\cbQ\veu=\partial_t \veu-\Delta\veu-\veL\veu$, where
\[
\veL:x\mapsto\frac{1}{1+\eta(x_1)+\eta\left(x_1-\frac{L_1}{3}\right)+\eta\left(x_1-\frac{2L_1}{3}\right)}
\begin{pmatrix} 1+\eta(x_1-L_1/3) & \eta(x_1) \\ \eta(x_1) & 1+\eta(x_1-2L_1/3)\end{pmatrix},
\]
where, this time, $\eta$ is the continuous $L_1$-periodic function that coincides on $[-L_1/2,L_1/2]$ with $x_1\mapsto\max(L_1/6-|x_1|,0)$.
Interval by interval, $\veL$ satisfies:
\[
    \veL(x)=\frac{1}{1+\eta\left(x_1-\frac{2L_1}{3}\right)}
    \begin{pmatrix}
        1 & 0 \\
        0 & 1+\eta\left(x_1-\frac{2L_1}{3}\right)
    \end{pmatrix}
    \quad\text{if }x_1\in\left[-\frac{L_1}{2},-\frac{L_1}{6}\right]+\mathbb{Z}L_1,
\]
\[
    \veL(x)=\frac{1}{1+\eta\left(x_1\right)}
    \begin{pmatrix}
        1 & \eta(x_1) \\
        \eta(x_1) & 1
    \end{pmatrix}
    \quad\text{if }x_1\in\left[-\frac{L_1}{6},\frac{L_1}{6}\right]+\mathbb{Z}L_1,
\]
\[
    \veL(x)=\frac{1}{1+\eta\left(x_1-\frac{L_1}{3}\right)}
    \begin{pmatrix}
        1+\eta\left(x_1-\frac{L_1}{3}\right) & 0 \\
        0 & 1
    \end{pmatrix}
    \quad\text{if }x_1\in\left[\frac{L_1}{6},\frac{L_1}{2}\right]+\mathbb{Z}L_1.
\]
Hence the Perron--Frobenius eigenvalue of $\veL(x)$ is $1$, for all $x\in[0,L]$, and the unique one-dimensional left-continuous unit 
Perron--Frobenius eigenvector is:
\[
    \veu_\upPF:x\mapsto
    \begin{cases}
        \frac{1}{\sqrt{2}}\begin{pmatrix}1\\1\end{pmatrix} & \text{if }x_1\in\left(0,\frac{L_1}{6}\right]\cup\left(\frac{5L_1}{6},L_1\right]+\mathbb{Z}L_1, \\
        \begin{pmatrix}1\\0\end{pmatrix} & \text{if }x_1\in\left(\frac{L_1}{6},\frac{L_1}{2}\right]+\mathbb{Z}L_1, \\
        \begin{pmatrix}0\\1\end{pmatrix} & \text{if }x_1\in\left(\frac{L_1}{2},\frac{5L_1}{6}\right]+\mathbb{Z}L_1.
    \end{cases}
\]
All other unit Perron--Frobenius eigenvectors coincide with this one at all continuity points.

Now, let $\veu\in\caC^2_\upp(\R^n,[\vez,\vei))$ and $\lambda\leq -1$ such that $-\Delta\veu-\veL\veu\leq\lambda\veu$.
Taking the scalar product in $\R^2$ with $\veu$ and integrating by parts in $[0,L]$, we obtain:
\[
0\leq \int_{[0,L]}\sum_{i=1}^2|\nabla u_i|^2\leq\int_{[0,L]}\veu^\upT(\veL-\vect{I})\veu.
\]
Since, at all $x\in[0,L]$, $\veL(x)-\vect{I}$ is a symmetric matrix with nonpositive eigenvalues, $\veu(x)^\upT(\veL(x)-\vect{I})\veu(x)\leq 0$.
Therefore all inequalities above are actually equalities, and in particular $\veu$ is a constant vector satisfying $\veu^\upT(\veL(x)-\vect{I})\veu=0$.
Since no Perron--Frobenius eigenvector is constant, necessarily $\veu=\vez$.

Therefore no $\lambda\leq -1$ can satisfy $-\Delta\veu-\veL\veu=\lambda\veu$ for some positive periodic $\caC^2$ eigenfunction $\veu$, 
and this fact, together with the reduction \eqref{eq:reduction_lambdaz_time_homogeneous},
directly implies that $\lambda_1'>-1$. 

Independently of this observation, Theorem \ref{thm:lambdaz_time_homogeneous} implies that $\lambda_1'=R$. Therefore we have indeed
proved that, for this counter-example,
\[
    -R<\frac{1}{|[0,L]|}\int_{[0,L]}\lambda_\upPF(\veL).
\]
\end{rem}

\subsubsection{Upper estimates: proof of Theorems \ref{thm:lambdaz_divergence-free}, \ref{thm:lambdaz_ari-geo_space_averages},
\ref{thm:lambdaz_ari-geo_time_averages}}

In a similar spirit, the following property requires a line-sum-symmetry assumption ($\veL\veo=\veL^\upT\veo$) and uses
the property described in Eaves--Hoffman--Rothblum--Schneider \cite[Corollary 3]{Eaves_1985}.

\begin{prop}\label{prop:lambdaz_line-sum-symmetry}
Let $z\in\R^n$. Assume:
\begin{enumerate}
    \item for all $i\in[N]$, $q_i\in\caC^1_{\upp}(\R^n,\R^n)$ and $\nabla\cdot(q_i-2A_i z)=0$,
    \item $\veL(t,x)$ is line-sum-symmetric for all $(t,x)\in\clOmper$.
\end{enumerate}

Then
\[
    \lambda_{1,z} \leq -\frac{1}{N}\left(\sum_{i,j=1}^N\langle\hat{l}_{i,j}\rangle+z\cdot\sum_{i=1}^N\left(\langle\hat{A}_i\rangle z-\langle\hat{q}_i\rangle\right)\right)
\]
and this inequality is an equality if $\veL+\diag(\nabla\cdot(A_i z)+z\cdot(A_i z -q_i))$ is irreducible at all $(t,x)\in\clOmper$
with Perron--Frobenius eigenvector $\veo$ and constant Perron--Frobenius eigenvalue.
\end{prop}

\begin{proof}
Denote, for all $i\in[N]$, $q_i-2A_i z=b_i$ and recall
\begin{align*}
    \cbQ_z & =\cbQ-\diag\left(2A_i z\cdot\nabla+z\cdot A_i z+\nabla\cdot\left(A_i z\right)-q_i\cdot z\right) \\
    & = \diag\left(\partial_t - \nabla\cdot(A_i\nabla) +b_i\cdot\nabla-\nabla\cdot(A_i z)-z\cdot\left(A_i z-q_i\right)\right)-\veL
\end{align*}

Denote $\veu=\veu_z$ the unit periodic principal eigenfunction associated with $\lambda_{1,z}$.
Taking the scalar product in $\R^N$ between $\left(1/u_{i}\right)_{i\in[N]}$ and $\cbQ_z\veu-\lambda_{1,z}\veu$, 
integrating by parts in $\clOmper$, using the fact that all $b_i$ are divergence-free and using 
\begin{equation}\label{eq:equality_line-sum-symmetry_gradient}
    \sum_{i=1}^N\frac{(\veL\veu)_i}{u_i}\geq \sum_{i,j=1}^N l_{i,j}
    \quad\text{and}\quad\int_{\clOmper}\frac{\nabla u_i}{u_i}\cdot A_i\frac{\nabla u_i}{u_i}\geq 0,
\end{equation}
we get
\[
    \lambda_{1,z} \leq -\frac{1}{NT|[0,L]|}\left(\sum_{i,j=1}^N\int_{\clOmper}l_{i,j}+z\cdot\sum_{i=1}^N\int_{\clOmper}\left(A_i z-q_i\right)\right).
\]
From the equality case in \eqref{eq:equality_line-sum-symmetry_gradient}, we deduce that this inequality is an equality if
$\veu_z\in\vspan(\veo)$ and $\veL(t,x)$ is irreducible at all $(t,x)\in\clOmper$. These conditions are satisfied if and only if
$\veL+\diag(\nabla\cdot(A_i z)+z\cdot(A_i z -q_i))$ is irreducible at all $(t,x)\in\clOmper$ with 
Perron--Frobenius eigenvector $\veo$ and Perron--Frobenius eigenvalue $\lambda_{1,z}$, both constant. Finally, by uniqueness of the
periodic principal eigenvalue, the assumption that the Perron--Frobenius eigenvalue is $\lambda_{1,z}$ can be replaced without loss
of generality by the assumption that the Perron--Frobenius eigenvalue is constant.
\end{proof}

\begin{rem}
Circulant matrices and doubly stochastic matrices are line-sum-symmetric and always admit $\veo$ as eigenvector. 
Hence all inequalities on $\lambda_{1,z}$ are equalities if:
\begin{enumerate}
    \item all $A_i$ are constant and coincide and all $q_i$ are constant and coincide,
    \item $\veL$ is, at all $(t,x)$, irreducible and either circulant or doubly stochastic,
    \item its Perron--Frobenius eigenvalue $\lambda_\upPF(\veL(t,x))=\sum_{j=1}^N l_{1,j}(t,x)$ is constant (this condition being automatically satisfied in the doubly stochastic case). 
\end{enumerate}
This shows in particular that the inequalities can all be equalities even if $\veL$ is not spatio-temporally constant.
\end{rem}

The following two corollaries are concerned with special cases.

\begin{cor}\label{cor:lambda1prime_line-sum-symmetry}
Assume:
\begin{enumerate}
    \item for all $i\in[N]$, $q_i\in\caC^1_{\upp}(\R^n,\R^n)$ and $\nabla\cdot q_i=0$,
    \item $\veL(t,x)$ is line-sum-symmetric for all $(t,x)\in\clOmper$.
\end{enumerate}

Then
\[
    \lambda_1' \leq -\frac{1}{N}\sum_{i,j=1}^N\langle\hat{l}_{i,j}\rangle=-\frac{1}{N}\veo^\upT\langle\hat{\veL}\rangle\veo
\]
and this inequality is an equality if $\veL$ is irreducible at all $(t,x)\in\clOmper$ with Perron--Frobenius eigenvector $\veo$ 
and constant Perron--Frobenius eigenvalue.
\end{cor}

\begin{cor}\label{cor:lambda1_line-sum-symmetry}
Assume:
\begin{enumerate}
    \item for all $i\in[N]$, $q_i\in\caC^1_{\upp}(\R^n,\R^n)$ and $q_i$ and each column of $A_i$ are divergence-free,
    \item $\veL(t,x)$ is line-sum-symmetric for all $(t,x)\in\clOmper$.
\end{enumerate}

Denote
\[
[A]=\frac{1}{N}\sum_{i=1}^N\langle\hat{A}_i\rangle,\quad [q]=\frac{1}{N}\sum_{i=1}^N\langle\hat{q}_i\rangle.
\]

Then
\[
    \lambda_1\leq -\frac{1}{N}\veo^\upT\langle\hat{\veL}\rangle\veo+\frac{1}{4}[q]\cdot [A][q]
\]
with equality if 
\[
    \veL+\frac12\diag\left(\nabla\cdot(A_i[A]^{-1}[q]+[A]^{-1}[q]\cdot\left(\frac12 A_i[A]^{-1}[q] -q_i\right)\right)
\]
is irreducible at all $(t,x)\in\clOmper$ with Perron--Frobenius eigenvector $\veo$ and constant Perron--Frobenius eigenvalue.
\end{cor}

\begin{proof}
The assumption that all $q_i$ and all columns of all $A_i$ are divergence-free implies that, for all $i\in[N]$ 
and $z\in\R^n$, $q_i-2A_i z$ is divergence-free (and actually the converse implication is also true: consider 
the special cases $z=0,e_{1},\dots,e_{n}$, where $e_\alpha=(\delta_{\alpha\beta})_{\beta\in[n]}$).

By \ref{ass:ellipticity}, $[A]$ is invertible, so that the inequality of Proposition 
\ref{prop:lambdaz_line-sum-symmetry} reads
\[
    \lambda_{1,z+\frac12[A]^{-1}[q]} \leq -\frac{1}{N}\veo^\upT\langle\hat{\veL}\rangle\veo
    -z\cdot[A]z+\frac{1}{4}[q]\cdot [A][q].
\]
The inequality on $\lambda_1=\max_{z\in\R^n}\lambda_{1,z}$ and the associated equality case follow directly.
\end{proof}

Now we turn to the proof of Theorem \ref{thm:lambdaz_ari-geo_space_averages}.

\begin{prop}
Let $z\in\R^n$. If, for all $i\in[N]$, $q_i\in\caC^1_{\upp}(\R^n,\R^n)$ and $\nabla\cdot(q_i-2A_iz)=0$, then
\[
    \lambda_{1,z}\leq \lambda_{1,\upp}\left(\partial_t-\veL^\#-\diag\left(z\cdot\left(\langle A_i\rangle z-\langle q_i\rangle\right)\right)\right),
\]
where the entries of the matrix $\veL^\#=\left(l_{i,j}^\#\right)_{(i,j)\in[N]^2}$ are defined by:
\[
    l_{i,j}^\#=
    \begin{cases}
        \frac{1}{|[0,L]|}\int_{[0,L]}l_{i,i} & \text{if }i=j, \\
        \exp\left(\frac{1}{|[0,L]|}\int_{[0,L]}\ln l_{i,j}\right) & \text{if }i\neq j\text{ and }\displaystyle\min_{(t,x)\in\clOmper}l_{i,j}(t,x)> 0, \\
        0 & \text{otherwise}.
    \end{cases}
\]
\end{prop}

\begin{proof}
The proof is quite similar to that of Proposition \ref{prop:lambdaz_line-sum-symmetry}, we only sketch it.

Multiply each line of $\cbQ_z\veu-\lambda_{1,z}\veu$ by $1/u_i$, integrate by parts in $[0,L]$, divide by $|[0,L]|$, 
define $J_i=\{j\in[N]\backslash\{i\}\ |\ \min_{\clOmper}l_{i,j}>0\}$, use the Jensen inequality:
\begin{align*}
    \frac{1}{|[0,L]|}\int_{[0,L]}\sum_{j\in[N]\backslash\{i\}}\frac{l_{i,j}u_j}{u_i} & 
    \geq \frac{1}{|[0,L]|}\int_{[0,L]}\sum_{j\in J_i}\upe^{-\ln(u_i)+\ln(l_{i,j})+\ln(u_j)} \\
    & \geq \sum_{j\in J_i}\upe^{\frac{1}{|[0,L]|}\int_{[0,L]}(-\ln(u_i)+\ln(l_{i,j})+\ln(u_j))} \\
    & = \sum_{j\in J_i}\frac{\upe^{\frac{1}{|[0,L]|}\int_{[0,L]}\ln(l_{i,j})}\upe^{\frac{1}{|[0,L]|}\int_{[0,L]}\ln(u_j)}}{\upe^{\frac{1}{|[0,L]|}\int_{[0,L]}\ln(u_i)}},
\end{align*}
define the positive function $\vev:t\mapsto\left(\exp\left(\frac{1}{|[0,L]|}\int_{[0,L]}\ln(u_i(t,x))\upd x\right)\right)_{i\in[N]}$, find 
\[
    \vev'-\left[\veL^\#+\diag\left(z\cdot\frac{1}{|[0,L]|}\int_{[0,L]}(A_iz-q_i)\right)\right]\vev\geq\lambda_{1,z}\vev,
\]
and subsequently use the min--max formula for the periodic principal eigenvalue. The result follows.
\end{proof}

Repeating the exact same procedure but this time with averages in $[0,T]$, we also find the estimate of Theorem 
\ref{thm:lambdaz_ari-geo_time_averages}, recalled below. 

\begin{prop}
Let $z\in\R^n$. If $(A_i)_{i\in[N]}$, $(q_i)_{i\in[N]}$ and $\veL$ do not depend on $x$, then
\[
    \lambda_{1,z}\leq -\lambda_{\upPF}\left(\veL^\flat+\diag\left(z\cdot\hat{A}_i z-\hat{q}_i\cdot z\right)\right),
\]
where the entries of the matrix $\veL^\flat=\left(l_{i,j}^\flat\right)_{(i,j)\in[N]^2}$ are defined by:
\[
    l_{i,j}^\flat=
    \begin{cases}
        \frac{1}{T}\int_0^T l_{i,i} & \text{if }i=j, \\
        \exp\left(\frac{1}{T}\int_0^T\ln l_{i,j}\right) & \text{if }i\neq j\text{ and }\displaystyle\min_{t\in[0,T]}l_{i,j}(t)> 0, \\
        0 & \text{otherwise}.
    \end{cases}
\]
\end{prop}

The estimate of Corollary \ref{cor:lambdaz_ari-geo_space-time_averages} follows directly.

\subsection{Optimization: proof of Theorems \ref{thm:optim_doubly_stochastic}--\ref{thm:optimization_spatial_distribution}}

\subsubsection{Optimization of the normalized mutation matrix: proof of Theorem \ref{thm:optim_doubly_stochastic}}

In this section we prove Theorem \ref{thm:optim_doubly_stochastic}.

We recall that a doubly stochastic matrix $\vect{S}\in\R^{N\times N}$ is a nonnegative matrix such that $\vect{S}\veo=\vect{S}^\upT\veo=\veo$.
Denote $\vect{\mathcal{S}}\subset\mathcal{L}^\infty_\upp(\R\times\R^n,\R^{N\times N})$ the set of all periodic functions whose values are doubly 
stochastic matrices almost everywhere and $\vect{\mathcal{S}}_{\{0,1\}}$ the subset of all functions valued almost
everywhere in the set of permutation matrices.

Although we assumed until now that the zeroth order term $\veL$ of $\cbQ$ is H\"{o}lder-continuous, it can be
verified that the family $(\lambda_{1,z})_{z\in\R^n}$ can still be defined if $\veL$ has only an $\mathcal{L}^\infty$ regularity, 
using for instance a standard regularization procedure not detailed here.

We begin with the following decomposition lemma.

\begin{lem}\label{lem:decomposition_L}
    Let $\vect{A}\in\R^{N\times N}$ be a non-diagonal, essentially nonnegative matrix and $\vect{r}=\vect{A}^\upT\veo$.
    
    Assume $\vect{A}-\diag\left(\vect{r}\right)$ admits a positive Perron--Frobenius eigenvector $\vev\in(\vez,\vei)$
    and let $\vect{\nu}=(1/v_i)_{i\in[N]}$.

    Then there exists $m_0>0$ such that, for any $m\in(0,m_0]$, the matrix 
    $\vect{S}=m(\vect{A}-\diag(\vect{r}))\diag(\vect{\nu})^{-1}+\vect{I}$
    is doubly stochastic.

    In other words, $\vect{A}$ can be decomposed as $\vect{A}=\diag(\vect{r})+(\vect{S}-\vect{I})\diag(\vect{\mu})$,
    with $\vect{r}\in\R^N$, $\vect{S}\in[0,1]^{N\times N}$ doubly stochastic, $\vect{\mu}\in(\vez,\vei)$.
\end{lem}

\begin{proof}
    By definition of $\vect{r}$, the matrix $\vect{A}-\diag(\vect{r})$ admits $\veo$ as left Perron--Frobenius
    eigenvector with Perron--Frobenius eigenvalue $0$. Hence 
    \[
        \lambda_\upPF(\vect{A}-\diag(\vect{r}))=\lambda_\upPF\left((\vect{A}-\diag(\vect{r}))^\upT\right)=0
    \]
    and $\vev$ satisfies $(\vect{A}-\diag(\vect{r}))\vev=\vez$.

    For any $m>0$, the essentially nonnegative matrix 
    \[
        \vect{S}=m(\vect{A}-\diag(\vect{r}))\diag(\vect{\nu})^{-1}+\vect{I}
    \]
    satisfies:
    \[
        \veo^\upT\vect{S}=m\veo^\upT\left[(\vect{A}-\diag(\vect{r}))\diag(\vect{\nu})^{-1}+\vect{I}\right]=\vez^\upT\diag(\vect{\nu})^{-1}+\veo^\upT=\veo^\upT,
    \]
    \[
        \vect{S}\veo=\left[m(\vect{A}-\diag(\vect{r}))\diag(\vect{\nu})^{-1}+\vect{I}\right]\veo=
        m(\vect{A}-\diag(\vect{r}))\vev+\veo=\veo,
    \]
    \textit{i.e.} $\veo$ is a left and right Perron--Frobenius eigenvector with Perron--Frobenius
    eigenvalue $1$. 
    
    Note that, by definition of $\vect{r}$, $r_i-a_{i,i}\geq 0$ for each $i\in[N]$. Since $\vect{A}$
    is non-diagonal, $\max_{i\in[N]}(r_i-a_{i,i})v_i>0$ and $\max_{i,j\in[N],i\neq j} a_{i,j}v_j>0$. 
    Therefore, for $\vect{S}$ to be doubly stochastic, it only remains to choose $m\in(0,m_0]$
    with $m_0$ defined as
    \[
        m_0=\min\left[\frac{1}{\displaystyle\max_{i,j\in[N],i\neq j} a_{i,j}v_j},\frac{1}{\displaystyle\max_{i\in[N]}\left(r_i-a_{i,i}\right)v_i}\right].
    \]
\end{proof}

\begin{rem}
    It can be easily verified that any other decomposition 
    $\vect{A}=\diag(\widetilde{\vect{r}})+(\widetilde{\vect{S}}-\vect{I})\diag(\widetilde{\vect{\mu}})$
    with $\widetilde{\vect{r}}\in\R^N$, $\widetilde{\vect{S}}\in[0,1]^{N\times N}$ doubly stochastic, 
    $\widetilde{\vect{\mu}}\in(\vez,\vei)$ satisfies, for some $m>0$, $\vect{r}=\widetilde{\vect{r}}$,
    $\vect{\mu}=\frac{1}{m}\widetilde{\vect{\mu}}$ and $\vect{S}=\vect{I}+m(\widetilde{\vect{S}}-\vect{I})$.
\end{rem}

\begin{rem}
    When such a decomposition $\vect{A}=\diag(\vect{r})+(\vect{S}-\vect{I})\diag(\vect{\mu})$ exists, 
    with $\vect{r}\in\R^N$, $\vect{S}\in[0,1]^{N\times N}$ doubly stochastic and $\vect{\mu}\in[\vez,\vei)$,
    the positivity of all $\mu_i$ is equivalent with the positivity of the 
    Perron--Frobenius eigenvector of $\vect{A}-\diag(\vect{r})$.
    On the contrary, when $\vect{A}$ is reducible and $\vect{A}-\diag(\vect{r})$ has 
    Perron--Frobenius eigenvectors only in $\partial(\vez,\vei)$, the existence of such a 
    decomposition can be both a true or a false statement, as shown by the following two examples:
    \[
	\vect{A}-\diag(\vect{r})=
	\begin{pmatrix}
	    0 & 1 & 1 \\
	    0 & -1 & 1 \\
	    0 & 0 & -2
	\end{pmatrix}
	=\left(
	\begin{pmatrix}
	    0 & 1/2 & 1/2 \\
	    0 & 1/2 & 1/2 \\
	    1 & 0 & 0
        \end{pmatrix}-\vect{I}\right)\diag(0,2,2)
    \]
    \[
	\vect{A}-\diag(\vect{r})=
	\begin{pmatrix}
	    -1 & 1 & 1/3 \\
	    1 & -1 & 2/3 \\
	    0 & 0 & -1
	\end{pmatrix}.
    \]
\end{rem}

\begin{prop}\label{prop:optimization_mutation_matrix}
Assume $\veL$ has the form 
\[
    \veL=\diag(\vect{r})+(\vect{S}-\vect{I})\diag(\vect{\mu})
\]
with $\vect{S}\in\vect{\mathcal{S}}$, $\vect{r}\in\mathcal{L}^\infty_\upp(\R\times\R^n,\R^N)$ and
$\vect{\mu}\in\mathcal{L}^\infty_\upp(\R\times\R^n,[\vez,\vei))$.

Then, for all $z\in\R^n$,
\[
    \inf_{\vect{S}\in\vect{\mathcal{S}}_{\{0,1\}}}\lambda_{1,z}(\vect{S})=\inf_{\vect{S}\in\vect{\mathcal{S}}}\lambda_{1,z}(\vect{S})
    \leq\sup_{\vect{S}\in\vect{\mathcal{S}}}\lambda_{1,z}(\vect{S})=\inf_{\vect{S}\in\vect{\mathcal{S}}_{\{0,1\}}}\lambda_{1,z}(\vect{S}).
\]
Furthermore, all $\inf$ and $\max$ above are actually $\min$ and $\max$ respectively.
\end{prop}

\begin{proof}
It suffices to prove that there exists an element of $\vect{\mathcal{S}}_{\{0,1\}}$ that minimizes 
$\vect{S}\in\vect{\mathcal{S}}\mapsto\lambda_{1,z}$, the property on the maximum being proved similarly. 
Also, it is sufficient to prove only the case $z=0$.

\begin{proof}[Step 1: exhibiting a minimizer in $\vect{\mathcal{S}}$]
The closed and bounded set 
\[
    \vect{\mathcal{S}}=\left\{ \vect{S}\in\mathcal{L}^\infty(\Omega_\upp,\R^{N\times N})\ |\ \vect{S}\geq\vez,\ \vect{S}\veo=\vect{S}^\upT\veo=\veo\ \text{a.e.}\right\}
\]
is, by virtue of the Banach--Alaoglu theorem, compact in the weak-$\star$ topology of
$\mathcal{L}^\infty(\Omega_\upp)=(\mathcal{L}^1(\Omega_\upp))'$. 
Hence a minimizing sequence $(\vect{S}_k)_{k\in\N}$ converges, up to extraction, to a weak-$\star$ limit 
$\vect{S}_\infty\in\mathcal{L}^\infty(\Omega_\upp,\R^{N\times N})$. 
Extending periodically $\vect{S}_\infty\in\mathcal{L}^\infty_\upp(\R\times\R^n,\R^{N\times N})$,
it only remains to verify $\vect{S}\in\vect{\mathcal{S}}$ and $\lambda_1'(\vect{S}_\infty)=\lim_{k\to+\infty}\lambda_1'(\vect{S}_k)$. 

The nonnegativity of $\vect{S}_\infty$ in the sense of linear forms is immediate, testing the convergence against arbitrary nonnegative
functions in $\mathcal{L}^1(\Omega_\upp)$. Subsequently, testing for all $(i,j)\in[N]^2$ against $\vect{e}_j$ multiplied by the 
indicator of $\{ s_{\infty,i,j}<0 \}\cap\clOmper$, we deduce the nonnegativity almost everywhere. Testing for any 
$(t_0,x_0)\in\clOmper$, $\rho>0$, against $\veo$ multiplied by the indicator of $B((t_0,x_0),\rho)$ and divided by $|B((t_0,x_0),\rho)|$, 
we deduce by virtue of the Lebesgue differentiation theorem $\vect{S}_\infty\veo=\veo$ almost everywhere. 
Next, testing the convergence against all $(\vect{e}_j)_{j\in[N]}$, we find that all entries $s_{k,i,j}$ converge in the weak-$\star$
topology of $\mathcal{L}^\infty(\Omega_\upp,\R)$, so that $(\vect{S}^\upT_k)_{k\in\N}$ also converges. 
Similarly, $\vect{S}_\infty^\upT\veo=\veo$ almost everywhere. Therefore $\vect{S}_\infty\in\vect{\mathcal{S}}$ indeed.

By continuity of $\lambda_1'$ with respect to the weak-$\star$ topology of $\mathcal{L}^\infty_\upp(\R\times\R^n,\R^{N\times N})$ 
(see Theorem \ref{thm:continuity_eigenvalue_L} and Remark \ref{rem:continuity_weak_topology}),
$\lim_{k\to+\infty}\lambda_1'(\vect{S}_k)=\lambda_1'(\vect{S}_\infty)$. This ends this step of the proof.
\end{proof}

Define the function:
\[
    \begin{matrix}
        \Phi: & \vect{\mathcal{S}}\times\R\times\R^n & \to & [N] \\
         & (\vect{S},t,x) & \mapsto & \#\{ (i,j)\in[N]^2\ |\ s_{i,j}(t,x)=1\}
    \end{matrix}
\]
and remark that, for any $\vect{S}\in\vect{\mathcal{S}}$, 
\[
    \vect{S}\in\vect{\mathcal{S}}_{\{0,1\}}\Longleftrightarrow |\{(t,x)\in\clOmper\ |\ \Phi(\vect{S},t,x)=N\}|=|\clOmper|=T|[0,L]|.
\]
In other words, defining $\Omega_\upp(\vect{S},N_0)=\{(t,x)\in\clOmper\ |\ \Phi(\vect{S},t,x)=N_0\}$,
\[
    \vect{S}\in\vect{\mathcal{S}}_{\{0,1\}}\Longleftrightarrow \forall N_0\in[N-1]\quad |\Omega_\upp(\vect{S},N_0)|=0.
\]

Let $\vect{S}^\wedge\in\vect{\mathcal{S}}$ be a minimizer. Assume $\vect{S}^\wedge\notin\vect{\mathcal{S}}_{\{0,1\}}$. This means that
there exists $N_0\in[N-1]$ such that $\Omega_{N_0}=\Omega_\upp(\vect{S}^\wedge,N_0)$ has a positive measure. 
We are now going to correct this minimizer step by step. 

In what follows, we first consider the case where $\vect{S}^\wedge$ is irreducible on average in $\Omega_{\upp}$. 
Of course this is not necessarily the case, and we will generalize afterward.
In order to ease the reading, we denote $\vect{B}=\diag(\vect{r})-\diag(\vect{\mu})$.

\begin{proof}[Step 2: when $\vect{S}^\wedge$ is irreducible on average, correcting the minimizer in a large subset of $\Omega_{N_0}$]
Let $\veu,\vev\in\caC^{1,2}_\upp(\R\times\R^n,(\vez,\vei))$
be respectively a periodic principal eigenfunction of $\cbQ$ and a periodic principal eigenfunction of the adjoint operator 
\[
    \cbQ^\star=-\partial_t-\diag(\nabla\cdot(A_i\nabla)+q_i\cdot\nabla+\nabla\cdot q_i)-\vect{B}^\upT-\diag(\vect{\mu})(\vect{S}^\wedge)^\upT.
\]
By full coupling of the operator $\cbQ$, $\veu$ and $\vev$ are positive.
With the normalizations $\int_{\clOmper}|\veu|^2=\int_{\clOmper}\vev^\upT\veu=1$ (the second one is possible because, by positivity,
$\veu$ and $\vev$ cannot be orthogonal), $\veu$ and $\vev$ are uniquely defined.

Let $(t,x)\in\Omega_{N_0}$. Following exactly the construction of Neumann--Sze \cite{Neumann_Sze_2007},
there exist two permutation matrices $\vect{P}(t,x)\in\vect{\mathcal{S}}_{\{0,1\}}$ and
$\vect{Q}(t,x)\in\vect{\mathcal{S}}_{\{0,1\}}$ such that:
\begin{enumerate}
    \item the matrix $\widetilde{\vect{S}}^\wedge(t,x)=\vect{P}(t,x)\vect{S}^\wedge(t,x)\vect{Q}(t,x)^\upT$ is doubly 
    stochastic and has a block diagonal form:
    \[
        \begin{pmatrix}
            \widetilde{\vect{S}}^\wedge_{\textup{top}}(t,x) & \vez \\ \vez & \widetilde{\vect{S}}^\wedge_{\textup{bottom}}(t,x)
        \end{pmatrix}
    \]
    with $\widetilde{\vect{S}}^\wedge_{\textup{bottom}}(t,x)\in\{0,1\}^{N_0\times N_0}$ empty (if $N_0=0$) or a permutation matrix (if $N_0>0$)
    and all entries in $\widetilde{\vect{S}}^\wedge_{\textup{top}}(t,x)$ smaller than $1$;
    \item for all $i\in[N-N_0]$, 
    \[
	\vect{e}_i^\upT\vect{P}(t,x)\vev(t,x)\leq \vect{e}_1^\upT\vect{P}(t,x)\vev(t,x);
    \]
    \item for all $j\in[N-N_0]$, 
    \[
	\vect{e}_j^\upT\vect{Q}(t,x)\diag(\vect{\mu}(t,x))\veu(t,x)\leq \vect{e}_1^\upT\vect{Q}(t,x)\diag(\vect{\mu}(t,x))\veu(t,x).
    \]
\end{enumerate}
The three properties together imply that, 
\begin{equation}\label{eq:doubly_stochastic_perm_1}
    (\widetilde{s}^\wedge_{i,1}(t,x))_{i\in[N]}^\upT\vect{P}(t,x)\vev(t,x)\leq \vect{e}_1^\upT\vect{P}(t,x)\vev(t,x),
\end{equation}
\begin{equation}\label{eq:doubly_stochastic_perm_2}
    (\widetilde{s}^\wedge_{1,j}(t,x))_{j\in[N]}^\upT\vect{Q}(t,x)\diag(\vect{\mu}(t,x))\veu(t,x)\leq \vect{e}_1^\upT\vect{Q}(t,x)\diag(\vect{\mu}(t,x))\veu(t,x).
\end{equation}

Next, define, for the same $(t,x)\in\Omega_{N_0}$,
\[
    \vect{a}(t,x)=-\vect{e}_1+(\widetilde{s}^\wedge_{i,1}(t,x))_{i\in[N]}=\begin{pmatrix}-(1-\widetilde{s}^\wedge_{1,1}(t,x)) \\ \widetilde{s}^\wedge_{2,1}(t,x) \\ \vdots \\ \widetilde{s}^\wedge_{N-N_0,1}(t,x) \\ 0 \\ \vdots \\ 0\end{pmatrix},
\]
\[
    \vect{b}(t,x)=-\vect{e}_1+(\widetilde{s}^\wedge_{1,j}(t,x))_{j\in[N]}=\begin{pmatrix}-(1-\widetilde{s}^\wedge_{1,1}(t,x)) \\ \widetilde{s}^\wedge_{1,2}(t,x) \\ \vdots \\ \widetilde{s}^\wedge_{1,N-N_0}(t,x) \\ 0 \\ \vdots \\ 0\end{pmatrix},
\]
\[
    \vect{T}^{\wedge}(t,x)=\frac{1}{1-\widetilde{s}^\wedge_{1,1}(t,x)}\vect{a}(t,x)\vect{b}(t,x)^\upT.
\]
Let us verify that $\widetilde{\vect{S}}^\wedge+\vect{T}^\wedge$ is doubly stochastic at $(t,x)$. Since
$\widetilde{\vect{S}}^\wedge(t,x)$ is doubly stochastic, we only have to verify that $\widetilde{\vect{S}}^\wedge(t,x)+\vect{T}^\wedge(t,x)$ is nonnegative
and that $\vect{a}(t,x)\vect{b}(t,x)^\upT\veo=\vect{b}(t,x)\vect{a}(t,x)^\upT\veo=\vez$. Both properties turn out to be obvious.
By construction, $\Phi(\widetilde{\vect{S}}^\wedge+\vect{T}^\wedge,t,x)=N_0+1$. Indeed, $\widetilde{\vect{S}}^\wedge(t,x)+\vect{T}^\wedge(t,x)$ has 
$N_0$ entries equal to $1$ in its bottom right block and its upper left entry satisfies
\[
    \widetilde{s}^{\wedge}_{1,1}(t,x)+t^\wedge_{1,1}(t,x)=\widetilde{s}^\wedge_{1,1}(t,x)+\frac{1}{1-\widetilde{s}^\wedge_{1,1}(t,x)}(1-\widetilde{s}^\wedge_{1,1}(t,x))^2=1.
\]

Let $\omega\subset\Omega_{N_0}$ be a measurable subset. Setting 
\[
    \vect{T}^\wedge_\omega:(t,x)\in\clOmper\mapsto
    \begin{cases}
        \vect{T}^\wedge(t,x) & \text{if }(t,x)\in\omega, \\
        \vez & \text{if }(t,x)\in\clOmper\backslash\omega,
    \end{cases}
\]
extending $\vect{T}^\wedge_\omega$ periodically in $\R\times\R^n$ and verifying routinely that $\vect{T}^\wedge_\omega$ 
is measurable, we are now in a position to verify that this construction does not modify the periodic principal eigenvalue $\lambda_1'$, 
namely $\widetilde{\vect{S}}^\wedge+\vect{T}^\wedge_\omega\in\mathcal{L}^\infty_\upp(\R\times\R^n,\R^{N\times N})$ is still a minimizer, 
provided $\omega$ is appropriately chosen.

Denote, for any $\alpha\in[0,1]$, 
\[
    \cbQ_\alpha=\diag(\caP_i)-\vect{B}-\vect{P}^\upT\left(\alpha\vect{T}^{\wedge}_\omega+\widetilde{\vect{S}}^\wedge\right)\vect{Q}\diag(\vect{\mu}),
\]
$\lambda(\alpha)=\lambda_1'(\cbQ_\alpha)$, and let $\veu_\alpha$ and $\vev_\alpha$ be two positive periodic principal eigenfunctions
of respectively $\cbQ_\alpha$ and of the adjoint operator 
\[
    \cbQ_\alpha^\star=-\partial_t-\diag(\nabla\cdot(A_i\nabla)+q_i\cdot\nabla+\nabla\cdot q_i)-\vect{B}^\upT-\diag(\vect{\mu})\vect{Q}^\upT\left(\alpha\vect{T}^{\wedge}_\omega+\widetilde{\vect{S}}^\wedge\right)^\upT\vect{P},
\]
normalized so that $\int_{\clOmper}|\veu_\alpha|^2=\int_{\clOmper}\vev_\alpha^\upT\veu_\alpha=1$.
For any $\alpha,\beta\in[0,1]$, $\alpha\neq\beta$,
\begin{align*}
    \frac{\lambda(\beta)-\lambda(\alpha)}{\beta-\alpha} & = \frac{\lambda(\beta)\int_{\clOmper}\vev_\beta^\upT\veu_\alpha-\lambda(\alpha)\int_{\clOmper}\vev_\beta^\upT\veu_\alpha}{(\beta-\alpha)\int_{\clOmper}\vev_\beta^\upT\veu_\alpha} \\
    & =\frac{1}{(\beta-\alpha)\int_{\clOmper}\vev_\beta^\upT\veu_\alpha}\left(\int_{\clOmper}\veu_\alpha^\upT(\cbQ_\beta^\star\vev_\beta)-\int_{\clOmper}\vev_\beta^\upT(\cbQ_\alpha\veu_\alpha)\right) \\
    & =\frac{1}{(\beta-\alpha)\int_{\clOmper}\vev_\beta^\upT\veu_\alpha}\left(\int_{\clOmper}\vev_\beta^\upT(\cbQ_\beta\veu_\alpha)-\int_{\clOmper}\vev_\beta^\upT(\cbQ_\alpha\veu_\alpha)\right) \\
    & =\frac{1}{(\beta-\alpha)\int_{\clOmper}\vev_\beta^\upT\veu_\alpha}\left(\int_{\clOmper}\vev_\beta^\upT((\cbQ_\beta-\cbQ_\alpha)\veu_\alpha)\right) \\
    & =\frac{1}{(\beta-\alpha)\int_{\clOmper}\vev_\beta^\upT\veu_\alpha}\left(\int_{\clOmper}\vev_\beta^\upT(((\beta-\alpha)\cbQ_1-(\beta-\alpha)\cbQ_0)\veu_\alpha)\right) \\
    & =-\frac{1}{\int_{\clOmper}\vev_\beta^\upT\veu_\alpha}\left(\int_{\clOmper}(\vect{P}\vev_\beta)^\upT\vect{T}^{\wedge}_\omega\vect{Q}\diag(\vect{\mu})\veu_\alpha\right) \\
    & =-\frac{1}{\int_{\clOmper}\vev_\beta^\upT\veu_\alpha}\int_{\omega}\frac{(\vect{a}^\upT\vect{P}\vev_\beta)(\vect{b}^\upT\vect{Q}\diag(\vect{\mu})\veu_\alpha)}{1-\widetilde{s}^\wedge_{1,1}}.
\end{align*}
Taking the limit $\beta\to\alpha$, this leads to 
\[
    \lambda'(\alpha)=-\int_{\omega}\frac{(\vect{a}^\upT\vect{P}\vev_\alpha)(\vect{b}^\upT\vect{Q}\diag(\vect{\mu})\veu_\alpha)}{1-\widetilde{s}^\wedge_{1,1}}.
\]
In view of this equality and of \eqref{eq:doubly_stochastic_perm_1}--\eqref{eq:doubly_stochastic_perm_2}, $\lambda'(0)\leq 0$. We claim that in fact $\lambda'(0)=0$.
Indeed, if this is not the case, then there exists a small $\alpha>0$ such that $\lambda(\alpha)<\lambda(0)$. Then the minimality of
$\lambda(0)=\lambda_1'(\vect{P}^\upT\widetilde{\vect{S}}^\wedge\vect{Q})=\lambda_1'(\vect{S}^\wedge)$ in $\vect{\mathcal{S}}$ is contradicted\footnote{
Note that we cannot in general extend $\lambda$ on the left of $\alpha=0$, since for $\alpha<0$, the matrix
$\alpha\vect{T}^{\wedge}+\widetilde{\vect{S}}^\wedge$ might loose the crucial property of essential nonnegativity. 
Thus the minimizer of $\lambda(\alpha)$, $\alpha=0$, is not in general an interior critical point and $\lambda'(0)=0$ cannot be deduced 
only from the first-order optimality condition. The role played by \eqref{eq:doubly_stochastic_perm_1}--\eqref{eq:doubly_stochastic_perm_2} 
is indeed crucial. By reversing one of the two inequalities, we obtain the proof of the complementary result on maximizers.}.

Since $\lambda'(0)=0$ for any choice of $\omega$, using the Lebesgue differentiation theorem, we obtain:
\[
-\frac{\left(\vect{a}^\upT\vect{P}\vev\right)\left(\vect{b}^\upT\vect{Q}\diag\left(\vect{\mu}\right)\veu\right)}{1-\widetilde{s}^\wedge_{1,1}}=0\quad\text{almost everywhere in }\Omega_{N_0}.
\]

Subsequently, for almost every $(t,x)\in \Omega_{N_0}$,
\[
    \vect{T}^{\wedge}_{\Omega_{N_0}}(t,x)\vect{Q}(t,x)\diag(\vect{\mu}(t,x))\veu(t,x)=\vez\quad\text{or}\quad
    \vect{T}^{\wedge}_{\Omega_{N_0}}(t,x)^\upT\vect{P}(t,x)\vev(t,x)=\vez.
\]
Let
\[
    \omega_{\veu}=\left\{(t,x)\in\Omega_{N_0}\ |\ \vect{T}^{\wedge}_{\Omega_{N_0}}(t,x)\vect{Q}(t,x)\diag(\vect{\mu}(t,x))\veu(t,x)=\vez\right\},
\]
\[
    \omega_{\vev}=\left\{(t,x)\in\Omega_{N_0}\ |\ \vect{T}^{\wedge}_{\Omega_{N_0}}(t,x)^\upT\vect{P}(t,x)\vev(t,x)=\vez\right\}.
\]
The subsets $\omega_{\veu}$ and $\omega_{\vev}\backslash\omega_{\veu}$ are measurable, disjoint and satisfy
$|\omega_{\veu}\cup(\omega_{\vev}\backslash\omega_{\veu})|=|\Omega_{N_0}|$.
One of the two, denoted below $\omega$, satisfies $|\omega|\geq\frac12|\Omega_{N_0}|$. Choosing this $\omega$ in the definition of
$\vect{T}^{\wedge}_\omega$, we deduce directly that the corresponding eigenvector at $\alpha=0$ ($\veu$ if
$\omega=\omega_{\veu}$, $\vev$ if $\omega=\omega_{\vev}$) remains a periodic principal eigenvector for any $\alpha\in[0,1]$,
with in addition $\lambda(\alpha)=\lambda(0)$ for any $\alpha\in[0,1]$.
In particular,
\[
    \lambda_1'(\vect{S}^\wedge)=\lambda(0)=\lambda(1)=\lambda_1'(\vect{S}^\wedge+\vect{P}^\upT\vect{T}^{\wedge}_{\omega}\vect{Q}).
\]
\end{proof}

\begin{proof}[Step 3: when $\vect{S}^\wedge$ is irreducible on average, correcting the minimizer in $\Omega_{N_0}$ almost everywhere]
Let $\vect{S}^\wedge_1=\vect{S}^\wedge+\vect{P}^\upT\vect{T}^\wedge_{\omega}\vect{Q}$ and
$\Omega_{N_0,1}=\Omega_\upp(\vect{S}^\wedge_1,N_0)$. Note that $\omega=\Omega_{N_0}\backslash\Omega_{N_0,1}$ up to a negligible set.
The new minimizer $\vect{S}^\wedge_1$ satisfies, by construction:
\begin{itemize}
    \item $|\Omega_\upp(\vect{S}^\wedge_1,N_0+1)|=|\omega|+|\Omega_\upp(\vect{S}^\wedge,N_0+1)|$,
    \item $(\vect{S}^\wedge_1)_{|\clOmper\backslash\Omega_{N_0}}=(\vect{S}^\wedge)_{|\clOmper\backslash\Omega_{N_0}}$,
    \item $|\Omega_{N_0,1}|\leq\frac12|\Omega_{N_0}|$,
    \item $\Omega_{N_0,1}\subset\Omega_{N_0}$.
\end{itemize}
Iterating the construction, we obtain a sequence $(\vect{S}^\wedge_k)_{k\in\N}\in\vect{\mathcal{S}}^\N$ of minimizers and a sequence of 
measurable sets $(\Omega_{N_0,k})_{k\in\N}$ such that, for each $k\geq 2$,
\begin{enumerate}
    \item $\Omega_\upp(\vect{S}^\wedge_k,N_0)=\Omega_{N_0,k}$,
    \item $|\Omega_\upp(\vect{S}^\wedge_k,N_0+1)|=|\Omega_{N_0,k-1}\backslash\Omega_{N_0,k}|+|\Omega_\upp(\vect{S}^\wedge_{k-1},N_0+1)|$,
    \item $(\vect{S}^\wedge_k)_{|\clOmper\backslash\Omega_{N_0,k-1}}=(\vect{S}^\wedge_{k-1})_{|\clOmper\backslash\Omega_{N_0,k-1}}$,
    \item $|\Omega_{N_0,k}|\leq\frac12|\Omega_{N_0,k-1}|$,
    \item $\Omega_{N_0,k}\subset\Omega_{N_0,k-1}$.
\end{enumerate}
In particular, 
\[
    0\leq|\Omega_{N_0,k}|\leq\frac{1}{2^k}|\Omega_{N_0}|,\quad
    |\Omega_\upp(\vect{S}^\wedge_k,N_0+1)|=|\Omega_{N_0}|-|\Omega_{N_0,k}|+|\Omega_\upp(\vect{S}^\wedge,N_0+1)|,
\]
whence, as $k\to+\infty$,
\[
    |\Omega_{N_0,k}|\to 0,\quad|\Omega_\upp(\vect{S}^\wedge_k,N_0+1)|\to|\Omega_{N_0}|+|\Omega_\upp(\vect{S}^\wedge,N_0+1)|.
\]

Let
\[
    \vect{S}^\wedge_{\infty}:(t,x)\in\clOmper\mapsto
    \begin{cases}
        \vect{S}^\wedge(t,x) & \text{if }(t,x)\in\clOmper\backslash\Omega_{N_0}, \\
        \vect{S}^\wedge_1(t,x) & \text{if }(t,x)\in\Omega_{N_0}\backslash\Omega_{N_0,1}, \\
        \vect{S}^\wedge_2(t,x) & \text{if }(t,x)\in\Omega_{N_0,1}\backslash\Omega_{N_0,2}, \\
        \vdots & \\
        \vect{I} & \text{if }(t,x)\in\bigcap_{k\in\N}\Omega_{N_0,k}.
    \end{cases}
\]
and extend it periodically in $\R\times\R^n$, so that $\vect{S}^\wedge_\infty\in\vect{\mathcal{S}}$. Note that $|\bigcap_{k\in\N}\Omega_{N_0,k}|=0$.

Then the sequence $(\vect{S}^\wedge_k)_{k\in\N}$ converges almost everywhere, and in any $\mathcal{L}^p_\upp(\R\times\R^n,\R^{N\times N})$
with $p\in[1,+\infty)$, to $\vect{S}^\wedge_\infty$. Moreover, by continuity of the mapping $\vect{S}\mapsto\lambda_1'(\vect{S})$ 
with respect to the topology of, say, $\mathcal{L}^2_\upp(\R\times\R^n,\R^{N\times N})$, $\vect{S}^\wedge_\infty$ is still a minimizer. 
Finally, it satisfies 
\[
    |\Omega_\upp(\vect{S}^\wedge_\infty,N_0)|=0\quad\text{and}\quad
    |\Omega_\upp(\vect{S}^\wedge_\infty,N_0+1)|=|\Omega_\upp(\vect{S}^\wedge,N_0)|+|\Omega_\upp(\vect{S}^\wedge,N_0+1)|.
\]
\end{proof}

\begin{proof}[Step 4: when $\vect{S}^{\wedge}$ is irreducible on average, correcting the minimizer in all possible sets $\Omega_{N_0}$]
Performing the construction of Steps 3 and 4, first for 
\[
    N_0^\wedge=\min\{N_0\in[N-1]\ |\ |\Omega_\upp(\vect{S}^\wedge,N_0)|>0\},
\]
and then for $N_0^\wedge+1$, etc., up to $N-1$, we obtain in the end a new minimizer whose restriction to $\clOmper$ is valued 
in the set of permutation matrices almost everywhere, that is a new minimizer in $\vect{\mathcal{S}}_{\{0,1\}}$. 
\end{proof}

\begin{proof}[Step 5: when $\vect{S}^\wedge$ is reducible on average]
The key tool for this generalization is a Frobenius normal form of the matrix with constant coefficients $\frac{1}{T|[0,L]|}\int_{\clOmper}\vect{S}^\wedge$.

This matrix is doubly stochastic. Indeed, its entries are valued in $[0,1]$ and
\begin{align*}
    \veo =\frac{1}{T|[0,L]|}\int_{\clOmper}\veo & =\frac{1}{T|[0,L]|}\int_{\clOmper}\left(\vect{S}^\wedge\veo\right) =\left(\frac{1}{T|[0,L]|}\int_{\clOmper}\vect{S}^\wedge\right)\veo \\
    & =\frac{1}{T|[0,L]|}\int_{\clOmper}\left(\left(\vect{S}^\wedge\right)^\upT\veo\right) =\left(\frac{1}{T|[0,L]|}\int_{\clOmper}\vect{S}^\wedge\right)^\upT\veo.
\end{align*}
    
By nonnegativity, there exists a permutation matrix $\vect{F}\in\{0,1\}^{N\times N}$ such that 
\[
    \vect{S}^{\textup{F}}=\vect{F}^{\upT}\left[\frac{1}{T|[0,L]|}\left(\int_{\clOmper}\vect{S}^\wedge\right)\right]\vect{F}
\]
is a block upper triangular matrix whose diagonal blocks are 
irreducible nonnegative square matrices (recall that $1\times 1$ matrices are by convention referred to as irreducible even if zero) 
and whose off-diagonal blocks are nonnegative.

Let us verify that this Frobenius normal form $\vect{S}^{\textup{F}}$ 
is actually block diagonal, with doubly stochastic diagonal blocks.

Indeed, since $\frac{1}{T|[0,L]|}\int_{\clOmper}\vect{S}^\wedge$ is doubly stochastic and since $\vect{F}$ is a permutation matrix, then $\vect{S}^\upF$
is also doubly stochastic. Moreover, the first diagonal block is a left-stochastic matrix (all columns sum to $1$) and, since 
all off-diagonal blocks are nonnegative, its lines sum to at most $1$. Let $N_1\in[N]$ such that this first block is a
$N_1\times N_1$ matrix; then the sum of all entries of the block is exactly $N_1$. Consequently, each line-sum is actually exactly $1$,
and the block is doubly stochastic. This, in turn, implies that all entries indexed by $(i,j)\in[N_1]\times[N]\backslash[N_1]$ 
are zero. Iterating on each diagonal block, we deduce that $\vect{S}^\upF$ has indeed the claimed form.

Now, remark that each diagonal block of $\vect{S}^\upF$ corresponds to a fully coupled subsystem in $\Omega_{\upp}$.

Hence, up to permutations that are constant in space-time (for regularity reasons, this matters), we can assume without loss of generality
that the operator $\dcbP-\vect{B}-\vect{S}^\wedge\diag(\vect{\mu})$ is in block diagonal form with each block fully coupled, and with
each block of $\vect{S}^\wedge$ a doubly stochastic matrix.

To conclude, it only remains to apply the correction of Steps 1--4 block by block. In the end,
we obtain indeed a minimizer in $\vect{\mathcal{S}}_{\{0,1\}}$. 
\end{proof}

This ends the proof.
\end{proof}

\begin{rem}\label{rem:optimization_more_general_decomposition}
Consistently with Neumann--Sze \cite{Neumann_Sze_2007}, the decomposition
$\veL=\diag(\vect{r})+(\vect{S}-\vect{I})\diag(\vect{\mu})$ can be replaced by a more general decomposition 
$\veL=\vect{B}+\vect{S}\vect{A}$ with $\vect{A}$ nonnegative and $\vect{B}$ essentially nonnegative. The generalization of the proof
is straightforward.
\end{rem}

\begin{rem}
From Proposition \ref{prop:optimization_mutation_matrix} and the fact that $\max_z$ and $\max_{\vect{S}}$ commute,
we can deduce a similar result on the maximizers of $\max_{z\in\R^n}\lambda_{1,z}(\vect{S})$. Yet we do not insist on it, for two reasons:
\begin{enumerate}
    \item from the discussion in Subsection \ref{sec:extension_coupling_default}, we know that $\max_{z\in\R^n}\lambda_{1,z}(\vect{S})$ is 
    not a satisfying generalization of $\lambda_1(\vect{S})$ when $\vect{S}$ ceases to satisfy \ref{ass:irreducible}, and clearly there are
    many $\vect{S}\in\vect{S}_{\{0,1\}}$ that do not satisfy \ref{ass:irreducible};
    \item since there is no reason why $\min_{\vect{S}}$ and $\max_z$ should commute (in particular, $(\vect{S},z)\mapsto\lambda_{1,z}(\vect{S})$ 
    is not convex--concave), the argument does not apply to minimizers.
\end{enumerate}
The problem of optimizing $\vect{S}\in\vect{\mathcal{S}}\mapsto\lambda_1(\vect{S})$, that needs both a unambiguous definition of $\lambda_1$
when $\veL$ ceases to satisfy \ref{ass:irreducible} and a new method of proof that applies to minimizers, remains therefore open.
\end{rem}

\subsubsection{Optimization of the mutation rate: proof of Theorem \ref{thm:generalized_karlin_time_homogeneous}}
Next we prove Theorem \ref{thm:generalized_karlin_time_homogeneous}. The proof relies on a dual convexity lemma of Altenberg 
\cite[Lemma 1]{Altenberg_2012} whose statement is recalled below.

\begin{lem}[Altenberg's dual convexity lemma]\label{lem:dual_convexity}
Let $f:(0,+\infty)\times[0,+\infty)\to\R$ be a function of two variables $r$ and $s$, positively homogeneous of degree $1$,
and convex with respect to its second variable $s$.

Then:
\begin{enumerate}
    \item $f$ is convex with respect to its first variable $r$; furthermore, $r\mapsto f(r,s)$ is strictly convex if $s\neq 0$ and if the 
    convexity ewith respect to $s$ is strict;
    \item for all $(r,s)\in(0,+\infty)\times[0,+\infty)$, $z\mapsto f(r,s)+zf(1,0)-f(r+z,s)$ is either identically zero or positive;
    furthermore it is positive if $s\neq 0$ and if $f$ is strictly convex with respect to $s$;
    \item for all $r\in(0,+\infty)$, 
    \[
        \lim_{\substack{r'\to r\\r'<r}}\frac{f(r',s)-f(r,s)}{r'-r}\leq\lim_{\substack{r'\to r\\r'>r}}\frac{f(r',s)-f(r,s)}{r'-r}\leq f(1,0)\quad\text{for all }s\in[0,+\infty),
    \]
    and the first inequality is an equality except possibly at a countable number of values of $r$.
\end{enumerate}
\end{lem}

In our context, this lemma brings forth Theorem \ref{thm:generalized_karlin_time_homogeneous}, as shown below.

\begin{cor}\label{cor:generalized_karlin_time_homogeneous}
Assume $(A_i)_{i\in[N]}$ is independent of $t$, $(q_i)_{i\in[N]}=0$, and $\veL$ has the form
$\veL=\diag(\vect{r})+(\vect{S}-\vect{I})\diag(\vect{\mu})$ with $\vect{r}\in\caC^{\delta/2,\delta}_{\upp}(\R^n,\R^N)$, 
$\vect{\mu}\in\caC^{\delta/2,\delta}_{\upp}(\R^n,(\vez,\vei))$ and $\vect{S}\in\vect{\mathcal{S}}$ all independent of $t$.

For any $\rho>0$, let $\cbQ_\rho$ be the operator with $(A_i)_{i\in[N]}$ and $\veL$ replaced by $(\rho A_i)_{i\in[N]}$
and $\diag(\vect{r})+\rho(\vect{S}-\vect{I})\diag(\vect{\mu})$ respectively.

Then $\rho\in[0,1]\mapsto\lambda_1'(\cbQ_\rho)$ is concave and nondecreasing. Furthermore, if $\vect{r}$ depends on $x$ and $s>0$, 
then it is strictly concave and increasing.
\end{cor}

\begin{proof}
We reduce the eigenvalue thanks to \eqref{eq:reduction_lambdaz_time_homogeneous} and apply Lemma \ref{lem:dual_convexity} to the function
\[
    f:(r,s)\mapsto -\lambda_1'\left(-r\diag(\nabla\cdot(A_i\nabla))-s\diag(\vect{r})-r(\vect{S}-\vect{I})\diag(\vect{\mu})\right)
\]
which is, by virtue of Theorem \ref{thm:concavity_eigenvalue_L}, convex with respect to $s\in(0,+\infty)$, strictly if
$\vect{r}$ depends on $x$, and which is of class $\caC^1$ in $(r,s)$ away from $r=0$.
Subsequently, $f(1,0)=0$ yields the monotonicity of $f$ with respect to $r$.
\end{proof}

Note that a version of the above corollary appropriate for $\lambda_{1,z}$, with an assumption 
$A_i^{-1}q_i=2z+\nabla Q$ reminiscent of Theorem \ref{thm:lambdaz_time_homogeneous}, could be just as easily established. 
For the sake of brevity, we focus here on $\lambda_1'$ only.

\subsubsection{Optimization of the spatial distribution of $\veL$ in dimension $1$: proof of Theorem \ref{thm:optimization_spatial_distribution}}

In this section, we prove Theorem \ref{thm:optimization_spatial_distribution}.

First, we investigate a Talenti inequality for cooperative elliptic systems, as such estimates 
are milestones to proving spectral comparison \cite[Theorem 3.9]{Nadin_2007}. Let us recall that the core idea underlying 
these estimates is to compare some $\mathcal{L}^p$ norms (here, the $\mathcal{L}^\infty$ norms) of the solution of an elliptic 
problem with that of a related equation, the coefficients of which have been replaced by their symmetrization. 
It should be noted that our results would also hold for boundary conditions of Dirichlet type in the spatial domain $B(0,R)$, $R>0$.

\begin{rem}
In what follows, we will use a few specific notations.

It will be convenient to identify $[0,L_1]$ with $\left[-\frac{L_1}2,\frac{L_1}2\right]$; this amounts to translating the functions, and has the advantage of having ${0}$ as a symmetry point. In this context,  let us recall the fundamental ordering on the set of functions:
for two functions $\vect{f},\vect{g}\in \mathcal{L}^2([0,L_1],[\vez,\vei))$, the notation $\vect{f}\prec \vect{g}$ stands for:
\[
    \int_{-\frac{r}{2}}^{\frac{r}{2}}\vect{f}\leq \int_{-\frac{r}{2}}^{\frac{r}{2}}\vect{g}\quad\text{for all } r\leq \frac{L_1}{2}.
\]
In particular, these inequalities hold component wise.

For any non-negative scalar function, we may identify its rearrangement $u^\dagger$ with a non-increasing mapping $\overline u:[0,L_1/2]\to \mathbb{R}$. For any non-negative $\veu$, we denote by $\veu^\dagger$ its periodic rearrangement.
\end{rem}

The first step in the proof of Theorem \ref{thm:optimization_spatial_distribution} is the following comparison result for elliptic systems.

\begin{prop}\label{prop:Talenti_Dirichlet}
Assume $\dcbP=\partial_t-\vect{D}\Delta$ for some diagonal matrix $\vect{D}$ with constant, positive diagonal entries and assume
that $\veL$ is nonnegative and depends only on $x$.

Let $c>0$ and $\vect{\phi},\vect{\psi}\in\mathcal{L}^2(\left[-\frac{L_1}2,\frac{L_1}2\right],[\vez,\vei))$ such 
that $\vect{\phi}\prec\vect{\psi}$. Let $\veu$ and $\vev$ be the (unique) solutions of 
\begin{equation*}
\begin{cases}
    -\vect{D}\Delta\veu+c\veu =\veL \vect{\phi} & \text{ in }\left[-\frac{L_1}2,\frac{L_1}2\right]\,, 
    \\\veu\in \mathcal W^{1,2}_{\upp{}}
\end{cases}
\end{equation*} and 
\begin{equation*}
\begin{cases}
    -\vect{D}\Delta\vev+c\vev =\veL^\dagger\vect{\psi}^\dagger & \text{ in }\left[-\frac{L_1}2,\frac{L_1}2\right]
    \\\vev\in \mathcal W^{1,2}_\upp
\end{cases}
\end{equation*}

Then $\veu\prec\vev$.
\end{prop}

\begin{proof}
First, let us verify that if $\vect{\phi},\vect{\psi}$ are nonnegative and satisfy $\vect{\phi}\prec\vect{\psi}$, then
\begin{equation}\label{eq:Talenti_intermediate}
    \veL^\dagger\vect{\phi}^\dagger\prec\veL^\dagger\vect{\psi}^\dagger.
\end{equation}
First of all, for any $s\in \left[0,\frac{L_1}2\right]$ and any $j\in [N]$,
\begin{equation*}
    \chi_{[-s,s]} \phi_j^\dagger \prec \chi_{[-s,s]}\psi_j^\dagger,
\end{equation*}
where $\chi$ denotes the characteristic function of a set. This property is stable by addition and multiplication by a nonnegative constant \cite{Alvino_Lions_Trombetti_1991}. 
As a consequence, since any nonnegative nonincreasing function can be approximated from below by a nonnegative step function,
\eqref{eq:Talenti_intermediate} follows from the monotone convergence theorem.

For the sake of simplicity, assume the level sets of each $u_i$ have zero Lebesgue measure 
-- should this not be the case, we can argue exactly as in \cite{Talenti_1976}. 
Let $\tau \geq 0$ be a fixed real number and let $i\in[N]$. 
Integrating the $i$-th equation on the level set $\left\{u_i \geq \tau\right\}$, we get 
\begin{equation*}
d_i\int_{\left\{u_i = \tau\right\}}|\nabla u_i|=-c\int_{\left\{u_i \geq \tau\right\}}u_i+\sum_{j=1}^N\int_{\left\{u_i \geq \tau\right\}}l_{i,j} \phi_j.
\end{equation*} 

Since $u$ and $u^\dagger$ are equimeasurable by the definition of the periodic rearrangement, there holds \[
    c\int_{\left\{u_i \geq \tau\right\}}u_i=c\int_{\left\{u_i^\dagger \geq \tau\right\}}u_i^\dagger.
\]
By the Hardy--Littlewood inequality, since all the $l_{i,j}$ are nonnegative, 
\[
    \sum_{j=1}^N\int_{\left\{u_i \geq \tau\right\}}l_{i,j} \phi_j\leq \sum_{j=1}^N\int_{\left\{u_i^\dagger \geq \tau\right\}}l_{i,j}^\dagger\phi_j^\dagger.
\]

At this point, we have obtained 
\[d_i\int_{\left\{u_i = \tau\right\}}|\nabla u_i|\leq-c\int_{\left\{u_i^\dagger \geq \tau\right\}}u_i^\dagger+\sum_{j=1}^N\int_{\left\{u_i^\dagger \geq \tau\right\}}l_{i,j}^\dagger \phi_j^\dagger.
\]
By \eqref{eq:Talenti_intermediate}, we thus conclude that 
\[d_i\int_{\left\{u_i = \tau\right\}}|\nabla u_i|\leq-c\int_{\left\{u_i^\dagger \geq \tau\right\}}u_i^\dagger+\sum_{j=1}^N\int_{\left\{u_i^\dagger \geq \tau\right\}}l_{i,j}^\dagger \psi_j^\dagger.
\]

We introduce the distribution function $\mu_i$ of $u_i$,
\[
    \mu_i(\tau)=|\{u_i >\tau\}|.
\]
From the co-area formula,
\[
-\mu_i'(\tau)=\int_{\{u_i=\tau\}}\frac1{|\nabla u_i|}.
\]

Since the periodic rearrangement decreases the perimeter of level-sets, we have 
\begin{equation*}
    \operatorname{Per}\left(\{u_i^\dagger=\tau\}\right)\leq \operatorname{Per}\left(\{u_i=\tau\}\right).
\end{equation*}
 From the Cauchy-Schwarz inequality, we obtain
\begin{align*}
    \operatorname{Per}\left(\{u_i^\dagger=\tau\}\right)^2 & \leq \operatorname{Per}\left(\{u_i=\tau\}\right)^2 \\
    & \leq \int_{\{u_i=\tau\}}\frac1{|\nabla u_i|}\int_{\{u_i=\tau\}}|\nabla u_i| \\
    & \leq -\mu_{i}'(\tau)\int_{\{u_i=\tau\}}|\nabla u_i| \\
    & \leq -\frac{\mu_{i}'(\tau)}{d_i}\left(-c\int_{\left\{u_i^\dagger \geq \tau\right\}}u_i^\dagger+\sum_{j=1}^N\int_{\left\{u_i^\dagger \geq \tau\right\}}l_{i,j}^\dagger \psi_j^\dagger\right).
\end{align*}  
Since we are working in one dimension, for any $\tau\in\left(\min(u_i),\max(u_i)\right)$, there holds
\[ 4\leq  \operatorname{Per}\left(\{u_i^\dagger=\tau\}\right)^2.\] Furthermore, by definition of the rearrangement,
\[ \int_{\{u_i>\tau\}}u_i=\int_0^{\mu_i(\tau)}\overline u_i.\]
We define, for $i\in[N]$, 
\[
    k^{\veu}_i:\xi\in\left[0,\frac{L_1}{2}\right]\mapsto\int_0^{\xi}\overline{u_i}.
\]
From this definition, we obtain
\begin{equation*}
    (k^{\veu}_i)''(\mu_i(\tau))=\overline{u_i}'(\mu_i(\tau))=\frac1{\mu_i'(\tau)}.
\end{equation*}

With these notations, we obtain the following differential inequality: for any $\xi \in [0,L_1/2]$, 
\[
-4 \left(k_i^{\veu}\right)''(\xi)\leq -\frac{c}{d_i}k_i^{\veu}(\xi)+\sum_{j=1}^N \int_0^{\xi} l_{i,j}^\dagger \psi_j^\dagger.
\] 
Furthermore, 
\[
k_i^{\veu}(0)=0.
\]

Working with $\vect{\psi}$ instead of $\vect{\phi}$ and with $\vev$ instead of $\veu$, all the previous inequalities are equalities: indeed, by the variational formulation of the equation on each of the coordinates $v_i$, it appears that $v_i^\dagger=v_i$, so that, for any $\tau\in \left(\min(v_i),\max(v_i)\right)$, $\mathrm{Per}(\{v_i=\tau\})^2=4$.
Thus, with transparent notations, $k_i^{\vev}$ solves the differential equation
\[
-4 \left(k_i^{\vev}\right)''(\xi)= -\frac{c}{d_i}k_i^{\vev}(\xi)+\sum_{j=1}^N \int_0^{\xi} l_{i,j}^\dagger \psi_j^\dagger.
\] 
Similarly, 
\[
k_i^{\vev}(0)=0.
\]

Hence, from \eqref{eq:Talenti_intermediate}, the vector $\vect{k}=\vect{k}^{\vev}-\vect{k}^{\veu}$ satisfies
\[
-4\diag(\vect{d})\vect{k}''+c\vect{k}\geq \vez,\quad \vect{k}(0)=\vez.
\]
Finally, integrating both equations in $\veu$ and $\vev$ on the domain we obtain 
\[
\int_{[0,L_1/2]}\veu_i=\frac1c\int_{[0,L_1/2]} \left(\veL \vect{\phi}\right)_i\leq\frac1c\int_{[0,L_1/2]} \left(\veL \vect{\psi}\right)_i=\int_{[0,L_1/2]}\vev_i,
\] 
so that
\[
\vect{k}\left(\frac{L_1}{2}\right)\geq \vez.
\]
From the maximum principle, 
\[
\vect{k}\geq \vez\quad\text{in }\left(0,\frac{L_1}{2}\right).
\]
However, this is exactly the desired conclusion.
\end{proof}

We now apply Proposition \ref{prop:Talenti_Dirichlet} to derive a comparison principle. 

\begin{prop}\label{prop:Talenti_Dirichlet_Parabolic}
Assume $\dcbP=\partial_t-\vect{D}\Delta$ for some diagonal matrix $\vect{D}$ with constant, positive diagonal entries.

Let $\veu_0\in \mathcal{L}^\infty_\upp(\R,[\vez,\vei))$ and let $\veu,\vev$ be the respective space periodic solutions of
\begin{equation*}
    \begin{cases}
        \cbQ\veu=\vez & \text{in }(0,T)\times\R, \\
        \veu=\veu_0 & \text{on }\{0\}\times \R
    \end{cases}
\end{equation*}
and
\begin{equation*}
    \begin{cases}
        \dcbP\vev-\veL^\dagger\vev=\vez & \text{in }(0,T)\times\R, \\
        \vev=\veu_0^\dagger & \text{on }\{0\}\times\R.
    \end{cases}
\end{equation*}

Then, for all $t\in[0,T]$, $\veu(t,\cdot)\prec\vev(t,\cdot)$.
\end{prop}

\begin{proof}
We use a classical time discretization of the system, following \cite[Proof of Theorem 3]{Alvino_Trombetti_Lions_1990}.

Let $c>0$ so large that $l_{i,i}+c\geq 0$ for all $i\in[N]$.
Let $K\in\N$ and $\delta=\frac{T}{K}>0$. We define, for any $\omega\in [K]$, 
\[
    \veL_\omega=K\int_{(\omega-1) \delta}^{\omega\delta} (\veL(\tau,\cdot)+c\vect{I})\upd\tau,
    \quad \veL_{\omega,\dagger}=K\int_{(\omega-1)\delta}^{\omega\delta} (\veL^\dagger(\tau,\cdot)+c\vect{I})\upd\tau.
\]
Clearly $\veL_\omega+K\vect{I}\prec \veL_{\omega,\dagger}+K\vect{I}$ and both are nonnegative and only depend on space. 
We set $\veu^0=\veu_0$, $\vev^0=\veu_0^\dagger$ and consider, for any $\omega\in[K]$, the space periodic solutions of the elliptic systems
\begin{equation*}
(K+c) \veu^{\omega}-\vect{D}\Delta \veu^{\omega}=K\veu^{\omega-1}+\veL_{\omega}\veu^{\omega-1},
\end{equation*}
\begin{equation*}
(K+c) \vev^{\omega}-\vect{D}\Delta \vev^{\omega}=K\vev^{\omega-1}+\veL_{\omega,\dagger}\vev^{\omega-1}.
\end{equation*}

By an immediate recursion, and since the coefficients of the system satisfied by $\vev^{\omega}$ are spatially rearranged, 
$\vev^\omega$ itself is rearranged, for any $\omega\in [K]$. 
Indeed, this follows from the uniqueness of the solutions of the above systems and from the existence of radial solutions by 
using radial coordinates. 
Subsequently, by Proposition \ref{prop:Talenti_Dirichlet}, for any $\omega\in [K]$, $\veu^\omega\prec\vev^\omega$. 

Proceeding as in \cite{Alvino_Trombetti_Lions_1990}, we may pass to the limit $K\to \infty$ to conclude the proof. 
\end{proof}

\begin{rem}
Regarding the convergence of the sequence of elliptic problems to the parabolic one, although we do not detail it for the sake of conciseness, the easiest way to proceed here is to assume that $\veL$ is $\mathcal{C}^1$ in the time variable, which is always possible through a standard approximation argument. From this point of view, a Taylor expansion shows that, if $\veu$ is the solution of the parabolic equation,  if $\bar\veu^\omega:=K\int_{(\omega-1)\delta}^{\omega\delta} \veu$  and if  $\veu$ is the (piecewise) affine interpolation of $(\veu^\omega)_{\omega\in[K]}$, then we have $\|\bar \veu-\veu\|_{\mathcal{L}^2((0,T),\mathcal{W}^{1,2})}\leq C\delta$. Although the idea of such a discretization was used in a systematic way in \cite{Alvino_Trombetti_Lions_1990}, it was introduced in \cite{Vazquez_1982} where the discretization procedure was justified by the use of the (more abstract) Crandall--Liggett theorem. 
\end{rem}

\begin{prop}\label{prop:optimization_spatial_distribution}
Assume $\dcbP=\partial_t-\vect{D}\Delta$ for some diagonal matrix $\vect{D}$ with constant, positive diagonal entries.

Then
\[
\lambda_{1,\upp}(\cbQ)\geq \lambda_{1,\upp}(\dcbP-\veL^\dagger)
\]
where $\veL^\dagger$ is the entry-wise periodic rearrangement of $\veL$.
\end{prop}

\begin{proof}
The proof relies on Proposition \ref{prop:Talenti_Dirichlet_Parabolic}.

We proceed as in \cite{Nadin_2007} and introduce, for some $c>0$ so large that $\veL+c\vect{I}\geq\vez$, the Poincaré mapping 
\begin{equation*}
\mathcal{G}_{\veL+c\vect{I}}: 
\begin{cases} 
\mathcal{L}^\infty_\upp(\R\times\R,\R^N) & \to \mathcal{L}^\infty_\upp(\R\times\R,\R^N) \\
\veu_0 & \mapsto \veu(\veu_0,\veL+c\vect{I};T,\cdot)
\end{cases}
\end{equation*}
where $(t,x)\mapsto\veu(\veu_0,\veL+c\vect{I};t,x)$ is the solution of $\dcbP\veu+c\veu=(\veL+c\vect{I})\veu$ with initial condition $\veu_0$.

We define $r(\veL+c\vect{I})$ as the principal eigenvalue of the operator $\mathcal{G}_{\veL+c\vect{I}}$. 
As is classical, this eigenvalue can be obtained as 
\[
r(\veL+c\vect{I})=\lim_{k\to \infty}\left\| \mathcal{G}_{\veL+c\vect{I}}^k\right\|^{\frac{1}{k}},
\]
where the notation $\|\ \|$ stands for the norm on the vector space of linear mappings from $\mathcal{L}^\infty_\upp(\R\times\R,\R^N)$ into
itself (for the $\mathcal L^\infty-\mathcal L^\infty$ norm), and the two quantities $r(\veL+c\vect{I})$ and 
$\lambda_{1,\upp}(\dcbP-\veL)$ are immediately related through $r(\veL+c\vect{I})=-\frac{1}{T}\lambda_{1,\upp}(\dcbP-\veL)$.
As a consequence, in order to obtain the required comparison result, it suffices to establish that, for any 
$\veu_0\in\mathcal{L}^\infty_\upp(\R\times\R,\R^N)$, 
\begin{equation*}
\| \mathcal{G}_{\veL+c\vect{I}} \veu_0\|_{\mathcal{L}^\infty_\upp(\R\times\R,\R^N)}
\leq \| \mathcal{G}_{\veL^\dagger+c\vect{I}}\veu_0^\dagger\|_{\mathcal{L}^\infty_\upp(\R\times\R,\R^N)}.
\end{equation*}
Yet this is a direct consequence of Proposition \ref{prop:Talenti_Dirichlet_Parabolic}.
\end{proof}

\begin{rem}
The previous result established the optimality of periodically rearranged entries of $\veL$ when rearrangements are all 
centered at the same arbitrary spatial position, $x=0$ in our construction. The spatially periodic entries of $\veL$ must be 
``in phase''. If rearrangements of different entries of $\veL$ are centered at different positions, the optimality fails, as 
the following counter-example shows.

We consider, in the one-dimensional case $n=1$ with $L_1=2$, a space periodic function $\chi$ whose restriction to $[-1,1]$ 
is the indicator function of $(-y,y)$, $0<y<1$, a real number $\eta$, and the matrix 
\[
    \veL:x\mapsto\begin{pmatrix}-1+\chi(x) & 1+\chi(x-\eta) \\ 1+\chi(x-\eta) & -1+\chi(x)\end{pmatrix}.
\]
The vector $\veo$ is a Perron--Frobenius eigenvector of $\veL(x)$ with Perron--Frobenius eigenvalue $\mu_{\eta}(x)=\chi(x)+\chi(x-\eta)$. 
Let $\veu$ be a periodic principal eigenfunction of $\cbQ=\partial_t-\Delta-\veL$.
The function $u_{\eta}=\veo^\upT\veu$ is positive, time homogeneous, space periodic and solves
\begin{equation*}
-\Delta u_{\eta}=\mu_{\eta}u_{\eta}+\lambda_{1,\upp}(\cbQ)u_{\eta}.
\end{equation*}
Therefore $\lambda_{1,\upp}(\cbQ)=\lambda_{1,\upp}(-\Delta-\mu_{\eta})$, where
the last operator is a scalar space periodic elliptic operator.
Note that $\chi$ and $\chi(\cdot-\eta)$ are invariant by periodic rearrangement centered at $x=0$ and $x=\eta$ respectively.
Nonetheless, $\veL$ is not optimal as soon as $\eta\neq 0$. Indeed,
\[
\mu_{\eta}(x)=
\begin{cases}
0 & \text{if }x\in (-1,-y)\cup (y+\eta,1),\\ 
1 & \text{if }x\in (-y,-y+\eta)\cup(y,y+\eta),\\ 
2 & \text{if }x\in (-y+\eta,y).
\end{cases}
\]
Hence all $(\mu_{\eta})_{\eta\in\R}$ are piecewise-constant, space periodic functions of total mass equal to $4y>0$. It is well-known \cite{Nadin_2007} that among these the one that minimizes $\lambda_{1,\upp}(-\Delta-\mu_{\eta})$ is 
the one corresponding to $\eta=0$. This is of course consistent with our optimization result.
\end{rem}

\section{Acknowledgements}
The authors are grateful to the anonymous reviewer for a careful reading of the manuscript and insightful comments
that lead to substantial clarifications and improvements. They are also grateful to G. Nadin for several inspiring 
discussions and useful advice.

L. G. acknowledges support from the ANR via the project Indyana under grant agreement ANR-21-CE40-0008 and via the project Reach under grant agreement ANR-23-CE40-0023-01. I. M.-F. was partially funded by a PSL Young Researcher Starting Grant 2023.

\bibliographystyle{plain}
\bibliography{ref}

\end{document}